\documentclass[a4paper, French]{article}
\usepackage[all]{xy}
\usepackage{mathtools}
\usepackage{pifont}
\usepackage[english,french]{babel}
\usepackage{amssymb}
\usepackage{tikz-cd}
\usepackage{indentfirst}

\usepackage{mathptmx}
\usepackage{amsthm}
\usepackage{fancyhdr}
\usepackage{color}
\usepackage{enumitem}   
\usepackage[utf8]{inputenc}
\usepackage{framed}

\usepackage{newtxtext,newtxmath}
\usepackage[colorlinks,citecolor=red,linkcolor=blue]{hyperref}
\usepackage{epigraph}
\usepackage{geometry}
\usepackage{tcolorbox}
\usepackage{thmtools}
\numberwithin{equation}{subsection}
\usepackage{cite}
\usepackage{amsmath}
\usepackage[colorinlistoftodos]{todonotes}
\usepackage{graphicx}
\usepackage{extpfeil}
\usepackage[all]{xy}
\usepackage{comment}
\usepackage[title]{appendix}
\usepackage{cleveref}
\addto\captionsfrench{}

\newtheorem{theorem}{\color{blue}Théorème}[subsection]

\theoremstyle{definition}
\newtheorem{definition}[theorem]{\color{blue}Définition}

\theoremstyle{definition}  
\newtheorem{remark}[theorem]{\color{blue}Remarque}

\theoremstyle{plain} 
\newtheorem{lemma}[theorem]{\color{blue}Lemme}

\newtheorem{proposition}[theorem]{\color{blue}Proposition}
\newtheorem{corollary}[theorem]{\color{blue}Corollaire}

\newcommand{\z}{\mathbb{Z}}

\title{Revêtements du demi-plan de Drinfeld et Langlands $p$-adique catégorique }
\author{Yang Pei}
\begin{document}

\date{\today}
	\maketitle
 
	\selectlanguage{french} 
\begin{abstract} Nous généralisons à tous les étages de la tour de revêtements du demi-plan de Drinfeld une suite exacte établie par Lue Pan pour le premier revêtement. Par ailleurs, nous introduisons deux foncteurs, inspirés par la catégorification de la correspondance de Langlands locale $p$-adique, en versions de Banach et localement analytique respectivement. Nous calculons ensuite les faisceaux associés aux représentations intervenant dans notre suite. Comme application, nous montrons que tous les quotients propres du complété unitaire universel d'une représentation supercuspidale sont de longueur finie.
\end{abstract}
\selectlanguage{english} 
\begin{abstract} 
We generalize to all levels of the tower of coverings of the Drinfeld upper plane an exact sequence established by Lue Pan for the first covering. Furthermore, we introduce two functors, inspired by the categorification of the $p$-adic local Langlands correspondence, in Banach and locally analytic versions respectively. We then compute the sheaves associated with the representations appearing in our sequence. As an application, we show that all proper quotients of the universal unitary completion of a supercuspidal representation have finite length.
\end{abstract}
\thispagestyle{fancy}
	\tableofcontents
	\section{Introduction}
\subsection{Correspondance de Langlands locale $p$-adique}
En 2010, Colmez a construit \cite{colmez2010representations}, dans une série de travaux fondés sur la théorie des $(\varphi,\Gamma)$-modules de Fontaine, et en combinaison avec les travaux de Breuil, Kisin et d'autres, un foncteur, parfois appelé \emph{le foncteur de Montréal}, de la catégorie des représentations continues de $\mathrm{Gal}_{\mathbb{Q}_p}$ vers celle des représentations unitaires admissibles de $\mathrm{GL}_2(\mathbb{Q}_p)$, ainsi qu'un foncteur dans l'autre sens, qui réalisent la correspondance de Langlands locale $p$-adique. Voir aussi \cite{berger2011la}, \cite{breuil2012correspondance} et \cite{egh2023an} pour des présentations expositives. Cette correspondance a été complétée en 2016 par Colmez, Dospinescu et Pa\v{s}k\={u}nas \cite{cdp2014the}, qui ont démontré que les deux foncteurs induisent une équivalence entre la catégorie des représentations unitaires admissibles absolument irréductibles et non ordinaires de $\mathrm{GL}_2(\mathbb{Q}_p)$, et celle des représentations absolument irréductibles continues de dimension $2$ de $\mathrm{Gal}_{\mathbb{Q}_p}$. 

La correspondance de Langlands locale $p$-adique, raffinant la correspondance classique au sens de \cite[Theorem 1.3]{cdp2014the}, s'inscrit également dans un cadre plus large grâce à ses compatibilités essentielles. D'une part, elle est compatible avec la correspondance de Langlands locale mod $p$. D'autre part, les représentations de Banach obtenues par cette correspondance apparaissent naturellement dans la cohomologie complétée des tours de courbes modulaires, établissant ainsi un lien profond avec la correspondance globale.

Les constructions des foncteurs de Colmez sont purement algébriques et ne fonctionnent que pour $\mathrm{GL}_2(\mathbb{Q}_p)$. Même pour $\mathrm{GL}_2(K)$ où $2\leq [K: \mathbb{Q}_p]<\infty$, le problème devient beaucoup plus compliqué. En fait, pour $\mathrm{GL}_2(\mathbb{Q}_p)$, nous avons les résultats de classification des représentations lisses admissibles irréducibles modulo $p$ de $\mathrm{GL}_2(\mathbb{Q}_p)$ d'après Barthel-livné \cite{blirreducible1994} et Breuil \cite{breuilsur2003}, et il n'y a qu'un nombre fini de classes d'isomorphismes de représentations supersingulières. Cependant, pour $\mathrm{GL}_2(K)$, une classification similaire fait défaut, et les travaux de Breuil et Pa\v{s}k\={u}nas \cite{bp2012towards} montrent que la situation est bien plus complexe. Par exemple, ils ont démontré, en utilisant la théorie des diagrammes, qu'il existe un nombre infini de classes d'isomorphismes de représentations supersingulières dans ce cas. Cela révèle que les méthodes efficaces pour $\mathrm{GL}_2(\mathbb{Q}_p)$ ne se généralisent pas directement à des extensions finies non triviales de $\mathbb{Q}_p$. La structure des représentations de $\mathrm{GL}_2(K)$ requiert donc de nouvelles techniques pour surmonter l'absence de finitude des représentations supersingulières.

La correspondance de Langlands $p$-adique a de nombreuses applications. Par exemple, Kisin \cite{kisin2009the} et Emerton \cite{emerton2011local} ont prouvé la conjecture de Fontaine-Mazur dans la plupart des cas en utilisant les résultats de Colmez. Il est à noter que Lue Pan \cite{pan2022on, pan2022locally} a récemment obtenu les mêmes résultats sans recourir à la correspondance de Langlands $p$-adique. De plus, supposons que $\Pi$ est une représentation de Banach unitaire, résiduellement de longueur finie, admettant un caractère central de $\mathrm{GL}_2(\mathbb{Q}_p)$, alors Colmez et Dospinescu \cite{cd2014completes} ont montré, en utilisant encore la correspondance de Langlands $p$-adique, que $\Pi$ est le complété unitaire universel de $\Pi^{\mathrm{an}}$, qui est le sous-espace des vecteurs localement analytiques de $\Pi$.

Il existe maintenant au moins trois généralisations du foncteur de Colmez. Deux de ces généralisations \cite{caraiani2016patching, breuil2024conjectures, breuil2025multivariable} utilisent le patching à la Taylor-Wiles-Kisin comme outil global. En revanche, Colmez, Dospinescu et Nizio\l{} \cite{cdn2020cohomologie} ont construit un foncteur de la catégorie des représentations de Banach de $\mathrm{GL}_2(K)$ vers la catégorie des représentations galoisiennes en utilisant la cohomologie du demi-plan de Drinfeld, une approche qui est entièrement locale.
\subsection{Généralisation de la suite exacte de Lue Pan}
Soient $L$ une extension finie de $\mathbb{Q}_p$ et $M$ un $L$-$(\varphi, N, \mathcal{G}_{\mathbb{Q}_p})$-module supercuspidal, de pente $1/2$ et rang $2$ sur $L \otimes \mathbb{Q}_p^{\mathrm{nr}}$. On note $M_{\mathrm{dR}}:= (\overline{\mathbb{Q}}_p\otimes_{\mathbb{Q}_p^{\mathrm{nr}}}M)^{\mathcal{G}_{\mathbb{Q}_p}}$, qui un $L$-module de rang $2$, où $\mathcal{G}_{\mathbb{Q}_p}$ est le groupe de Galois absolu de $\mathbb{Q}_p$, alors toutes les représentations $V$ potentiellement semi-stables de $\mathcal{G}_{\mathbb{Q}_p}$ de dimension $2$ à poids de Hodge-Tate $0, 1$ telles que $D_{\mathrm{pst}}(V) = M$ sont classifiées par l'espace $\mathbf{P}^1=\mathbf{P}(M_{\mathrm{dR}})$. Soit $\mathscr{L}\in\mathbf{P}^1$, alors nous avons une représentation galoisienne $V_{M, \mathscr{L}}$ et une représentation de Banach $\Pi_{M, \mathscr{L}} := \mathbf{\Pi}(V_{M, \mathscr{L}})$, où $\mathbf{\Pi}$ est le foncteur de Colmez. Notons $\widehat{\mathrm{LL}(M)}$ le complété unitaire universel de la représentation supercuspidale $\mathrm{LL}(M)$ associée à $M$, et $\Omega^1[M]^{\mathrm{b}, *}$ le dual de l'espace des vecteurs $\mathrm{GL}_2(\mathbb{Q}_p)$-bornés dans la $\mathrm{JL}(M)$-partie de l'espace des formes différentielles $\Omega^1$ sur le $n$-ième revêtement $\Sigma_n$ du demi-plan de Drinfeld. Pour des définitions plus concrètes, voir la section \ref{section2.1}. Dans \cite{pan2017first}, Pan a prouvé, pour toute $\mathscr{L}$, l'exactitude de la suite suivante pour $\Sigma_1$ 
$$0\to\widehat{\mathrm{LL}(M)}\to \Omega^1[M]^{\mathrm{b}, *}\to\Pi_{M, \mathscr{L}}\to0.$$

De plus, Pan a montré plusieurs résultats pour le premier revêtement $\Sigma_1$ en construisant un modèle explicite, mais il est difficile d'adapter sa méthode aux niveaux supérieurs. En utilisant les travaux de Dospinescu et Le Bras \cite{dl2017revetements}, on peut montrer que la plupart des résultats de Pan restent vrais pour tous les revêtements. Cependant, comme l'indique la \cite[Remarque 1.13]{pan2017first}, la suite exacte mentionnée ci-dessus ne semble pas en découler directement. Dans la première partie de cette thèse, nous déduisons cette suite exacte courte pour tous les revêtements en utilisant les résultats de \cite{dl2017revetements} au lieu de construire un modèle explicite.

Notre premier résultat principal est le théorème suivant, ce qui fournit une construction géométrique de la représentation $\Pi_{M, \mathscr{L}}$.
\begin{theorem}[Théorème \ref{3.7.2}]\label{0.0.1} Pour tout $\Sigma_n$, le choix d'une $\mathscr{L}\in\mathbf{P}^1$ induit une suite exacte non scindée de représentations de Banach de $\mathrm{GL}_2(\mathbb{Q}_p)$ $$0\to\widehat{\mathrm{LL}(M)}\to \Omega^1[M]^{\mathrm{b}, *}\to\Pi_{M, \mathscr{L}}\to0.$$
\end{theorem}
\subsection{Foncteurs des représentations vers les faisceaux}
Inspiré par la catégorification de la correspondance de Langlands locale $p$-adique \cite{egh2023an}, nous construisons deux foncteurs des représentations vers les faisceaux en versions de Banach et localement analytique respectivement. 

Concernant les représentations localement analytiques, nous construisons une suite de foncteurs $\mathbf{m}^i$ de la catégorie des modules topologiques sur l'algèbre des distributions localement analytiques de $\mathrm{GL}_2(\mathbb{Q}_p)$ vers la catégorie des faisceaux sur $\mathbf{P}^1$. Nous calculons ensuite certains entrelacements entre les représentations localement analytiques provenant de $\Sigma_n$, ce qui nous permet de déterminer les faisceaux associés à ces représentations. Le résultat suivant résume nos calculs
\begin{theorem}[Proposition \ref{4.1.4}, Proposition \ref{4.1.5}, Proposition \ref{4.1.6} et Proposition \ref{4.1.7}] On a 
	
	(i) $\mathbf{m}^0(\mathrm{LL}(M))=\mathscr{O}_{\mathbf{P}^1}(-1)$.
	
	(ii) $\mathbf{m}^0(\Omega^1[M]^*)=\mathscr{O}_{\mathbf{P}^1}$.
	
	(iii) $\mathbf{m}^0(\mathscr{O}[M]^*)=0$.
	
	(iv) $\mathbf{m}^0(\Pi_ {M, \mathscr{L}, j}^{\mathrm{an}})$ est le faisceau gratte-ciel $\mathscr{O}_{\mathbf{P}^1, \mathscr{L}}/\mathfrak{m}_{\mathscr{L}}^j$ concentré en le point correspondant à $\mathscr{L}$.
\end{theorem}
Concernant les représentations de Banach, nous construisons une suite de foncteurs $\hat{\mathbf{m}}^i$ de la catégorie des modules topologiques sur l'algèbre d'Iwasawa augmentée de $\mathrm{GL}_2(\mathbb{Q}_p)$ vers la catégorie des faisceaux sur $\mathbf{P}^1$. Ensuite, pour les représentations de Banach, nous calculons quelques entrelacements entre les représentations de Banach qui apparaissent dans le Théorème \ref{0.0.1}. Parmi ces résultats, le résultat suivant joue un rôle central
\begin{theorem}[Corollaire \ref{3.7.6}, Proposition \ref{3.8.1}]\label{0.0.2} 
	
	(i) $\mathrm{End}_G(\widehat{\mathrm{LL}(M)})=L$.
	
	(ii) $\mathrm{Ext}^1_{G, \psi}(\widehat{\mathrm{LL}(M)}, \widehat{\mathrm{LL}(M)})=0$.
\end{theorem}

La démonstration de l'assertion (i) repose sur le fait que la partie lisse de $\widehat{\mathrm{LL}(M)}$ est $\mathrm{LL}(M)$ \cite[Proposition 3.1]{cdn2023correspondance}, et nous renforçons ce résultat sous la forme
\begin{theorem}[Proposition \ref{3.2.3}]\label{0.3.4} Le morphisme naturel de $\mathrm{LL}(M)$ vers la partie localement analytique $\widehat{\mathrm{LL}(M)}^{\mathrm{an}}$ de $\widehat{\mathrm{LL}(M)}$ est un isomorphisme.
\end{theorem}
\begin{remark} (i) Le Théorème \ref{0.3.4} est un peu surprenant, car $\Pi_{M, \mathscr{L}, j}$ est un quotient de $\widehat{\mathrm{LL}(M)}$ \cite[Lemme 5.3]{cdn2023correspondance} et $\Pi_{M, \mathscr{L}, j}^{\mathrm{an}}$ est beaucoup plus grosse que $\mathrm{LL}(M)$. 
	
	(ii) Soit $\pi$ une représentation localement analytique admissible de $\mathrm{GL}_2(\mathbb{Q}_p)$ admettant un complété unitaire universel $\hat{\pi}$, la sous-représentation $\hat{\pi}^{\mathrm{an}}$ n'est pas forcément égale à $\pi$. Voir la \cite[Remarque 0.3 (ii)]{cd2014completes} pour plus de détails.
\end{remark}
Le Théorème \ref{0.0.2} nous permet ensuite d'identifier les faisceaux associés aux représentations considérées, comme énoncé dans le théorème suivant
\begin{theorem}[Proposition \ref{4.2.4}, Proposition \ref{4.2.5} et Corollaire \ref{4.2.6}]\label{0.0.3} On a 
	
	(i) $\hat{\mathbf{m}}^0(\widehat{\mathrm{LL}(M)})=\mathscr{O}_{\mathbf{P}^1}(-1)$.
	
	(ii) $\hat{\mathbf{m}}^0(\Omega^1[M]^{b, *})=\mathscr{O}_{\mathbf{P}^1}$.
	
	(iii) $\hat{\mathbf{m}}^0(\Pi_{M, \mathscr{L}, j})$ est isomorphe au faisceau gratte-ciel $\mathscr{O}_{\mathbf{P}^1, \mathscr{L}}/\mathfrak{m}_{\mathscr{L}}^j$ concentré en le point correspondant à $\mathscr{L}$.
\end{theorem}
\subsection{Quotients propres de $\widehat{\mathrm{LL}(M)}$}
Une méthode principale pour analyser les représentations admissibles de Banach de $\mathrm{GL}_2(\mathbb{Q}_p)$ repose sur l'utilisation de la théorie des $(\varphi, \Gamma)$-modules. Cependant, les représentations $\widehat{\mathrm{LL}(M)}$ et $\Omega^1[M]^{b, *}$ ne sont pas admissibles, ce qui signifie que le foncteur de Colmez n'est pas défini pour ces représentations. Néanmoins, nous pouvons utiliser notre foncteur en version de Banach. Comme application, on montre le théorème ci-dessous.
\begin{theorem}[Corollaire \ref{5.2.7}] Tous les quotients propres de $\widehat{\mathrm{LL}(M)}$ sont de longueur finie.
\end{theorem}
On note $\Pi_{M, \mathscr{L}, j}$ la représentation unique, à isomorphisme près, non scindée, de de Rham, de caractère central $\psi$, et de longueur $j$ dont les facteurs de Jordan-Hölder sont tous isomorphes à $\Pi_{M, \mathscr{L}}$ \cite[Section 4.2]{cdn2023correspondance}, alors on a le théorème suivant de classification des quotients de $\widehat{\mathrm{LL}(M)}$.
\begin{theorem}[Corollaire \ref{5.2.8}]\label{0.0.5} Soit $\Pi$ un quotient propre de $\widehat{\mathrm{LL}(M)}$, alors on a 
	$$\Pi\cong\oplus_{i\in I}\Pi_{M, \mathscr{L}_i, j_i}$$
	où $I$ est un ensemble d'indices fini.
\end{theorem}
Il suit directement du Théorème \ref{0.0.5} que tous les quotients propres de $\widehat{\mathrm{LL}(M)}$ sont admissibles.
\newpage
\section*{Notations}
Dans tout le texte, $L$ désigne une extension finie de $\mathbb{Q}_p$, avec l'anneau d'entiers $\mathscr{O}_L$ et le corps résiduel $\kappa_L$. On note $\pi_L$ une uniformisante de $L$. Notons $G$ le groupe de Lie $p$-adique $\mathrm{GL}_2(\mathbb{Q}_p)$ et $\breve{G}$ le groupe des unités de l'unique algèbre des quaternions non déployée de centre $\mathbb{Q}_p$. Pour tout $i \geq 1$, notons $G_i := 1 + p^i M_2(\mathbb{Z}_p)$ le sous-groupe compact de $G$. Notons $B$ le groupe de Borel de $G$, $K=\mathrm{GL}_2(\z_p)$ et $Z\cong\mathbb{Q}_p^*$ le centre de $G$. Notons $\mathcal{G}_{\mathbb{Q}_p}$ le groupe de Galois absolu de $\mathbb{Q}_p$ et $W_{\mathbb{Q}_p}$ son groupe de Weil. 

Soit $H$ un sous-groupe fermé de $G$, on note $\mathrm{ind}_H^G \pi$ l'induite à support compact d'une représentation $\pi$ de $H$. L'espace sous-jacent est constitué des fonctions $f: G\to\pi$ telles que $f(gh)=h^{-1}\cdot f(g)$ pour tout $g\in G, h\in H$, et que l'image de $\mathrm{supp}(f):=\{g\in G\mid f(g)\neq0\}$ dans $H\backslash G$ est compacte. L'action de $G$ est donnée par $(g\cdot f)(x):=f(xg)$ pour tout $g, x\in G$.

Définissons $\Lambda(G):= L[G] \otimes_{L[H]} \Lambda(H) = \mathrm{ind}_H^G L[[H]]$ l'algèbre d'Iwasawa augmentée, comme décrite dans \cite{kohlhaase2017smooth}, où $H$ est un sous-groupe ouvert compact de $G$ et $\Lambda(H):=L[[H]]$ est l'algèbre d'Iwasawa usuelle définie par Lazard \cite{lazard1965groupes}. Il convient de noter que la définition de $\Lambda(G)$ ne dépend pas du choix de $H$. Toutes les représentations de Banach disposent d'une action de $\Lambda(G)$ qui étend l'action de $L[G]$. 

Soit $H$ un sous-groupe ouvert de $G$. Nous notons $D(H)$ l'algèbre des distributions localement analytiques, comme définie dans \cite{st2003algebras}, qui est le dual fort de l'espace $C^{\mathrm{la}}(H, L)$ des fonctions localement analytiques sur $H$ à valeurs dans $L$. Si, de plus, $H$ est compact, alors l'algèbre de Fréchet-Stein $D(H)$ est munie d'une norme sous-multiplicative $||\cdot||_r$ pour chaque $\frac{1}{r} \leq p < 1$. Nous notons $D_r(H)$ la complétion de $D(H)$ par rapport à la norme $||\cdot||_r$. Par conséquent, nous avons $D(H)=\varprojlim_r D_r(H)$. Toutes les représentations localement analytiques de $G$ disposent d'une action de $D(G)$ par le morphisme d'intégration qui étend l'action de $L[G]$, comme indiqué dans le \cite[Théorème 2.2]{st2002locally}. 

Soit $\Pi$ une représentation de Banach de $G$. Notons $\Pi^{\mathrm{an}}$ l'ensemble des vecteurs localement analytiques dans $\Pi$. Nous munissons cet espace de la topologie induite par le morphisme $\Pi^{\mathrm{an}} \to C^{\mathrm{an}}(G, \Pi)$, qui est plus forte que la topologie héritée de $\Pi$. Ainsi, $\Pi^{\mathrm{an}}$ devient une représentation localement analytique de $G$. La représentation $\Pi^{\mathrm{an}}$ dispose d'une sous-représentation fermée $\Pi^{\mathrm{lisse}}$, qui est constituée des vecteurs lisses de $\Pi$. 

Enfin, soit $\pi$ une $L$-représentation de $G$, nous notons $\hat{\pi}$ le complété unitaire universel de $\pi$ au sens d'Emerton \cite[Definition 1.1]{emerton2005p-adic}, s'il existe. La définition et les propriétés associées seront rappelées dans la section \ref{section1.3}.
	\section*{Remerciements} Cet article est issu de ma thèse que j'ai réalisée sous la direction de Pierre Colmez et Christophe Cornut. Je remercie chaleureusement Pierre Colmez pour m'avoir proposé ce sujet et pour avoir généreusement partagé avec moi ses idées et ses connaissances. Son soutien constant m'a permis de surmonter de nombreuses difficultés et d'avancer dans mes réflexions. Je suis également très reconnaissant envers Christophe Cornut pour son aide précieuse tout au long de ce travail, ainsi que pour le temps et l'énergie qu'il m’a consacrés.

Je souhaiterais aussi remercier chaleureusement Yiwen Ding, Yongquan Hu, Stefano Morra, Vytautas Pa\v sk\=unas, Benchao Su, Arnaud Vanhaecke, Marie-France Vignéras, et Zhixiang Wu pour les discussions éclairantes, qui ont grandement contribué à l'enrichissement de ce travail. En outre, j'aimerais remercier Gabriel Dospinescu et Ramla Abdellatif, examinateur et examinatrice de ma thèse, pour le temps et l'attention qu'ils ont accordés à mon travail.

Une partie de ce travail a été réalisée lors d'un séjour, à l'invitation de Yongquan Hu, au Morningside Center of Mathematics, Chinese Academy of Sciences. Je remercie vivement l'institut pour son accueil chaleureux et son soutien.
\section{Préliminaires} 
\subsection{Théorie des blocs d'après Pa\v{s}k\={u}nas}
\subsubsection{Représentations modulo $p$}
Fixons un caractère continu $\psi: \mathbb{Q}_p^\times\to\mathscr{O}_L^\times$. Soit $\mathrm{Mod}_{G, \psi}^{\mathrm{sm}}(\mathscr{O}_L)$ la catégorie des $\mathscr{O}_L$-modules de torsion munis d'une action continue de $G$ pour la topologie discrète sur les modules à caractère central $\psi$. Soit $\mathrm{Mod}_{G, \psi}^{\mathrm{l.adm}}(\mathscr{O}_L)$ la sous-catégorie pleine formée des représentations localement admissibles, alors $\mathrm{Mod}_{G, \psi}^{\mathrm{l.adm}}(\mathscr{O}_L)$ est une catégorie abélienne \cite[Proposition 2.2.18]{emerton2010ordinary}. Notons $\mathrm{Mod}_{G, \psi}^{\mathrm{l.fin}}(\mathscr{O}_L)$ la sous-catégorie pleine de $\mathrm{Mod}_{G, \psi}^{\mathrm{sm}}(\mathscr{O}_L)$ formée des représentations localement de longueur finie. Il suit du \cite[Theorem 2.3.8]{emerton2010ordinary} qu'une représentation lisse de $G$ à caractère central est localement admissible si et seulement si elle est localement de longueur finie, on a donc $\mathrm{Mod}_{G, \psi}^{\mathrm{l.adm}}(\mathscr{O}_L)=\mathrm{Mod}_{G, \psi}^{\mathrm{l.fin}}(\mathscr{O}_L)$.

Soit $\mathrm{Irr}_G^{\mathrm{adm}}$ l'ensemble des représentations irréductibles dans $\mathrm{Mod}_{G, \psi}^{\mathrm{l.adm}}(\mathscr{O}_L)$, on définit une relation d'équivalence $\sim$ sur $\mathrm{Mod}_G^{\mathrm{l.adm}}(\mathscr{O}_L)$ par: $\pi\sim\tau$ s'il y a une suite de représentations admissibles irréductibles $\pi=\pi_1, \pi_2,..., \pi_n=\tau$ telles que pour chaque $i$, on a $\pi_i\cong\pi_{i+1}$, ou $\mathrm{Ext}^1(\pi_i, \pi_{i+1})\neq0$ ou $\mathrm{Ext}^1(\pi_{i+1}, \pi_i)\neq0$. Comme une représentation admissible lisse de type fini admettant un caractère central de $G$ est de longueur finie \cite[2.3.8]{emerton2010ordinary}, la catégorie $\mathrm{Mod}_{G, \psi}^{\mathrm{l.adm}}(\mathscr{O}_L)$ est localement finie. Il résulte de \cite{gabriel1962des} que toute catégorie localement finie se décompose en blocs. 
\begin{theorem}[{\cite{paskunas2013the}}] On a $$\mathrm{Mod}_{G, \psi}^{\mathrm{l.adm}}(\mathscr{O}_L)\cong\prod_{\mathfrak{B}\in\mathrm{Irr}_G^{\mathrm{adm}}/\sim}\mathrm{Mod}_{G, \psi}^{\mathrm{l.adm}}(\mathscr{O}_L)[\mathfrak{B}],$$
	où $\mathrm{Mod}_{G, \psi}^{\mathrm{l.adm}}(\mathscr{O}_L)[\mathfrak{B}]$ est la sous-catégorie pleine de $\mathrm{Mod}_{G, \psi}^{\mathrm{l.adm}}(\mathscr{O}_L)$ constituée des représentations avec tous sous-quotients irréductibles dans $\mathfrak{B}$.
\end{theorem} 
Si $\pi\in\mathrm{Irr}_G^{\mathrm{adm}}$, alors $\pi$ est isomorphe à une somme directe finie de représentations absolument irréductibles de $G$. Pour les blocs qui contiennent une représentation absolument irréductible, on a le résultat suivant
\begin{theorem}[{\cite{paskunas2014blocks}}] Soit $\mathfrak{B}$ un bloc qui contient une représentation absolument irréductible, alors $\mathfrak{B}$ est de l'un des six types suivants
	\begin{enumerate}[label=(\roman*)]
		
		\item $\mathfrak{B}=\{\pi\}$ avec $\pi$ supersingulier,
		
		\item $\mathfrak{B}=\{(\mathrm{Ind}_B^G\chi_1\otimes\chi_2\omega^{-1})_{\mathrm{lisse}}, (\mathrm{Ind}_B^G\chi_2\otimes\chi_1\omega^{-1})_{\mathrm{lisse}}\}$ avec $\chi_2\chi_1^{-1}\neq\omega^{\pm1}, 1$,
		
		\item $p>2$ et $\mathfrak{B}=\{(\mathrm{Ind}_B^G\chi\otimes\chi\omega^{-1})_{\mathrm{lisse}}\}$,
		
		\item $p=2$ et $\mathfrak{B}=\{1, \mathrm{Sp}\}\otimes\chi\circ\det$,
		
		\item $p\geq5$ et $\mathfrak{B}=\{1, \mathrm{Sp}, (\mathrm{Ind}_B^G\chi\otimes\chi\omega^{-1})_{\mathrm{lisse}}\}\otimes\chi\circ\det$,
		
		\item $p=3$ et $\mathfrak{B}=\{1, \mathrm{Sp}, \omega\circ\det, \mathrm{Sp}\otimes \omega\circ\det\}\otimes\chi\circ\det$,
	\end{enumerate}
	où $\chi, \chi_1, \chi_2: \mathbb{Q}_p^\times\to \kappa_L$ sont des caractères lisses, $\omega(x)=x|x|\mod\pi_L$ et $\mathrm{Sp}$ désigne la représentation de Steinberg.
\end{theorem}

On peut définir une bijection $\mathfrak{B}\mapsto\rho_\mathfrak{B}$ entre les blocs absolus et $\kappa_L$-représentations $\rho_\mathfrak{B}$ de $\mathcal{G}_{\mathbb{Q}_p}$ semi-simples de dimension $2$. La recette est
\begin{enumerate}[label=(\roman*)]
	\item Si $\mathfrak{B}=\{\pi\}$ avec $\pi$ supersingulier, alors $\rho_\mathfrak{B}$ est la représentation absolument irréductible de dimension $2$ telle que $\pi$ s'envoie vers $\rho_\mathfrak{B}$ sous l'action du foncteur de Colmez.
	
	\item Si $\mathfrak{B}=\{(\mathrm{Ind}_B^G\chi_1\otimes\chi_2\omega^{-1})_{\mathrm{lisse}}, (\mathrm{Ind}_B^G\chi_2\otimes\chi_1\omega^{-1})_{\mathrm{lisse}}\}$, alors $\rho_\mathfrak{B}=\chi_1\oplus\chi_2$.
	
	\item Si $\mathfrak{B}=\{(\mathrm{Ind}_B^G\chi\otimes\chi\omega^{-1})_{\mathrm{lisse}}\}$ ou $\mathfrak{B}=\{1, \mathrm{Sp}\}\otimes\chi\circ\det$, alors $\rho_\mathfrak{B}=\chi\oplus\chi$.
	
	\item Si $\mathfrak{B}=\{1, \mathrm{Sp}, (\mathrm{Ind}_B^G\chi\otimes\chi\omega^{-1})_{\mathrm{lisse}}\}\otimes\chi\circ\det$ ou $\mathfrak{B}=\{1, \mathrm{Sp}, \omega\circ\det, \mathrm{Sp}\otimes \omega\circ\det\}\otimes\chi\circ\det$, alors $\rho_\mathfrak{B}=\chi\oplus\chi\omega$.
\end{enumerate}
\subsubsection{Représentations de Banach}
Soit $\mathfrak{B}$ un bloc absolu. Notons $R^{\mathrm{ps}, \psi\omega}_{\mathrm{tr}\bar{\rho}_{\mathfrak{B}}}$ l'anneau de pseudo-déformations universelles du pseudo-caractère $\mathrm{tr}\bar{\rho}_{\mathfrak{B}}$, de dimension $2$ et de déterminant $\psi\omega$. 

Notons $\mathrm{Ban}_{G, \psi}^{\mathrm{adm}}(L)$ la catégorie des $L$-représentations de Banach admissibles unitaires de $G$ à caractère central $\psi$. Il suit de \cite{st2002banach} et de \cite[6.2.16]{emerton2017locally} que la catégorie $\mathrm{Ban}_{G, \psi}^{\mathrm{adm}}(L)$ est abélienne.
\begin{theorem}[{\cite[Proposition 5.36]{paskunas2013the}}] On a une décomposition de la catégorie $\mathrm{Ban}_{G, \psi}^{\mathrm{adm}}(L)$ ci-dessous
	$$\mathrm{Ban}_{G, \psi}^{\mathrm{adm}}(L)\cong\underset{{\mathfrak{B}\in\mathrm{Irr}_G^{\mathrm{adm}}/\sim}}{\bigoplus}\mathrm{Ban}_{G, \psi}^{\mathrm{adm}}(L)_\mathfrak{B}$$
	où les objets de $\mathrm{Ban}_{G, \psi}^{\mathrm{adm}}(L)_\mathfrak{B}$ sont les représentations $\Pi$ dans $\mathrm{Ban}_{G, \psi}^{\mathrm{adm}}(L)$ telles que pour tout réseau $G$-invariant ouvert borné $\Theta$ de $\Pi$, les sous-quotients irréductibles de $\Theta\otimes_{\mathscr{O}_L}\kappa_L$ sont dans $\mathfrak{B}$.
\end{theorem}
On note $\mathrm{Ban}_{G, \psi}^{\mathrm{adm}}(L)_\mathfrak{B}^{\mathrm{fl}}$ la sous-catégorie pleine de $\mathrm{Ban}_{G, \psi}^{\mathrm{adm}}(L)_\mathfrak{B}$ formée des objets de longueur finie. Soit $\Pi\in\mathrm{Ban}_{G, \psi}^{\mathrm{adm}}(L)_\mathfrak{B}^{\mathrm{fl}}$. Choisissons un réseau borné ouvert $\Pi^+$ dans $\Pi$, alors pour chaque $n\geq1$, $\Pi/\pi_L^n$ est un objet dans $\mathrm{Mod}_{G, \psi}^{\mathrm{l.adm}}(\mathscr{O}_L)[\mathfrak{B}]$. Comme $R^{\mathrm{ps}, \psi\epsilon}_{\mathrm{tr}\bar{\rho}_{\mathfrak{B}}}$ est isomorphe naturellement au centre de la catégorie $\mathrm{Mod}_{G, \psi}^{\mathrm{l.adm}}(\mathscr{O}_L)[\mathfrak{B}]$ \cite[Theorem 1.5]{paskunas2013the}, l'anneau $R^{\mathrm{ps}, \psi\epsilon}_{\mathrm{tr}\bar{\rho}_{\mathfrak{B}}}$ agit sur $\Pi/\pi_L^n$. En passant à la limite et en inversant $p$, on obtient une action de $R^{\mathrm{ps}, \psi\epsilon}_{\mathrm{tr}\bar{\rho}_{\mathfrak{B}}}[\tfrac{1}{p}]$ sur $\Pi$. Soit $\mathbf{m}$ un idéal maximal de $R^{\mathrm{ps}, \psi\epsilon}_{\mathrm{tr}\bar{\rho}_{\mathfrak{B}}}[\tfrac{1}{p}]$, on note $\mathrm{Ban}_{G, \psi}^{\mathrm{adm}}(L)_{\mathfrak{B}, \mathbf{m}}^{\mathrm{fl}}$ la sous-catégorie de $\mathrm{Ban}_{G, \psi}^{\mathrm{adm}}(L)_\mathfrak{B}^{\mathrm{fl}}$ constituée des représentations tuées par une puissance de $\mathbf{m}$.  
\begin{theorem}[{\cite[Theorem 4.36]{paskunas2013the}}] On a une décomposition de catégorie 
	$$\mathrm{Ban}_{G, \psi}^{\mathrm{adm}}(L)_\mathfrak{B}^{\mathrm{fl}}\cong\underset{\mathbf{m}\in\mathrm{Spm}R^{\mathrm{ps}, \psi\omega}_{\mathrm{tr}\bar{\rho}_{\mathfrak{B}}}[\tfrac{1}{p}]}{\bigoplus}\mathrm{Ban}_{G, \psi}^{\mathrm{adm}}(L)_{\mathfrak{B}, \mathbf{m}}^{\mathrm{fl}}.$$
\end{theorem}
Soit $\Pi$ une représentation de $L$-Banach non ordinaire unitaire admissible absolument irréductible de $G$ de caractère central $\psi$, on note $\mathrm{Ban}_{G, \psi}^{\mathrm{adm}}(L)_\Pi^{\mathrm{fl}}$ la sous-catégorie pleine de $\mathrm{Ban}_{G, \psi}^{\mathrm{adm}}(L)$, constituée des représentations de longueur finie, dont tous les sous-quotients irréductibles sont isomorphes à $\Pi$. On note $\mathfrak{B}$ le bloc correspondant à la réduction modulo $p$ de $\mathbf{V}(\Pi)$. D'après la preuve du \cite[Theorem 11.7]{paskunas2013the}, on a $\mathrm{Ban}_{G, \psi}^{\mathrm{adm}}(L)_\Pi^{\mathrm{fl}}=\mathrm{Ban}_{G, \psi}^{\mathrm{adm}}(L)^{\mathrm{fl}}_{\mathfrak{B}, \mathbf{m}}$ pour un idéal maximal $\mathbf{m}\in \mathrm{Spm}R^{\mathrm{ps}, \psi\omega}_{\mathrm{tr}\bar{\rho}_{\mathfrak{B}}}[\tfrac{1}{p}]$ déterminé par $\Pi$.
\begin{theorem}[{\cite[Corollary 6.8]{paskunas2021finiteness}}]\label{1.1.4} Soit $\mathfrak{B}$ un bloc constitué de représentations absolument irréductibles. Soient $\Pi_1\in\mathrm{Ban}_{G, \psi}^{\mathrm{adm}}(L)^{\mathrm{fl}}_{\mathfrak{B}, \mathbf{m}_1}$ et $\Pi_2\in\mathrm{Ban}_{G, \psi}^{\mathrm{adm}}(L)^{\mathrm{fl}}_{\mathfrak{B}, \mathbf{m}_2}$ où $\mathbf{m}_1$ et $\mathbf{m}_2$ sont deux idéaux maximaux distincts de $R^{\mathrm{ps}, \psi\omega}_{\mathrm{tr}\bar{\rho}_{\mathfrak{B}}}[\tfrac{1}{p}]$, alors les groupes d'extensions de Yoneda $\mathrm{Ext}^i_{G, \psi}(\Pi_1, \Pi_2)$ calculés dans $\mathrm{Ban}_{G, \psi}^{\mathrm{adm}}(L)$ sont nuls pour tout $i\geq0$.
\end{theorem}
\begin{theorem}[{\cite[Theorem 1.1]{ds2013endomorphism}}]\label{1.1.7} Soit $\Pi$ une représentation de Banach (resp. localement analytique) admissible topologiquement irréductible de $G$ sur $L$, alors $\Pi$ est absolument irréductible si et seulement si $\mathrm{End}_G(\Pi)=L$.
\end{theorem}
\begin{proposition}[{\cite[Corollary 2.26]{paskunas2016on}}]\label{1.1.6} Soit $\Pi\in\mathrm{Ban}_{G, \psi}^{\mathrm{adm}}(L)$ une représentation supersingulière absolument irréductible, alors on a $\dim_L\mathrm{Ext}_{G, \psi}^1(\Pi, \Pi)=3$.
\end{proposition}
\begin{remark} Si l'on ne fixe pas le caractère central dans la Proposition \ref{1.1.6}, alors on a $\dim\mathrm{Ext}_G^1(\Pi, \Pi)=5$.
\end{remark}
\subsection{Demi-plan de Drinfeld}
Soit $D$ l'unique algèbre de quaternions ramifiée sur $\mathbb{Q}_p$ à isomorphisme près, $\mathscr{O}_D$ son unique ordre maximal à conjugaison près et soit $\varpi_D$ une uniformisante de $D$. Soit $S$ un $\breve{\z}_p$-schéma, un $\mathscr{O}_D$-module formel spécial sur $S$ est un groupe formel $p$-divisible $X$ sur $S$, de dimension $2$ et de hauteur $4$, muni d'une action de $\mathscr{O}_D$ telle que l'action induite de $\z_{p^2}$ sur l'algèbre de Lie de $X$ fait de celle-ci un $\mathscr{O}_S\otimes_{\z_p}\z_{p^2}$-module localement libre de rang $1$. Il existe une unique classe de $\mathscr{O}_D$-isogénie de $\mathscr{O}_D$-modules formels spéciaux sur $\bar{\mathbb{F}}_p$. Fixons un tel $\mathscr{O}_D$-module formel spécial $\mathbb{X}$. Le foncteur des déformations de $\mathbb{X}$ par quasi-isogénies $\mathscr{O}_D$-équivariantes est représentable \cite{rz1996period} par un schéma formel $p$-adique sur $\breve{\z}_p$. On note $\breve{\mathscr{M}_0}$ la fibre générique rigide de ce schéma formel. Un théorème fondamental de Drinfeld \cite{drinfeld1976coverings} fournit un isomorphisme $\breve{\mathscr{M}_0}\cong\check{\Omega}\times\z$, où $\check{\Omega}=\Omega\hat{\otimes}_{\mathbb{Q}_p}\breve{\mathbb{Q}}_p$ et $\Omega$ est le demi-plan de Drinfeld, un espace rigide sur $\mathbb{Q}_p$ dont les $\mathbb{C}_p$-points sont
$$\Omega(\mathbb{C}_p)=\mathbf{P}^1(\mathbb{C}_p)-\mathbf{P}^1(\mathbb{Q}_p).$$
Il admet une action de $G$ définie par $$g\cdot z=\frac{az+b}{cz+d}, \quad \forall~ g=\begin{psmallmatrix}
	a&b\\
	c&d
\end{psmallmatrix},~ z\in\Omega(\mathbb{C}_p).$$
L'espace $\breve{\mathscr{M}_0}$ est muni d'une action à gauche de $G$ donnée par 
$$g\cdot(X, \rho)=(X, \rho\circ g^{-1}),$$
qui correspond, par l'isomorphisme de Drinfeld, à l'action de $G$ sur le demi-plan et au décalage par $-v_p(\det g)$ sur $\z$, ainsi que d'une donnée de descente à la Weil, qui correspond via l'isomorphisme de Drinfeld au composé de la donnée de descente canonique et du décalage par $1$. Elle n'est donc pas effective, mais pour tout entier $t>0$, cette donnée de descente sur le quotient $p^{t\z}\backslash\breve{\mathscr{M}_0}$ de $\breve{\mathscr{M}_0}$ par l'action de l'élément $p^t$ du centre de $G$ devient effective. En prenant $t=1$, on obtient un modèle $\Sigma_0$ de $p^\z\backslash\breve{\mathscr{M}_0}$ sur $\mathbb{Q}_p$.

Soit $X^{\mathrm{un}}$ le groupe $p$-divisible rigide universel sur $\breve{\mathscr{M}_0}$. Drinfeld a défini \cite{drinfeld1976coverings} une tour de revêtements $\breve{\mathscr{M}_n}$ de $\breve{\mathscr{M}_0}$ par
$$\breve{\mathscr{M}_n}=X^{\mathrm{un}}[p^n]-X^{\mathrm{un}}[\varpi_D^{2n-1}]$$
pour $n\in\mathbb{N}^*$. Pour chaque $n\in\mathbb{N}^*$, l'espace $\breve{\mathscr{M}_n}$ est un revêtement étale galoisien de $\breve{\mathscr{M}_0}$ de groupe de Galois $\mathscr{O}_D^*/(1+p^n\mathscr{O}_D)$. Les revêtements $\breve{\mathscr{M}}_n$ sont définis sur $\breve{\mathbb{Q}}_p:=\widehat{\mathbb{Q}_p^{\mathrm{nr}}}$ et munis d'une action de $W_{\mathbb{Q}_p}$ compatible avec l'action naturelle sur $\breve{\mathbb{Q}}_p$. De plus, les revêtements $\breve{\mathscr{M}_n}$ sont munis des actions de $G$ et de $\check{G}$ commutant entre elles ainsi qu'avec l'action de $W_{\mathbb{Q}_p}$. Les flèches de transition $\breve{\mathscr{M}}_{n+1}\to\breve{\mathscr{M}}_n\to\Omega$ sont $W_{\mathbb{Q}_p}, \check{G}$ et $G$-équivariantes, où l'action de $\check{G}$ sur $\Omega$ est triviale. Les espaces $p^\z\backslash\breve{\mathscr{M}}_n$ admettent les modèles sur $\mathbb{Q}_p$ que l'on note $\Sigma_n$. Pour chaque $n\geq0$, on note $\mathscr{O}(\Sigma_n)$ l'anneau des fonctions analytiques sur $\Sigma_n$. Le but de la conjecture de Breuil et Strauch est de décrire les espaces $\mathscr{O}(\Sigma_n)$ en tant que représentations de $G\times\check{G}$.

Puisque $\Sigma_0$ est un espace de Stein, il en est de même pour $\Sigma_n$ pour tout $n\geq0$ puisque $\Sigma_n$ est un revêtement étale fini de $\Sigma_0$. Fixons $n\geq0$, il résulte de la \cite[Annexe A]{cdn2020cohomologie} que l'espace $\Sigma_n$ possède un modèle formel $\mathfrak{X}$ semi-stable $G\times\check{G}\times\mathcal{G}_{\mathbb{Q}_p}$-équivariant sur une extension finie de $\mathbb{Q}_p$. Notons $\tau_n$ la composée du morphisme $\Sigma_n\to\Sigma_0$ et de la rétraction de $\Sigma_0$ sur l'arbre de Bruhat-Tits. On fixe l'origine en le réseau standard et $B_i$ la boule fermée centrée en l'origine de rayon $i$ dans l'arbre. Alors la famille des affinoïdes $U_i:=\tau_n^{-1}(B_i)$ forme un recouvrement de Stein de $\Sigma_n$. De plus, $U_i$ est stable sous l'action de $K$ pour tout $i\geq1$. On a également un recouvrement $\{\mathfrak{X}_i\}$ de $\mathfrak{X}$, où la fibre générique associée à chaque $\mathfrak{X}_i$ est $U_i$.

On a $\Omega^1(\Sigma_n)=\varprojlim_i(\Omega^1(\mathfrak{X}_i)[\frac{1}{p}])$. L'espace $\Omega^1(\mathfrak{X}_i)[\frac{1}{p}]$ est un Banach et la norme $||\cdot||_{\mathfrak{X}_i}$ est induite par le réseau $\Omega^1(\mathfrak{X}_i)$. Cela induit une structure de Fréchet sur $\Omega^1(\Sigma_n)$. 

\begin{definition} Une forme différentielle $\omega\in\Omega^1(\Sigma_n)$ est dite bornée si $\omega\in\Omega^1(\mathfrak{X})[\frac{1}{p}]$.
\end{definition}
\begin{definition} Soient $V$ un $L$-espace localement convexe et $W$ un sous-ensemble de $V$, alors $W$ est dit borné dans $V$ si, pour tout voisinage ouvert $U$ de $0$, il existe $\lambda\in L$ tel que $\lambda W\subseteq U$. 
\end{definition}
Si $V=\varprojlim V_i$ est de plus un espace de Fréchet, alors $W$ est borné si et seulement si la projection de $W$ dans $V_i$ est bornée pour tout $i$.
\begin{definition} Une forme différentielle $\omega\in\Omega^1(\Sigma_n)$ est dite $G$-bornée si le sous-ensemble $\{g\cdot \omega\mid\forall g\in G\}$ est borné dans $\Omega^1(\Sigma_n)$.
\end{definition}
\begin{lemma}\label{1.2.6} Soit $\omega\in\Omega^1(\Sigma_n)$ une forme différentielle, alors $\omega$ est bornée si et seulement si $\omega$ est $G$-bornée. 
\end{lemma} 
\begin{proof}[Preuve] Fixons $i$. Comme les $g^{-1}\mathfrak{X}_i$ pour $g\in G$ forment un recouvrement de $\mathfrak{X}$, on a une injection $$\Omega^1(\mathfrak{X})\hookrightarrow\prod_{g\in G}\Omega^1(g^{-1}\mathfrak{X}_i).$$
	Supposons que $\omega$ est bornée, alors il existe $N>0$, qui est indépendant de $i$, tel que $||p^N\omega|_{g^{-1}\mathfrak{X}_i}||_{g^{-1}\mathfrak{X}_i}\leq 1$ pour tout $g\in G$. On en déduit que
	$$||p^N (g\cdot\omega|_{\mathfrak{X}_i})||_{\mathfrak{X}_i}\leq 1$$
	pour tout $g\in G$. Cela implique que $\{g\cdot \omega| g\in G\}$ est bornée.
	
	Réciproquement, fixons $i$ et supposons que $\omega\in\Omega^1(\Sigma_n)$ est $G$-bornée, alors il existe $N$, qui dépend de $i$, tel que $||p^N(g\cdot \omega|_{\mathfrak{X}_i})||_{\mathfrak{X}_i}\leq 1$ pour tout $g\in G$. Cela implique que $||(p^N\omega|_{g^{-1}\mathfrak{X}_i})||_{g^{-1}\mathfrak{X}_i}\leq1$ pour tout $g\in G$. Par conséquent, $p^N\omega\in\Omega^1(g^{-1}\mathfrak{X}_i)$. Comme les $g^{-1}\mathfrak{X}_i$ forment un recouvrement de $\mathfrak{X}$, on en déduit que $p^N\omega\in\Omega^1(\mathfrak{X})$, donc $\omega\in p^{-N}\Omega^1(\mathfrak{X})$.
\end{proof}
\begin{proposition}\label{1.2.4} On a $$H^1_{\mathrm{dR}}(\Sigma_n)=\varprojlim_i(H^1_{\mathrm{dR}}(\mathfrak{X}_i)[\tfrac{1}{p}]).$$
\end{proposition}
\begin{proof}[Preuve] On considère la suite exacte courte de Milnor suivante
	$$0\to R^1\varprojlim_i(H^0_{\mathrm{dR}}(\mathfrak{X}_i)[\tfrac{1}{p}])\to H^1_{\mathrm{dR}}(\Sigma_n)\to\varprojlim_i (H^1_{\mathrm{dR}}(\mathfrak{X}_i)[\tfrac{1}{p}])\to0.$$
	Comme le morphisme $H^0_{\mathrm{dR}}(\mathfrak{X}_{i+1})[\tfrac{1}{p}]\to H^0_{\mathrm{dR}}(\mathfrak{X}_i)[\tfrac{1}{p}]$ est surjectif pour chaque $i$, il suit du lemme de Mittag-Leffler que $R^1\varprojlim_i(H^0_{\mathrm{dR}}(\mathfrak{X}_i)[\tfrac{1}{p}])=0$. Par conséquent, on a $H^1_{\mathrm{dR}}(\Sigma_n)=\varprojlim_i(H^1_{\mathrm{dR}}(\mathfrak{X}_i)[\tfrac{1}{p}])$, d'où le résultat.
\end{proof}
\begin{proposition}\label{1.2.5} Soit $[\omega]\in H^1_{\mathrm{dR}}(\Sigma_n)$, alors $[\omega]\in H^1_{\mathrm{dR}}(\mathfrak{X})[\tfrac{1}{p}]$ si et seulement si $[\omega]$ est $G$-bornée.
\end{proposition}
\begin{proof}[Preuve] Considérons le morphisme $\tau_n$ qui est la composée du morphisme $\Sigma_n\to\Sigma_0$ et de la rétraction de $\Sigma_0$ sur l'arbre de Bruhat-Tits, alors il suit de \cite{coleman1982} que la préimage $Y:=\tau_n^{-1}(U)$ d'une boule ouverte de l'arbre est un sous-ensemble "wide open". D'après le \cite[Theorem 4.2]{coleman1989}, $\mathrm{H}^1_{\mathrm{dR}}(Y)$ est séparé et de dimension finie. Si le rayon de $U$ est suffisamment grand, alors les $gY$ pour $g\in G$ forment un recouvrement de $\Sigma_n$, et on a une injection $H^1_{\mathrm{dR}}(\Sigma_n)\hookrightarrow\prod_gH^1_{\mathrm{dR}}(gY)$ car le noyau est isomorphe à $H^1(\Gamma, L)=0$ où $\Gamma$ est l'arbre associé à $\Sigma_n$ d'après la construction dans \cite[Section 0.2.1]{cdn2022cohomologie}. Pour tout $g$, l'image de $H^1_{\mathrm{dR}}(\mathfrak{X})\subseteq H^1_{\mathrm{dR}}(\Sigma_n)$ par l'injection $H^1_{\mathrm{dR}}(\Sigma_n)\hookrightarrow\prod_gH^1_{\mathrm{dR}}(gY)$ est un réseau de $H^1_{\mathrm{dR}}(gY)$, ce qui induit une norme $||\cdot||_{gY}$ sur $H^1_{\mathrm{dR}}(gY)$.
	
	Supposons que $[\omega]$ est $G$-bornée, alors il existe $N$ tel que l'on a $||p^N(g\cdot[\omega]|_Y)||_Y\leq1$ pour tout $g\in G$, ce qui implique que $||p^N[\omega]|_{g^{-1}Y}||_{g^{-1}Y}\leq1$ pour tout $g\in G$. Par conséquent, $p^N[\omega]\in\mathrm{Im}(H^1_{\mathrm{dR}}(\mathfrak{X})\to H^1_{\mathrm{dR}}(g^{-1}Y))$ pour tout $g\in G$. Donc $p^N[\omega]\in H^1_{\mathrm{dR}}(\mathfrak{X})$. 
	
	Réciproquement, supposons que $[\omega]\in H^1_{\mathrm{dR}}(\mathfrak{X})[\frac{1}{p}]$, alors il existe $N>0$, qui est indépendant de $i$, tel que $||p^N[\omega]|_{g^{-1}Y}||_{g^{-1}Y}\leq1$ pour tout $g\in G$. On en déduit que
	$||p^N(g\cdot[\omega]|_Y)||_Y\leq 1$ pour tout $g\in G$. Cela implique que $\{g\cdot [\omega]| g\in G\}$ est bornée. Cela permet de conclure. 
\end{proof} 
\subsection{Complétés unitaires universels}\label{section1.3}
Étant donnée une représentation de Banach de $G$, on peut obtenir une représentation localement analytique en considérant les vecteurs localement analytiques. Dans l'autre sens, à partir d'une représentation localement analytique de $G$, une méthode naturelle pour obtenir une représentation de Banach est de prendre son complété unitaire universel. Rappelons la définition du complété unitaire universel
\begin{definition} Soit $\pi$ une représentation de $G$, le complété unitaire universel $\hat{\pi}$ (unique à isomorphisme près) de $\pi$ est une $L$-représentation de Banach unitaire de $G$, munie d'une application $L$-linéaire continue, $G$-équivariante $\iota: \pi\to\hat{\pi}$, qui est universelle au sens suivant: pour toute $L$-représentation de Banach unitaire $\Pi$ de $G$, tout morphisme continu $\pi\to\Pi$ se factorise de manière unique à travers $\iota$.
\end{definition}
D'après \cite{schneider2002nonarchimedean}, un réseau dans un espace vectoriel sur $L$ est un sous-$\mathscr{O}_L$-module engendrant. Notons qu'avec cette définition, un réseau n’est pas nécessairement séparé.
\begin{lemma}[{\cite[Lemma 1.3]{emerton2005p-adic}}] La $G$-représentation $\pi$ admet un complété unitaire universel si et seulement si l'ensemble des classes de commensurabilité des réseaux ouverts $G$-invariants dans $\pi$, ordonné par inclusion, contient un élément minimal.
\end{lemma}
\begin{remark} (i) Le complété unitaire universel de $\pi$ n'existe pas nécessairement. Même s'il existe, $\hat{\pi}$ peut être nul, voir le Lemme \ref{2.2.4} par exemple. De plus, si $\pi$ admet un réseau $\pi^+$ de type fini, alors la complétion de $\pi^+$ est la boule unité du complété unitaire universel de $\pi$, voir la \cite[Proposition 1.17]{emerton2005p-adic}.
	
	(ii) Supposons que $\pi$ est une représentation lisse (ou localement analytique) admissible et $\hat{\pi}$ existe, alors $\hat{\pi}$ peut être non admissible (Proposition \ref{3.2.1}). 
	
	(iii) Le foncteur de passage au complété unitaire universel n'est ni exact à gauche, ni à droite en général. Donc on ne peut pas déduire le Théorème \ref{0.0.1} directement de la suite de Dospinescu et Le Bras (Corollaire \ref{2.1.3}) en prenant les complétés unitaires universels. Cependant, pour l'exactitude de ce foncteur, on a le résultat de Colmez et Dospinescu ci-dessous.
\end{remark}
\begin{proposition}[{\cite[Proposition VII.8, Corollaire VII.9]{cd2014completes}}]\label{1.3.3} Soit $0\to\Pi_1\to\Pi\to\Pi_2\to0$ une suite exacte stricte de représentations de $G$ sur des $L$-espaces vectoriels localement convexes. Si $\hat{\Pi}$ existe, alors
	
	(i) $\hat{\Pi}_2$ existe aussi, et le morphisme $\hat{\Pi}\to\hat{\Pi}_2$ induit par $\Pi\to\Pi_2$ est surjectif.
	
	(ii) Si de plus $\hat{\Pi}_1$ existe, alors $\mathrm{Im}(\hat{\Pi}_1\to\hat{\Pi})$ est dense dans $\mathrm{Ker}(\hat{\Pi}\to\hat{\Pi}_2)$.
	
	(iii) Supposons de plus que $\hat{\Pi}_1$ est admissible, alors on a une suite exacte de $L$-espaces vectoriels $$\hat{\Pi}_1\to\hat{\Pi}\to\hat{\Pi}_2\to0.$$
\end{proposition}
En général, il est difficile de décrire le complété unitaire universel d'une représentation de $G$. Heureusement, on a le résultat suivant, dont la preuve repose sur la correspondance de Langlands $p$-adique pour $G$.
\begin{theorem}[{\cite[Théorème 0.2]{cd2014completes}}]\label{1.3.4} Soit $\Pi$ une représentation de $G$ unitaire, résiduellement de longueur finie, admettant un caractère central, alors $\Pi$ est le complété unitaire universel de $\Pi^{\mathrm{an}}$.
\end{theorem}
 		\section{Représentations localement analytiques de $G$} 
\subsection{Conjecture de Breuil-Strauch d'après Dospinescu-Le Bras}
\label{section2.1}
Rappelons qu'un $L-(\varphi, N, \mathcal{G}_{\mathbb{Q}_p})$-module est un $L\otimes\mathbb{Q}_p^{\mathrm{nr}}$-module muni d'un Frobenius semi-linéaire $\varphi$, d'un opérateur $N$ tel que $N\varphi=p\varphi N$ et d'une action semi-linéaire lisse de $\mathcal{G}_{\mathbb{Q}_p}$ commutant à $N$ et $\varphi$. Soit $M$ un $L-(\varphi, N, \mathcal{G}_{\mathbb{Q}_p})$-module irréductible de rang $2$. La recette de Fontaine \cite{fontaine1994representations} nous permet d'associer une représentation de Weil-Deligne $\mathrm{WD}(M)$ à $M$ qui est irréductible puisque $M$ l'est. En utilisant le foncteur de Langlands local, on obtient une $L$-représentation supercuspidale lisse irréductible $\mathrm{LL}(M):=\mathrm{LL}(\mathrm{WD}(M))$. 

D'après la classification des représentations supercuspidales de $G$ (voir le Corollaire \ref{classsuper}), il existe une représentation irréductible $\sigma_M$ de $KZ$ de dimension finie, telle que $$\mathrm{ind}_{KZ}^G\sigma_M=\mathrm{LL}(M)$$   
ou bien $$\mathrm{ind}_{KZ}^G\sigma_M=\mathrm{LL}(M)\oplus(\mathrm{LL}(M)\otimes\mu_{-1}),$$ 
où, pour $\lambda\in\mathbb{Q}_p^*$, $\mu_\lambda$ désigne le caractère de $\mathbb{Q}_p^*$ défini par $x\mapsto\lambda^{v_p(x)}$. 

Ensuite, on note $\mathrm{JL}(M):=\mathrm{JL}(\mathrm{LL}(M))$ la $L$-représentation lisse irréductible de dimension finie de $\breve{G}$ qui est attachée à $\mathrm{LL}(M)$ par la correspondance de Jacquet-Langlands locale \cite{jl1970automorphic}. Posons
$$\mathscr{O}[M]=\mathrm{Hom}_{\breve{G}}(\mathrm{JL}(M), L\otimes_{\mathbb{Q}_p}\mathscr{O}(\Sigma_n)),$$
$$\Omega^1[M]=\mathrm{Hom}_{\breve{G}}(\mathrm{JL}(M), L\otimes_{\mathbb{Q}_p}\Omega^1(\Sigma_n)),$$
$$H_{\mathrm{dR}}^1[M]=\mathrm{Hom}_{\breve{G}}(\mathrm{JL}(M), L\otimes_{\mathbb{Q}_p}H^1_{\mathrm{dR}}(\Sigma_n)).$$
L'espace $\Sigma_n$ est une courbe de Stein, ce qui fournit une suite exacte d'espaces de Fréchet avec action de $G\times\breve{G}$
$$0\to H_{\mathrm{dR}}^0(\Sigma_n)\to\mathscr{O}(\Sigma_n)\to\Omega^1(\Sigma_n)\to H_{\mathrm{dR}}^1(\Sigma_n)\to0.$$
Comme $\dim_L\mathrm{JL}(M)>1$, il résulte de \cite{strauch2008deformation} et \cite{faltings2002a, fgl2008le} que l'on a $\mathrm{Hom}_{\breve{G}}(\mathrm{JL}(M), L\otimes_{\mathbb{Q}_p}H_{\mathrm{dR}}^0(\Sigma_n))=0$. Donc en passant aux composantes $\mathrm{JL}(M)$-isotypiques, on obtient la suite exacte de $G$-représentations suivante
$$0\to\mathscr{O}[M]\to\Omega^1[M]\to H_{\mathrm{dR}}^1[M]\to0.$$ 

On note $\mathscr{O}(k)(\Sigma_n)$ la représentation de $G$ sur $\mathscr{O}(\Sigma_n)$ définie par 
$$\begin{psmallmatrix}
	a&b\\
	c&d
\end{psmallmatrix}*_kf=(a-cz)^{-k}\cdot(\begin{psmallmatrix}
	a&b\\
	c&d
\end{psmallmatrix}\cdot f),$$
où $z$ est la coordonnée sur le demi-plan de Drinfeld. Puisque $\Sigma_n$ est étale sur $\Sigma_0$, on a une trivialisation $\Omega^1(\Sigma_n)\cong\mathscr{O}(\Sigma_n)dz$, qui induit un isomorphisme de $G$-représentations $\Omega^1(\Sigma_n)\cong\mathscr{O}(2)(\Sigma_n)\otimes\det$.

D'après le \cite[Théorème 3.2]{dl2017revetements}, les représentations $\mathscr{O}[M]^*$ et $\Omega^1[M]^*$ sont localement analytique. En appliquant le \cite[Lemma 3.6]{st2003algebras} et le \cite[Théorème 8.8]{dl2017revetements}, on obtient que la représentation $\mathscr{O}[M]^*$ est admissible. D'après la \cite[Proposition 2.2]{st2001u}, la représentation lisse $\mathrm{LL}(M)$ est admissible en tant que représentation localement analytique. Comme la catégorie des représentations localement analytiques admissibles est abélienne \cite[Proposition 6.4]{st2003algebras}, la représentation localement analytique $\Omega^1[M]^*$ est admissible. De plus, Colmez a montré que $\mathscr{O}[M]^*$ est topologiquement absolument irréductible dans \cite{colmez2019correspondance}.

Comme $\mathrm{WD}(M)$ est irréductible, les pentes de $\varphi$ sont toutes égales à un même nombre rationnel appelé la pente de $M$. On dit que $M$ est supercuspidal si $\mathrm{WD}(M)$ est irréductible et de pente $\frac{1}{2}$. Supposons maintenant que $M$ est supercuspidal. On définit un $L$-module de rang $2$ 
$$M_{\mathrm{dR}}:= (\overline{\mathbb{Q}}_p\otimes_{\mathbb{Q}_p^{\mathrm{nr}}}M)^{\mathcal{G}_{\mathbb{Q}_p}}.$$
On note $\Pi(M)\subseteq\mathrm{Ban}_{G, \psi}^{\mathrm{adm}}(L)$ l'ensemble des représentations admissibles absolument irréductibles $\Pi$ de $G$ telles que $\Pi^{\mathrm{lisse}}=\mathrm{LL}(M)$. On note $\mathcal{V}(M)$ l'ensemble des $L$-représentations absolument irréductibles $V$ de dimension $2$ de $\mathcal{G}_{\mathbb{Q}_p}$, potentiellement semi-stables à poids de Hodge-Tate $0, 1$ et telles que $D_{\mathrm{pst}}(V)=M$. Un des résultats principaux de \cite{cdp2014the} montre que le foncteur de Colmez \cite{colmez2010representations} induit une bijection $\Pi\mapsto V(\Pi)$ entre $\Pi(M)$ et $\mathcal{V}(M)$.
\begin{proposition}\label{2.1.1} Il y a une bijection naturelle entre $\Pi(M)$ et $\mathbf{P}(M_\mathrm{dR})$. 
\end{proposition}
\begin{proof}[Preuve] Compte tenu de ce qui précède, il suffit de montrer qu'il y a une bijection naturelle entre $\mathcal{V}(M)$ et $\mathbf{P}(M_\mathrm{dR})$. Soit $V\in\mathcal{V}(M)$, alors $D_{\mathrm{pst}}(V)=M$, ce qui induit un isomorphisme de $L$-espaces vectoriels 
	$$\mathrm{D}_{\mathrm{dR}}(V)\cong(\overline{\mathbb{Q}}_p\otimes_{\mathbb{Q}_p^{\mathrm{nr}}}D_{\mathrm{pst}}(V))^{\mathcal{G}_{\mathbb{Q}_p}}\cong M_\mathrm{dR}.$$
	La filtration de Hodge $\mathrm{Fil}^0(\mathrm{D}_{\mathrm{dR}}(V))$ nous donne une droite dans $M_\mathrm{dR}$ puisque $V$ est à poids $0$ et $1$.
	
	Réciproquement, étant donné une droite $\mathscr{L}$, la filtration correspondante $\mathrm{Fil}_{\mathscr{L}}$ sur $M_\mathrm{dR}$ est faiblement admissible car la représentation de Weil-Deligne attachée à $M$ est irréductible \cite[Theorem 5.2]{bs2007first}. Le théorème de Colmez-Fontaine \cite{cf2000construction} permet donc de construire une $L$-représentation 
	$$V_{M, \mathscr{L}}=\mathrm{Ker}((B^+_{\mathrm{cris}}\otimes_{\mathbb{Q}_p^{\mathrm{nr}}}M)^{\varphi=p}\to\mathbb{C}_p\otimes_{\mathbb{Q}_p}(M_{\mathrm{dR}}/\mathscr{L}))$$
	dans $\mathcal{V}(M)$. 
	
	Les deux constructions sont inverses l'une de l'autre. Cela permet de conclure.
\end{proof}
Notons $\Pi_{M, \mathscr{L}}:=\boldsymbol{\Pi}(V_{M, \mathscr{L}})$ où $\boldsymbol{\Pi}$ est le foncteur de Langlands $p$-adique défini par Colmez \cite{colmez2010representations}. Les représentations de Banach $\Pi_{M, \mathscr{L}}$ sont topologiquement irréductibles et on a $\Pi_{M, \mathscr{L}}\ncong\Pi_{M, \mathscr{L}'}$ pour $\mathscr{L}\neq\mathscr{L}'$.
\begin{theorem}[{\cite[Théorème 1.4]{dl2017revetements}}]\label{2.1.2} Il existe un isomorphisme de $G$-modules topologiques
	$$H_{\mathrm{dR}}^1[M]\cong M_{\mathrm{dR}}^*\otimes_L\mathrm{LL}(M)^*$$
	De plus, la préimage de $\mathscr{L}^\perp\otimes_L\mathrm{LL}(M)^*$ dans $\Omega^1[M]$ est isomorphe à $\Pi_{M, \mathscr{L}}^{\mathrm{an},*}$.
\end{theorem}
\begin{remark} On va dire un mot sur la preuve du Théorème \ref{2.1.2}. Tout d'abord, les résultats \cite[Théorème 5.1, Proposition 11.6]{dl2017revetements}, dont les preuves utilisent la compatibilité locale-globale d'Emerton et l'uniformisation de Cerednik-Drinfeld, montrent qu'il existe un plongement $G$-équivariant de $\Pi_{M \mathscr{L}}^{\mathrm{an}, *}$ dans $\Omega^1[M]$, qui fait de $\Pi_{M \mathscr{L}}^{\mathrm{an}, *}$ un sous-espace $G$-stable de $\Omega^1[M]$ contenant l'image de $\mathscr{O}[M]\hookrightarrow\Omega^1[M]$. En utilisant la théorie du modèle de Kirillov, on peut identifier $H_{\mathrm{dR}}^1[M]$ à $M_{\mathrm{dR}}^*\otimes_L\mathrm{LL}(M)^*$, tandis que le quotient de $(\Pi_{M \mathscr{L}}^{\mathrm{an}})^*$ par $\mathscr{O}[M]$ s'identifie à $\mathscr{L}^\perp\otimes_L\mathrm{LL}(M)^*$.
\end{remark}
En passant au dual, on a un diagramme commutatif à lignes exactes de représentations localement analytiques
\begin{equation}\label{equation4}\xymatrix@R=5mm@C=4mm{
		0 \ar[r] & M_{\mathrm{dR}}\otimes_L\mathrm{LL}(M)\ar[r] \ar[d] &\Omega^1[M]^*\ar[r] \ar[d]&\mathscr{O}[M]^*\ar[r] \ar@{=}[d]&0\\
		0\ar[r]&(M_{\mathrm{dR}}/\mathscr{L})\otimes_L\mathrm{LL}(M)\ar[r] &\Pi_{M, \mathscr{L}}^\mathrm{an}\ar[r] & \mathscr{O}[M]^*\ar[r] &0,
	}
\end{equation}
d'où on déduit le résultat suivant
\begin{corollary}\label{2.1.3} On a une suite exacte courte
	$$0\to\mathscr{L}\otimes_L\mathrm{LL}(M)\to\Omega^1[M]^*\to\Pi_{M, \mathscr{L}}^{\mathrm{an}}\to0.$$
\end{corollary}
De plus, on va montrer que la suite exacte du Corollaire \ref{2.1.3} est non scindée (Corollaire \ref{2.3.10}).
\subsection{Complétés unitaires universels de $\mathrm{LL}(M)$ et $\mathscr{O}[M]^*$}\label{section2.2}
Tout d'abord, on a besoin du lemme suivant, qui généralise la \cite[Proposition 3.2]{cdn2023correspondance}.
\begin{lemma}\label{2.2.2} Soit $f: \Pi_1\to\Pi_2$ un morphisme d'image dense de $L$-représentations unitaires de $G$. Supposons que $\Pi_2^+/\pi_L$ est de type fini sur $\kappa_L[G]$ où $\Pi_2^+$ est la boule unité de $\Pi_2$, alors $f$ est surjectif.
\end{lemma}
\begin{proof}[Preuve] Par la densité de $f(\Pi_1)$ dans $\Pi_2$ et la finitude de $\Pi_2^+/\pi_L$, il existe $w_1,..., w_r\in\Pi_1$ tels que $f(w_1),..., f(w_r)\in\Pi_2^+$ et que les $f(w_i)\mod\pi_L$ engendrent $\Pi_2^+/\pi_L$ sur $\kappa_L[G]$. On note $W$ le sous-$\mathscr{O}_L[G]$-module fermé de $\Pi_1$ engendré par $w_1,..., w_r$. Par la continuité, $f$ envoie $W$ dans $\Pi_2^+$, et par la construction de $W$, induit une surjection modulo $\pi_L$. On en déduit, puisque $W$ et $\Pi_2^+$ sont complets pour la topologie $\pi_L$-adique, que $f|_W: W\to\Pi_2^+$ est surjectif. Cela permet de conclure.
\end{proof}
On note $\Pi_{M, \mathscr{L}, j}$ la représentation unique, à isomorphisme près, non scindée, de de Rham, de caractère central $\psi$, et de longueur $j$ dont tous les facteurs de Jordan-Hölder sont isomorphes à $\Pi_{M, \mathscr{L}}$. L'existence et l'unicité d'une telle représentation sont données par la \cite[Section 4.2]{cdn2023correspondance}. Notons $N_{\mathscr{L}, j}$ le noyau de la surjection $\widehat{\mathrm{LL}(M)}\to\Pi_{M, \mathscr{L}, j}$. D'après la \cite[Section 4.2]{cdn2023correspondance}, la représentation $\Pi_{M, \mathscr{L}, j}$ est un module sur $L[T]/T^j$, et les sous-représentations fermées de $\Pi_{M, \mathscr{L}, j}$ sont les $T^k\Pi_{M, \mathscr{L}, j}=\Pi_{M, \mathscr{L}, j-k}$ pour $1\leq k\leq j-1$. On en déduit le résultat suivant.
\begin{proposition} \label{finalbanach} Soient $j, n$ deux entiers positifs, alors on a $$\mathrm{Hom}_G(\Pi_{M, \mathscr{L}, j}, \Pi_{M, \mathscr{L}, n})=L[T]/T^{\min(n, j)}.$$
\end{proposition} 
Par ailleurs, il découle de la \cite[Remarque 4.7 (i)]{cdn2023correspondance} qu'il existe un isomorphisme 
$$\Pi_{M, \mathscr{L}, j}^{\mathrm{lisse}}=(L[T]/T^j)\otimes_L\mathrm{LL}(M).$$
Comme la composée $1\otimes\mathrm{LL}(M)\hookrightarrow\Pi_{M, \mathscr{L}, j}^{\mathrm{lisse}}\hookrightarrow\Pi_{M, \mathscr{L}, j}$ est d'image dense \cite[Remarque 4.7 (ii)]{cdn2023correspondance}, on obtient un morphisme $\widehat{\mathrm{LL}(M)}\to\Pi_{M, \mathscr{L}, j}$ qui est d'image dense pour tout $j$. 
\begin{corollary}[{\cite[Remarque 4.7]{cdn2023correspondance}}] \label{2.2.3} Le morphisme construit ci-dessous $\widehat{\mathrm{LL}(M)}\to\Pi_{M, \mathscr{L}, j}$ est surjectif pour tout $j$.
\end{corollary}
\begin{proof}[Preuve] Le résultat se déduit directement du Lemme \ref{2.2.2}.
\end{proof}	
\begin{lemma}\label{2.2.4} Le complété unitaire universel $\widehat{\mathscr{O}[M]^*}$ de $\mathscr{O}[M]^*$ est nul. 
\end{lemma}
\begin{proof}[Preuve] En appliquant la Proposition \ref{1.3.3} (i) à la deuxième ligne du diagramme (\ref{equation4}), et en utilisant le Théorème \ref{1.3.4}, on obtient une surjection $\Pi_{M, \mathscr{L}}\twoheadrightarrow\widehat{\mathscr{O}[M]^*}$ pour toute $\mathscr{L}$. Le résultat voulu en découle car on a $\Pi_{M, \mathscr{L}}\ncong\Pi_{M, \mathscr{L}'}$ si $\mathscr{L}\neq\mathscr{L}'$.
\end{proof}
On peut aussi montrer le Lemme \ref{2.2.4} plus directement. Comme $\widehat{\mathscr{O}[M]^*}=\mathscr{O}[M]^{b, *}$, cela découle du lemme suivant
\begin{lemma} Il n'y a pas de vecteurs $G$-bornés non nuls dans $\mathscr{O}[M]$, autrement dit, on a $\mathscr{O}[M]^b=0$.
\end{lemma}
\begin{proof}[Preuve] Soit $0\neq f\in\mathscr{O}[M]=\mathrm{Hom}_{\breve{G}}(\mathrm{JL}(M), L\otimes_{\mathbb{Q}_p}\mathscr{O}(\Sigma_n))$. D'une part, il suit du \cite[Remarque A.2]{cdn2020cohomologie} ou du \cite[Corollaire 2.10]{cdn2022cohomologie} que $f(v)$ est constante sur chaque composante connexe de $\Sigma_n$ pour tout $v\in\mathrm{JL}(M)$. D'autre part, il suit des résultats de Strauch \cite{strauch2008geometrically} pour la tour de Lubin-Tate et de l'isomorphisme de Faltings-Fargues \cite{faltings2002a, fgl2008le} que $\check{G}$ agit par la norme réduite $N: \check{G}\to\mathbb{Q}_p^\times$ sur l'ensemble des composantes connexes de $\Sigma_n$. Il en résulte que $f(g\cdot v)=g\cdot f(v)=N(g)f(v)=f(N(g)\cdot v)$ pour tout $v\in \mathrm{JL}(M)$ et $g\in\check{G}$. Comme $\mathrm{JL}(M)$ est irréductible, $f$ est injectif, donc on a $g\cdot v=N(g)\cdot v$ pour tout $v\in \mathrm{JL}(M)$, ce qui est absurde car $\mathrm{JL}(M)$ est irréductible. On en déduit que $f=0$. Cela permet de conclure.
\end{proof}
\begin{corollary}\label{2.2.6} La suite exacte des représentations localement analytiques
	$$0\to\mathscr{L}\otimes_L\mathrm{LL}(M)\to\Pi_{M, \mathscr{L}}^\mathrm{an}\to\mathscr{O}[M]^*\to0$$
	est non scindée.
\end{corollary}
\begin{proof}[Preuve] Supposons qu'il y a une section non nulle $\mathscr{O}[M]^*\to\Pi_{M, \mathscr{L}}^\mathrm{an}$, ce qui induit un morphisme non nul $\mathscr{O}[M]^*\to\Pi_{M, \mathscr{L}}$. D'après la propriété universelle du complété unitaire universel, ce morphisme se factorise par $\mathscr{O}[M]^*\to\widehat{\mathscr{O}[M]^*}$, qui est un morphisme nul par le Lemme \ref{2.2.4}, donc on a une contradiction. Cela permet de conclure.
\end{proof}
\subsection{Entrelacements entre $\mathrm{LL}(M)$, $\Pi_{M, \mathscr{L}}^{\mathrm{an}}$, $\Omega^1[M]^*$ et $\mathscr{O}[M]^*$}
On conserve la notation précédente: $M$ désigne un $L-(\varphi, N, \mathcal{G}_{\mathbb{Q}_p})$-module irréductible de rang $2$. Il découle de la construction de $\mathrm{LL}(M)$ que le caractère central $\psi$ de $\mathrm{LL}(M)$ est $\det\mathrm{WD}(M)\cdot|\ |$, vu comme caractère de $W_{\mathbb{Q}_p}^{\mathrm{ab}}\cong\mathbb{Q}_p^\times$ par l'application de réciprocité locale normalisée de telle sorte que le frobenius arithmétique s'envoie vers $p$. Rappelons que la catégorie des représentations localement analytiques admissibles est abélienne \cite[Proposition 6.4]{st2003algebras}. Les représentations $\mathrm{LL}(M)$, $\Pi_{M, \mathscr{L}}^{\mathrm{an}}$, $\Omega^1[M]^*$ et $\mathscr{O}[M]^*$ sont toutes localement analytiques admissibles, et on va calculer leurs entrelacements dans cette catégorie. Notons $\mathrm{Ext}^1_{G, \psi}$ le groupe d'extensions localement analytique à caractère central $\psi$ fixé.
\subsubsection{Résultat de Ding}
D'après le Théorème \ref{2.1.2}, la représentation $\Omega^1[M]^*$ est une extension de $\mathscr{O}[M]^*$ par $M_{\mathrm{dR}}\otimes_L\mathrm{LL}(M)$. Grâce au produit cup
$$\mathrm{Ext}^1_{G, \psi}(\mathscr{O}[M]^*, M_{\mathrm{dR}}\otimes_L\mathrm{LL}(M))\times M_{\mathrm{dR}}^*\to\mathrm{Ext}^1_{G, \psi}(\mathscr{O}[M]^*, \mathrm{LL}(M)),$$
la représentation $\Omega^1[M]^*$ induit un morphisme
$$M_{\mathrm{dR}}^*\to\mathrm{Ext}^1_{G, \psi}(\mathscr{O}[M]^*, \mathrm{LL}(M)),$$
qui envoie $\mathscr{L}^\perp=(M_{\mathrm{dR}}/\mathscr{L})^*\hookrightarrow M_{\mathrm{dR}}^*$ vers $L[\Pi_{\mathscr{L}}^{\mathrm{an}}]$ en vertu du Théorème \ref{2.1.2}. Ding a démontré que ce morphisme est un isomorphisme, ce qui était une conjecture de Breuil dans \cite{breuil2019ext}.
\begin{theorem}[{\cite[Theorem 2.5]{ding2022locally}}]\label{2.3.3} Le morphisme ci-dessus $$M_{\mathrm{dR}}^*\xrightarrow{\sim}\mathrm{Ext}^1_{G, \psi}(\mathscr{O}[M]^*, \mathrm{LL}(M))$$
	est un isomorphisme. En particulier, on a $\dim_L\mathrm{Ext}^1_{G, \psi}(\mathscr{O}[M]^*, \mathrm{LL}(M))=2$.
\end{theorem}
\begin{remark} Le Théorème \ref{2.3.3} implique que toute extension de $\mathrm{LL}(M)$ par $\mathscr{O}[M]^*$ est isomorphe à $\Pi_{M, \mathscr{L}}^{\mathrm{an}}$ pour une certaine $\mathscr{L}$ en prenant le push-out de
	$$\xymatrix@R=5mm@C=4mm{
		0\ar[r]&M_{\mathrm{dR}}\otimes_L\mathrm{LL}(M)\ar[r] \ar[d] &\Omega^1[M]^*\ar[r]&\mathscr{O}[M]^*\ar[r]&0.\\&
		(M_{\mathrm{dR}}/\mathscr{L})\otimes_L\mathrm{LL}(M)&&
	}$$
	Autrement dit, $\Omega^1[M]^*$ est l'extension universelle de $M_{\mathrm{dR}}\otimes_L\mathrm{LL}(M)$ par $\mathscr{O}[M]^*$.
\end{remark}
\begin{remark} On remarque que Su \cite{su2025on} a également montré le Théorème \ref{2.3.3} en utilisant le foncteur de wall-crossing. De plus, ce foncteur lui permet aussi montrer que $\dim_L\mathrm{Ext}^1_{G, \psi}(\mathrm{LL}(M), \mathscr{O}[M]^*)=2$.
\end{remark}
\subsubsection{Autres entrelacements}
On note 
$$a^+=\begin{psmallmatrix}
	1&0\\
	0&0
\end{psmallmatrix},\quad a^-=\begin{psmallmatrix}
	0&0\\
	0&1
\end{psmallmatrix},\quad u^+=\begin{psmallmatrix}
	0&1\\
	0&0
\end{psmallmatrix},\quad u^-=\begin{psmallmatrix}
	0&0\\
	1&0
\end{psmallmatrix}$$
une base de $\mathfrak{g}=\mathrm{Lie}(G)$, alors on a le lemme suivant
\begin{lemma}\label{2.3.1} On a 
	
	(i) $\mathrm{Hom}_G(\mathrm{LL}(M), \mathscr{O}[M]^*)=0$.
	
	(ii) $\mathrm{Hom}_G(\mathscr{O}[M]^*, \mathrm{LL}(M))=0$.
	
	(iii) $\mathrm{Hom}_G(\mathscr{O}[M]^*, \Pi_{M, \mathscr{L}}^{\mathrm{an}})=0$.
	
	(iv) $\mathrm{Hom}_G(\mathscr{O}[M]^*, \Omega^1[M]^*)=0$.
	
	(v) $\mathrm{Hom}_G(\Omega^1[M]^*, \mathscr{O}[M]^*)=L$.
\end{lemma}
\begin{proof}[Preuve] (i) Comme $\mathrm{LL}(M)$ est lisse, il suffit de montrer que $\mathscr{O}[M]^{*, \mathrm{lisse}}=0$. Sur $\mathscr{O}[M]^*$, on a $a^+=\partial u^+-1$ par le \cite[Lemme 3.5]{dl2017revetements}, donc les actions de $a^+$ et $u^+$ ne peuvent pas être nulles simultanément, ce qui permet de conclure.
	
	(ii) Soit $0\neq\phi\in\mathrm{Hom}_G(\mathscr{O}[M]^*, \mathrm{LL}(M))$. La composée de $\phi$ et de l'injection naturelle $\mathrm{LL}(M)\hookrightarrow\widehat{\mathrm{LL}(M)}$ est non nul, mais cette composée se factorise par le morphisme $\mathscr{O}[M]^*\to\widehat{\mathscr{O}[M]^*}$ d'après la propriété universelle de $\widehat{\mathscr{O}[M]^*}$, qui est nul en vertu du Lemme \ref{2.2.4}, donc on a une contradiction.  
	
	(iii) Soit $0\neq\phi\in\mathrm{Hom}_G(\mathscr{O}[M]^*, \Pi_{M, \mathscr{L}}^{\mathrm{an}})$. La composée de $\phi$ et de l'injection naturelle $\Pi_{M, \mathscr{L}}^{\mathrm{an}}\hookrightarrow\Pi_{M, \mathscr{L}}$ est non nul, mais cette composée se factorise par le morphisme $\mathscr{O}[M]^*\to\widehat{\mathscr{O}[M]^*}$ d'après la propriété universelle de $\widehat{\mathscr{O}[M]^*}$, qui est nul en vertu du Lemme \ref{2.2.4}, donc on a une contradiction.  
	
	(iv) Soit $0\neq \phi\in\mathrm{Hom}_G(\mathscr{O}[M]^*, \Omega^1[M]^*)$. La composée de $\phi$ et de la surjection naturelle $\Omega^1[M]^*\twoheadrightarrow\Pi_{M, \mathscr{L}}^{\mathrm{an}}$ du Corollaire \ref{2.1.3} est nulle en vertu de (iii), ce qui induit un morphisme non nul $\mathscr{O}[M]^*\to\mathscr{L}\otimes_L\mathrm{LL}(M)$. Or, cela contredit l'assertion (ii).
	
	(v) En utilisant le foncteur $\mathrm{Hom}_G(-, \mathscr{O}[M]^*)$, on a une suite exacte 
	$$0\to\mathrm{Hom}_G(\mathscr{O}[M]^*, \mathscr{O}[M]^*)\to\mathrm{Hom}_G(\Omega^1[M]^*, \mathscr{O}[M]^*)\to\mathrm{Hom}_G(M_{\mathrm{dR}}\otimes_L\mathrm{LL}(M), \mathscr{O}[M]^*).$$
	Le résultat découle de (i) et du Théorème \ref{1.1.7}. Cela permet de conclure.
\end{proof}
\begin{lemma}\label{2.3.2} On a $$\mathrm{Ext}^1_{G, \psi}(\mathrm{LL}(M), \mathrm{LL}(M))=0.$$
\end{lemma}
\begin{proof}[Preuve] Supposons que l'on a une suite exacte des représentations localement analytiques
	$$0\to\mathrm{LL}(M)\to \pi\to\mathrm{LL}(M)\to0.$$
	Comme $\mathrm{LL}(M)$ est lisse, l'action de $\mathfrak{sl}_2$ sur $\mathrm{LL}(M)$ est nulle, ce qui implique que $\pi$ est tué par $[\mathfrak{sl}_2, \mathfrak{sl}_2]=\mathfrak{sl}_2$. On en déduit que $\pi$ est lisse. Donc on peut calculer $\mathrm{Ext}^1_{G, \psi}(\mathrm{LL}(M), \mathrm{LL}(M))$ dans la catégorie des représentations lisses à caractère central fixé. Il est classique que $\mathrm{LL}(M)$ est à la fois injective et projective dans cette catégorie \cite{v1996rep}, ce qui permet de conclure.
\end{proof}
\begin{lemma}\label{2.3.4} On a 
	
	(i) $\mathrm{Hom}_G(\Pi_{M, \mathscr{L}}^{\mathrm{an}}, \Pi_{M, \mathscr{L}'}^{\mathrm{an}})=0$ pour $\mathscr{L}\neq\mathscr{L}'$.
	
	(ii) $\mathrm{Hom}_G(\Omega^1[M]^*, \mathrm{LL}(M))=0$.
	
	(iii) $\mathrm{Hom}_G(\mathrm{LL}(M), \Omega^1[M]^*)=M_{\mathrm{dR}}$.
	
	(iv) $\mathrm{Hom}_G(\Pi_{M, \mathscr{L}}^{\mathrm{an}}, \mathrm{LL}(M))=0$.
	
	(v) $\mathrm{Hom}_G(\Pi_{M, \mathscr{L}, j}^{\mathrm{an}}, \Pi_{M, \mathscr{L}, n}^{\mathrm{an}})=L[T]/T^{\min(n, j)}$.
\end{lemma}
\begin{proof}[Preuve] (i) Tout d'abord, la représentation $\Pi_{M, \mathscr{L}}^{\mathrm{an}}$ est localement analytique, donc l'image d'un morphisme $\Pi_{M, \mathscr{L}}^{\mathrm{an}}\to\Pi_{M, \mathscr{L}'}$ est contenu dans $\Pi_{M, \mathscr{L}'}^{\mathrm{an}}$. Cela implique que 
	$$\mathrm{Hom}_G(\Pi_{M, \mathscr{L}}^{\mathrm{an}}, \Pi_{M, \mathscr{L}'}^{\mathrm{an}})=\mathrm{Hom}_G(\Pi_{M, \mathscr{L}}^{\mathrm{an}}, \Pi_{M, \mathscr{L}'}).$$
	Ensuite, en utilisant la propriété universelle du complété universel et le fait que $\widehat{\Pi_{M, \mathscr{L}}^{\mathrm{an}}}=\Pi_{M, \mathscr{L}}$ (Théorème \ref{1.3.4}), on a
	$$\mathrm{Hom}_G(\Pi_{M, \mathscr{L}}^{\mathrm{an}}, \Pi_{M, \mathscr{L}'})=\mathrm{Hom}_G(\Pi_{M, \mathscr{L}}, \Pi_{M, \mathscr{L}'}).$$
	Enfin, la Proposition \ref{2.1.1} implique que $$\mathrm{Hom}_G(\Pi_{M, \mathscr{L}}, \Pi_{M, \mathscr{L}'})=0.$$
	Le résultat s'en déduit en combinant les isomorphismes ci-dessus.
	
	(ii) En utilisant le foncteur $\mathrm{Hom}_G(-, \mathrm{LL}(M))$ à la première ligne du diagramme (\ref{equation4}), on a une suite exacte 
	\begin{align*}0&\to\mathrm{Hom}_G(\mathscr{O}[M]^*, \mathrm{LL}(M))\to\mathrm{Hom}_G(\Omega^1[M]^*, \mathrm{LL}(M))\to\mathrm{Hom}_G(M_{\mathrm{dR}}\otimes_L\mathrm{LL}(M), \mathrm{LL}(M))\\&\xrightarrow{\delta}\mathrm{Ext}^1_{G, \psi}(\mathscr{O}[M]^*, \mathrm{LL}(M)).
	\end{align*}
	Il découle de la définition du morphisme de connection $\delta$ (par push-out) et du paragraphe précédant le Théorème \ref{2.3.3} que, le morphisme du Théorème \ref{2.3.3} coïncide avec $\delta$ $$\mathrm{Hom}_G(M_{\mathrm{dR}}\otimes_L\mathrm{LL}(M), \mathrm{LL}(M))\cong M_{\mathrm{dR}}^*\to\mathrm{Ext}^1_{G, \psi}(\mathscr{O}[M]^*, \mathrm{LL}(M)),$$
	qui sont tous définis par le même push-out. En particulier, $\delta$ est un isomorphisme. Par conséquent, d'après le Lemme \ref{2.3.1} (ii), on a $$\mathrm{Hom}_G(\Omega^1[M]^*, \mathrm{LL}(M))=\mathrm{Hom}_G(\mathscr{O}[M]^*, \mathrm{LL}(M))=0,$$
	ce que l'on voulait.
	
	(iii) En appliquant le foncteur $\mathrm{Hom}_G(\mathrm{LL}(M), -)$ à la première ligne du diagramme (\ref{equation4}), on a une suite exacte
	$$0\to\mathrm{Hom}_G(\mathrm{LL}(M), M_{\mathrm{dR}}\otimes_L\mathrm{LL}(M))\to\mathrm{Hom}_G(\mathrm{LL}(M), \Omega^1[M]^*)\to\mathrm{Hom}_G(\mathrm{LL}(M), \mathscr{O}[M]^*).$$
	Il suit du Lemme \ref{2.3.1} (i) que cette suite induit une suite exacte
	$$0\to M_{\mathrm{dR}}\to\mathrm{Hom}_G(\mathrm{LL}(M), \Omega^1[M]^*)\to0,$$
	d'où le résultat.
	
	(iv) En appliquant le foncteur $\mathrm{Hom}_G(-, \mathrm{LL}(M))$ à la suite du Corollaire \ref{2.1.3}, on obtient une suite exacte
	$$0\to\mathrm{Hom}_G(\Pi_{M, \mathscr{L}}^{\mathrm{an}}, \mathrm{LL}(M))\to\mathrm{Hom}_G(\Omega^1[M]^*, \mathrm{LL}(M))\to\mathrm{Hom}_G(\mathscr{L}\otimes_L\mathrm{LL}(M), \mathrm{LL}(M)).$$
	Alors le résultat découle de (ii). 
	
	(v) La preuve est similaire à celle de l'assertion (i). On a 
	$$\mathrm{Hom}_G(\Pi_{M, \mathscr{L}, j}^{\mathrm{an}}, \Pi_{M, \mathscr{L}, n}^{\mathrm{an}})=\mathrm{Hom}_G(\Pi_{M, \mathscr{L}, j}^{\mathrm{an}}, \Pi_{M, \mathscr{L}, n})=\mathrm{Hom}_G(\widehat{\Pi_{M, \mathscr{L}, j}^{\mathrm{an}}}, \Pi_{M, \mathscr{L}, n})=\mathrm{Hom}_G(\Pi_{M, \mathscr{L}, j}, \Pi_{M, \mathscr{L}, n}),$$
	alors le résultat suit de la Proposition \ref{finalbanach}. Cela permet de conclure.
\end{proof}
\begin{corollary}\label{2.3.5} On a 
	
	(i) $\mathrm{Ext}^1_{G, \psi}(\Omega^1[M]^*, \mathrm{LL}(M))=0$.
	
	(ii) $\mathrm{Ext}^1_{G, \psi}(\Pi_{M, \mathscr{L}}^{\mathrm{an}}, \mathrm{LL}(M))=L$.
\end{corollary}
\begin{proof}[Preuve] (i) En appliquant le foncteur $\mathrm{Hom}_G(-, \mathrm{LL}(M))$ à la première ligne du diagramme (\ref{equation4}), on obtient une suite exacte 
	\begin{align*}\mathrm{Hom}_G(\Omega^1[M]^*, \mathrm{LL}(M))&\to\mathrm{Hom}_G(M_{\mathrm{dR}}\otimes_L\mathrm{LL}(M), \mathrm{LL}(M))\to\mathrm{Ext}^1_{G, \psi}(\mathscr{O}[M]^*, \mathrm{LL}(M))\\&\to\mathrm{Ext}_{G, \psi}^1(\Omega^1[M]^*, \mathrm{LL}(M))\to\mathrm{Ext}_{G, \psi}^1(M_{\mathrm{dR}}\otimes_L\mathrm{LL}(M), \mathrm{LL}(M)).
	\end{align*}
	D'après le Lemme \ref{2.3.2}, le Théorème \ref{2.3.3} et le Lemme \ref{2.3.4} (ii), $\mathrm{Hom}_G(\Omega^1[M]^*, \mathrm{LL}(M))$ et $\mathrm{Ext}_{G, \psi}^1(M_{\mathrm{dR}}\otimes_L\mathrm{LL}(M), \mathrm{LL}(M))$ sont nuls, $\mathrm{Hom}_G(M_{\mathrm{dR}}\otimes_L\mathrm{LL}(M), \mathrm{LL}(M))$ et $\mathrm{Ext}^1_{G, \psi}(\mathscr{O}[M]^*, \mathrm{LL}(M))$ sont de dimension $2$, alors le résultat se déduit en comparant les dimensions.
	
	(ii) La démonstration est tout à fait analogue à celle de (i). Plus précisément, en appliquant le foncteur $\mathrm{Hom}_G(-, \mathrm{LL}(M))$ à la deuxième ligne du diagramme (\ref{equation4}), on obtient une suite exacte
	\begin{align*}\mathrm{Hom}_G(\Pi_{M, \mathscr{L}}^\mathrm{an}, \mathrm{LL}(M))&\to\mathrm{End}_G(\mathrm{LL}(M))\to\mathrm{Ext}^1_{G, \psi}(\mathscr{O}[M]^*, \mathrm{LL}(M))\\ &\to\mathrm{Ext}^1_{G, \psi}(\Pi_{M, \mathscr{L}}^\mathrm{an}, \mathrm{LL}(M))\to\mathrm{Ext}^1_{G, \psi}(\mathrm{LL}(M), \mathrm{LL}(M)).
	\end{align*}
	D'après le Lemme \ref{2.3.2}, le Théorème \ref{2.3.3} et le Lemme \ref{2.3.4} (iv), les espaces $\mathrm{Hom}_G(\Pi_{M, \mathscr{L}}^\mathrm{an}, \mathrm{LL}(M))$ et $\mathrm{Ext}^1_{G, \psi}(\mathrm{LL}(M), \mathrm{LL}(M))$ sont nuls, $\mathrm{Ext}^1_{G, \psi}(\mathscr{O}[M]^*, \mathrm{LL}(M))$ est de dimension $2$. Comme $\mathrm{End}_G(\mathrm{LL}(M))$ est de dimension $1$, le résultat se déduit en comparant les dimensions.
\end{proof}
\begin{proposition}\label{2.3.6} On a $$\mathrm{Hom}_G(\Pi_{M, \mathscr{L}}^{\mathrm{an}}, \Omega^1[M]^*)=0.$$
\end{proposition}
\begin{proof}[Preuve] D'après (\ref{equation4}), on a deux suites exactes
	$$0\to\mathrm{LL}(M)\to\Pi_{M, \mathscr{L}}^{\mathrm{an}}\to\mathscr{O}[M]^*\to0,
	$$
	$$0\to\mathrm{LL}(M)\otimes_L M_{\mathrm{dR}}\to\Omega^1[M]^*\to\mathscr{O}[M]^*\to0.
	$$
	Soit $0\neq f\in\mathrm{Hom}_G(\Pi_{M, \mathscr{L}}^{\mathrm{an}}, \Omega^1[M]^*)$. Comme
	$\mathrm{Hom}_G(\mathrm{LL}(M), \mathscr{O}[M]^*)=0$ (Lemme \ref{2.3.1} (i)),  on a un diagramme commutatif à lignes exactes
	$$ 
	\xymatrix@R=5mm@C=4mm{
		0\ar[r]& \mathrm{LL}(M)\ar[d]_i \ar[r]& \Pi_{M, \mathscr{L}}^{\mathrm{an}}\ar[r]\ar[d]_f&\mathscr{O}[M]^*\ar[r]\ar[d]_j&0\\
		0\ar[r]&\mathrm{LL}(M)\otimes_L M_{\mathrm{dR}}\ar[r]& \Omega^1[M]^*\ar[r]&\mathscr{O}[M]^*\ar[r]&0.
	}$$
	
	(i) Supposons que $i=0$ et $j=0$. Il découle du lemme du serpent que l'on a la suite exacte suivante
	$$0\to\mathrm{LL}(M)\to\mathrm{Ker}f\to\mathscr{O}[M]^*\to\mathrm{LL}(M)\otimes_L M_{\mathrm{dR}}\to\mathrm{Coker}f\to\mathscr{O}[M]^*\to0.$$
	Comme $\mathrm{Hom}_G(\mathscr{O}[M]^*, \mathrm{LL}(M))=0$ (Lemme \ref{2.3.1} (ii)), on a $\mathrm{Ker}f=\Pi_{M, \mathscr{L}}^{\mathrm{an}}$, donc $f=0$, ce qui contredit notre hypothèse.
	
	(ii) Supposons que $i=0$ et $j\neq0$. Comme $\mathscr{O}[M]^*$ est topologiquement irréductible, $j$ est un isomorphisme. Il découle du lemme du serpent que l'on a la suite exacte suivante
	$$0\to\mathrm{LL}(M)\to\mathrm{Ker}f\to0\to\mathrm{LL}(M)\otimes_L M_{\mathrm{dR}}\to\mathrm{Coker}f\to0\to0.$$
	Donc on a $\mathrm{Coker}f=\mathrm{LL}(M)\otimes_L M_{\mathrm{dR}}$, ce qui contredit le Lemme \ref{2.3.4} (ii).
	
	(iii) Supposons que $i\neq0$ et $j=0$. Il découle du lemme du serpent que l'on a la suite exacte suivante
	$$0\to0\to\mathrm{Ker}f\to\mathscr{O}[M]^*\to\mathrm{LL}(M)\to\mathrm{Coker}f\to\mathscr{O}[M]^*\to0.$$
	Comme $\mathrm{Hom}_G(\mathscr{O}[M]^*, \mathrm{LL}(M))=0$ (Lemme \ref{2.3.1} (ii)), on a $\mathrm{Ker}f=\mathscr{O}[M]^*$, ce qui contredit le Lemme \ref{2.3.1} (iii).
	
	(iv) Supposons que $i\neq0$ et $j\neq0$. Il découle du lemme du serpent que l'on a la suite exacte suivante
	$$0\to0\to\mathrm{Ker}f\to0\to\mathrm{LL}(M)\to\mathrm{Coker}f\to0\to0.$$
	Donc on a $\mathrm{Coker}f\cong\mathrm{LL}(M)$, ce qui contredit le Lemme \ref{2.2.4} (ii). Le résultat s'en déduit.
\end{proof}
Le résultat ci-dessus implique le résultat suivant
\begin{corollary}\label{2.3.10} La suite exacte du Corollaire \ref{2.1.3} est non scindée.
\end{corollary}
\begin{proposition}\label{2.3.7} On a $$\mathrm{Hom}_G(\Omega^1[M]^*, \Pi_{M, \mathscr{L}}^{\mathrm{an}})=L.$$
\end{proposition}
\begin{proof}[Preuve] En appliquant le foncteur $\mathrm{Hom}_G(\Omega^1[M]^*,-)$ à la suite 
	$$0\to\mathscr{L}\otimes_L\mathrm{LL}(M)\to\Pi_{M, \mathscr{L}}^{\mathrm{an}}\to \mathscr{O}[M]^*\to0,$$
	on obtient une suite exacte 
	\begin{align*}0&\to\mathrm{Hom}_G( \Omega^1[M]^*, \mathscr{L}\otimes_L\mathrm{LL}(M))\to\mathrm{Hom}_G(\Omega^1[M]^*,\Pi_{M, \mathscr{L}}^{\mathrm{an}} )\to\mathrm{Hom}_G(\Omega^1[M]^*, \mathscr{O}[M]^*)\\ &\to  \mathrm{Ext}^1_{G, \psi}( \Omega^1[M]^*, \mathscr{L}\otimes_L\mathrm{LL}(M)).
	\end{align*}
	Alors le résultat se déduit du Lemme \ref{2.3.1} (v), du Lemme \ref{2.3.4} (ii) et du Corollaire \ref{2.3.5} (i).
\end{proof}
\begin{proposition}\label{2.3.8} On a $$\mathrm{End}_{G}(\Omega^1[M]^*)=L.$$
\end{proposition}
\begin{proof}[Preuve] En appliquant le foncteur $\mathrm{Hom}_G(\Omega^1[M]^*, -)$ à la suite 
	$$0\to\mathscr{L}\otimes_L\mathrm{LL}(M)\to \Omega^1[M]^*\to \Pi_{M, \mathscr{L}}^{\mathrm{an}}\to0,$$
	on obtient une suite exacte 
	$$0\to\mathrm{Hom}_G(\Omega^1[M]^*, \mathrm{LL}(M) )\to\mathrm{End}_G(\Omega^1[M]^*)\to\mathrm{Hom}_G(\Omega^1[M]^*, \Pi_{M, \mathscr{L}}^{\mathrm{an}} ).$$
	Le résultat suit du Lemme \ref{2.3.4} (ii) et de la Proposition \ref{2.3.7} car $\mathrm{End}_G(\Omega^1[M]^*)$ est de dimension $\geq1$. 
\end{proof}
 \section{Représentations de Banach de $G$}
\subsection{Finitude résiduelle pour les représentations de Banach de $G$}
Rappelons qu'une représentation $\pi$ de $G$ sur $\kappa_L$ est dite de présentation finie si l'on a une suite exacte
$$\mathrm{ind}_{K_1}^G\sigma_1\to\mathrm{ind}_{K_2}^G\sigma_2\to\pi\to0,$$
où $\sigma_i$ est une représentation lisse de dimension finie d'un sous-groupe compact $K_i$ de $G$ pour $i=1, 2$.

D'après \cite{shotton2020the}, la définition ci-dessus est équivalente à ce que $\pi$ est de présentation finie en tant que $\kappa_L[G]$-module. 
\begin{theorem}[{\cite[Corollary 1.2]{timmins2023coherence}}]\label{3.1.1} Soit $K$ une extension finie de $\mathbb{Q}_p$, alors l'algèbre d'Iwasawa augmenté de $\mathrm{GL}_n(K)$ est cohérent si et seulement si $n\leq2$.
\end{theorem}
\begin{theorem}[{\cite[Theorem 7.21]{timmins2023coherence}}]\label{3.1.2} Soit $G$ un groupe de Lie $p$-adique. Supposons que l'algèbre d'Iwasawa augmenté est cohérent, alors la catégorie des représentations lisses de présentation finie modulo $p$ de $G$ est abélienne.
\end{theorem}
En combinant le Théorème \ref{3.1.1} et le Théorème \ref{3.1.2}, on déduit le résultat suivant
\begin{corollary}\label{3.1.3} La catégorie des représentations lisses de présentation finie de $G$ est abélienne.
\end{corollary}
Maintenant on passe aux représentations de Banach de $G$.
\begin{definition} Une représentation de Banach $\Pi$ de $G$ est dite résiduellement de type fini si la réduction de la boule unité $\Pi^+$ de $\Pi$ est de type fini.
\end{definition}
\begin{definition} Une représentation de Banach $\Pi$ de $G$ est dite résiduellement de présentation finie si la réduction de la boule unité $\Pi^+$ de $\Pi$ est de présentation finie.
\end{definition}
On note $\mathrm{Ban}_{\mathrm{f.p.}}(G)$ la catégorie dont les objets sont les $L$-représentations de Banach résiduellement de présentation finie et dont les morphismes sont les morphismes $G$-équivariants. La proposition ci-dessous implique que tout morphisme dans cette catégorie est strict.
\begin{proposition} Soit $f: \Pi_1\to\Pi_2$ un morphisme de $L$-représentations unitaires de $G$. Supposons que $\Pi_2^+/\pi_L$ est de type fini où $\Pi_2^+$ est la boule unité de $\Pi_2$, alors $f$ est strict.
\end{proposition}
\begin{proof}[Preuve] Il s'agit de prouver que $\mathrm{Im}(f)$ est fermée dans $\Pi_2$. On note $T$ l'adhérence de $\mathrm{Im}(f)$ et $T^+:=T\cap\Pi_2^+$. On a $T^+/\pi_L\subseteq\Pi_2^+/\pi_L$, et donc $T^+/\pi_L$ est de type fini grâce au \cite[Theorem 2.4.1]{deg2023localization}. Le Lemme \ref{2.2.2} implique que $f(\Pi_1)=T$, ce qui permet de conclure.
\end{proof}
\begin{proposition}\label{3.1.6} La catégorie $\mathrm{Ban}_{\mathrm{f.p.}}(G)$ est abélienne.
\end{proposition}
\begin{proof}[Preuve] Supposons que $\Pi_1\to\Pi_2$ est une flèche dans $\mathrm{Ban}_{\mathrm{f.p.}}(G)$, alors il suffit de montrer que le noyau et le conoyau appartiennent à $\mathrm{Ban}_{\mathrm{f.p.}}(G)$. On note $\Pi_3$ le noyau de $\Pi_1\to\Pi_2$ et $\Pi_1^+$ la boule unité de $\Pi_1$. On note $\Pi_3^+:=\Pi_3\cap\Pi_1^+$ et $\Pi_2^+$ l'image de $\Pi_1^+$ dans $\Pi_2$, alors on a une suite exacte $0\to\Pi_3^+/\pi_L\to\Pi_1^+/\pi_L\to\Pi_2^+/\pi_L$, donc $\Pi_3^+/\pi_L$ est de présentation finie. La preuve pour le conoyau est similaire. Cela permet de conclure.
\end{proof}
\subsection{Vecteurs lisses et localement analytiques de $\widehat{\mathrm{LL}(M)}$}
Dans cette section, on montre quelques propriétés de $\widehat{\mathrm{LL}(M)}$. Notons d'abord le résultat suivant d'Emerton.
\begin{proposition}[{\cite[5.1.18]{emerton2006a}}]\label{3.2.1} Soit $\pi$ une représentation supercuspidale irréductible de $G$, alors le complété unitaire universel $\hat{\pi}$ de $\pi$ n'est pas admissible en tant que représentation de Banach.
\end{proposition}
\begin{proposition}[{\cite[Proposition 3.1]{cdn2023correspondance}}]\label{3.2.2} On a $$\widehat{\mathrm{LL}(M)}^{\mathrm{lisse}}=\mathrm{LL}(M).$$
\end{proposition}
Une preuve similaire nous donne le résultat suivant
\begin{proposition}\label{3.2.3} On a $$\widehat{\mathrm{LL}(M)}^{\mathrm{an}}=\mathrm{LL}(M).$$
\end{proposition}
\begin{proof}[Preuve] Il suffit de montrer le même résultat pour $\mathrm{ind}_{KZ}^G\sigma_M$. Notons $X_n$ la double classe $KZ\begin{psmallmatrix}
		p^n&0\\
		0&1
	\end{psmallmatrix}KZ$ pour tout $n$, alors on a $$\mathrm{ind}_{KZ}^G\sigma_M=\oplus_n\mathrm{ind}_{KZ}^{X_n}\sigma_M,$$ et le complété unitaire universel $\widehat{\mathrm{ind}_{KZ}^G\sigma_M}$ de $\mathrm{ind}_{KZ}^G\sigma_M$ est l'ensemble des $w=\sum_{n\geq0}w_n$, avec $w_n\in\mathrm{ind}_{KZ}^{X_n}\sigma_M$, et $w_n\to0$ quand $n\to\infty$. Notons $R_n(w):=\sum_{i\leq n}w_i$, alors $w=\lim_nR_n(w)$. 
	
	Soit $w\in\widehat{\mathrm{ind}_{KZ}^G\sigma_M}^{\mathrm{an}}$, alors il existe un sous-groupe ouvert compact $G_w$ de $G$, une bijection $\mathbf{c}=(c_0, c_1, c_2, c_3): G_w\to\z_p^4$ et une suite $(w_\mathbf{k})_{\mathbf{k}\in\mathbb{N}^4}$ dans $\widehat{\mathrm{ind}_{KZ}^G\sigma_M}$ qui tend vers $0$, tels que 
	$$g\cdot w=\sum_{\mathbf{k}\in\mathbb{N}^4}\mathbf{c}(g)^{\mathbf{k}}w_{\mathbf{k}}.$$
	Soit $n\in\mathbb{N}$. Comme $\mathrm{ind}_{KZ}^{X_n}\sigma_M$ est lisse, il existe un sous-groupe ouvert $G^n\subseteq G_w$ tel que
	$$g\cdot R_n(w)=R_n(w)$$ 
	pour tout $g\in G^n$, ce qui implique que la fonction $\sum_{\mathbf{k}\in\mathbb{N}^4}\mathbf{c}(g)^{\mathbf{k}}R_n(w_{\mathbf{k}})$ est constante sur $G^n$. Comme l'application $\mathbf{c}$ est bijective, il s'ensuit que $R_n(w_{\mathbf{k}})=0$ pour tout $\mathbf{k}\neq0$. On en déduit que $w_{\mathbf{k}}=0$ pour tout $\mathbf{k}\neq0$. Donc on a $g\cdot w=w_{\mathbf{0}}$ pour tout $g\in G_w$, en particulier, on a $w=w_{\mathbf{0}}$ et $g\cdot w=w$ pour tout $g\in G_w$. Cela implique que $w$ est lisse. Or, d'après la Proposition \ref{3.2.2}, on a $$\widehat{\mathrm{ind}_{KZ}^G\sigma_M}^{\mathrm{lisse}}=\mathrm{ind}_{KZ}^G\sigma_M,$$ 
	donc $w\in\mathrm{ind}_{KZ}^G\sigma_M$, ce qui implique que $$\widehat{\mathrm{LL}(M)}^{\mathrm{an}}=\mathrm{LL}(M).$$
	Cela permet de conclure.
\end{proof}	
\begin{remark} Notons que le sous-ensemble des vecteurs localement analytiques dans une représentation de Banach admissible de $G$ est dense. Or, en utilisant la Proposition \ref{3.2.3}, on peut montrer que la représentation  $N_{\mathscr{L}, j}$ définie dans la section \ref{section2.2} est une représentation de Banach telle que $N_{\mathscr{L}, j}^{\mathrm{an}}=0$, ce qui implique que la représentation $N_{\mathscr{L}, j}$ est non admissible. En effet, soit $0\neq v\in N_{\mathscr{L}, j}^{\mathrm{an}}$, alors $v\in\widehat{\mathrm{LL}(M)}^{\mathrm{an}}=\mathrm{LL}(M)$. Comme $\mathrm{LL}(M)$ est irréductible, on a $\mathrm{LL}(M)\subseteq N_{\mathscr{L}, j}$. Cependant, $\mathrm{LL}(M)$ est dense dans $\widehat{\mathrm{LL}(M)}$ et le sous-espace $N_{\mathscr{L}, j}$ est fermé dans $\widehat{\mathrm{LL}(M)}$, on a donc $N_{\mathscr{L}, j}=\widehat{\mathrm{LL}(M)}$, une contradiction. On en conclut que $N_{\mathscr{L}, j}^{\mathrm{an}}=0$. Cela donne une autre preuve de la non admissibilité de $\widehat{\mathrm{LL}(M)}$.
\end{remark}
\subsection{Complétions $\mathfrak{B}$-adiques des représentations lisses}\label{3.3}
La référence pour cette section est l'article \cite{cdn2023correspondance}. Soit $\pi$ une représentation lisse de type fini de $G$ à caractère central $\psi$ en caractéristique zéro. Soit $\bar{\rho}: \mathcal{G}_{\mathbb{Q}_p}\to\mathrm{GL}_2(\kappa_L)$ une représentation semi-simple, alors $\bar{\rho}$ détermine un bloc $\mathfrak{B}$. Soit $\Lambda$ un réseau de $\pi$ stable par $G$, alors on peut définir le complété $\mathfrak{B}$-adique $\Lambda_\mathfrak{B}$ de $\Lambda$ comme la limite projective des quotients de longueur finie de $\Lambda$ dont toutes les facteurs de Jordan-Hölder appartiennent à $\mathfrak{B}$, et on pose $\pi_{\mathfrak{B}}:=L\otimes_{\mathscr{O}_L}\Lambda_\mathfrak{B}$. La définition de $\pi_{\mathfrak{B}}$ est indépendante du choix de $\Lambda$.
\begin{lemma}[{\cite[Corollary 2.3]{cdn2023correspondance}}]\label{exact} Soit $\mathcal{C}^\psi$ la catégorie des représentations lisses de type fini de $G$ à caractère central $\psi$. Le foncteur de passage à la complétion $\mathfrak{B}$-adique $X\mapsto X_\mathfrak{B}$ de $\mathcal{C}^\psi$ vers la catégorie des $G$-modules pro-discrets est exact.
\end{lemma}
On note $U_{M, \mathfrak{B}}$ l'ensemble des droites $\mathscr{L}\in\mathbf{P}(M_{\mathrm{dR}})$ telles que $\Pi_{M, \mathscr{L}}\in\mathrm{Ban}_{G, \psi}^{\mathrm{adm}}(L)_\mathfrak{B}$. On note $R_{M, \mathfrak{B}}:=\mathrm{End}_G(\mathrm{LL}(M)_\mathfrak{B})$, $X_{M, \mathfrak{B}}:=\mathrm{Spec}R_{M, \mathfrak{B}}$ et $\rho_{M, \mathfrak{B}}:=\mathbf{V}(\mathrm{LL}(M)_\mathfrak{B})$ où $\mathbf{V}$ est le foncteur de Langlands $p$-adique de Colmez. On note $R_{\mathfrak{B}}^{\mathrm{ps}, \delta_M}$ l'anneau des déformations universelles de déterminant $\delta_M\epsilon$ du pseudo-caractère $\mathrm{Tr}\circ\rho_{\mathfrak{B}}$ où $\rho_{\mathfrak{B}}$ est la $\kappa_L$-représentation de $\mathcal{G}_{\mathbb{Q}_p}$ semi-simple de dimension $2$ correspondante à $\mathfrak{B}$. Alors $R_{M, \mathfrak{B}}$ est le quotient de $R_{\mathfrak{B}}^{\mathrm{ps}, \delta_M}[\tfrac{1}{p}]$ paramétrant les représentations de type $M$ (i.e., potentiellement semi-stables, à poids $0$ et $1$, dont le $D_{\mathrm{pst}}$ est isomorphe à $M$) et $\rho_{M, \mathfrak{B}}$ est la $R_{M, \mathfrak{B}}$-représentation de $\mathcal{G}_{\mathbb{Q}_p}$ de dimension $2$ interpolant les représentations de type $M$. Si $x\in X_{M, \mathfrak{B}}$, on note $\mathfrak{m}_x$ l'idéal maximal de $R_{M, \mathfrak{B}}$ qui lui est associé et $L_x$ le corps résiduel $R_{M, \mathfrak{B}}/\mathfrak{m}_x$. D'après \cite[Théorème 0.1]{cdn2023correspondance}, $\rho_{M, \mathfrak{B}}$ est libre de rang $2$ sur $R_{M, \mathfrak{B}}$ et l'application de spécialisation $x\mapsto \rho_x$ induit, pour toute extension finie $L'$ de $L$, une bijection entre $X_{M, \mathfrak{B}}(L')$ et l'ensemble des $L'$-représentations de $\mathcal{G}_{\mathbb{Q}_p}$ de réduction $\bar{\rho}$ et de type $M$.

\subsection{Résultats sur les sous-quotients de $\widehat{\mathrm{LL}(M)}$}  
\begin{lemma}\label{3.4.0} Soient $\{\mathscr{L}_i\}_i$ une famille finie de droites distinctes et $\Pi$ une sous-représentation fermée de $\oplus_{i=1}^n\Pi_{M, \mathscr{L}_i, j_i}$ telle que, pour tout $i$, la composée $\Pi\hookrightarrow\oplus_{i=1}^n\Pi_{M, \mathscr{L}_i, j_i}\twoheadrightarrow\Pi_{M, \mathscr{L}_i, j_i}$ est surjective, alors $\Pi\cong\oplus_{i=1}^n\Pi_{M, \mathscr{L}_i, j_i}$.
\end{lemma}
\begin{proof}[Preuve] La démonstration se fait par récurrence. Le cas $n=1$ est trivial. Supposons que le résultat est vrai pour $n$. Pour chaque $1\leq i\leq n+1$, on considère la projection $$\mathrm{Proj}_i: \Pi\to\Pi_{M, \mathscr{L}_1, j_1}\oplus\cdots\oplus\widehat{\Pi_{M, \mathscr{L}_i, j_i}}\oplus\cdots\oplus\Pi_{M, \mathscr{L}_{n+1}, j_{n+1}}$$
	où le chapeau désigne que l'on omet le terme $\Pi_{M, \mathscr{L}_i, j_i}$. En utilisant l'hypothèse de récurrence à l'image de $\mathrm{Proj}_i$, on déduit que la projection $\mathrm{Proj}_i$ est surjective, ce qui induit une suite exacte courte
	$$0\to\Pi\cap(0\oplus\cdots\oplus\Pi_{M, \mathscr{L}_i, j_i}\oplus\cdots\oplus0)\to\Pi\to\Pi_{M, \mathscr{L}_1, j_1}\oplus\cdots\oplus\widehat{\Pi_{M, \mathscr{L}_i, j_i}}\oplus\cdots\oplus\Pi_{M, \mathscr{L}_{n+1}, j_{n+1}}\to0$$
	pour tout $i$. 	La représentation $\Pi\cap(0\oplus\cdots\oplus\Pi_{M, \mathscr{L}_i, j_i}\oplus\cdots\oplus0)$ est isomorphe à une sous-représentation fermée de $\Pi_{M, \mathscr{L}_i, j_i}$. Il suit de ce qui précède la Proposition \ref{finalbanach} que $\Pi\cap(0\oplus\cdots\oplus\Pi_{M, \mathscr{L}_i, j_i}\oplus\cdots\oplus0)\cong\Pi_{M, \mathscr{L}_i, t_i}$ pour un certain entier positif $t_i\leq j_i$. On déduit du Théorème \ref{1.1.4} que l'on a un isomorphisme
	$$\Pi\cong\Pi_{M, \mathscr{L}_i, t_i}\oplus(\Pi_{M, \mathscr{L}_1, j_1}\oplus\cdots\oplus\widehat{\Pi_{M, \mathscr{L}_i, j_i}}\oplus\cdots\oplus\Pi_{M, \mathscr{L}_{n+1}, j_{n+1}}).$$
	
	Considérons la surjection $\Pi\twoheadrightarrow\Pi_{M, \mathscr{L}_i, j_i}$, il suit de la Proposition \ref{2.1.1} que ce morphisme se factorise par $\Pi_{M, \mathscr{L}_i, t_i}\to\Pi_{M, \mathscr{L}_i, j_i}$. On en déduit que $t_i=j_i$. Il en résulte un diagramme commutatif à lignes exactes dans la catégorie abélienne des représentations de Banach admissibles
	$$ 
	\xymatrix@R=5mm@C=4mm{
		0\ar[r]&\Pi_{M, \mathscr{L}_i, j_i}\ar@{=}[d] \ar[r]& \Pi\ar[r]\ar@{^{(}->}[d]&\Pi_{M, \mathscr{L}_1, j_1}\oplus\cdots\oplus\widehat{\Pi_{M, \mathscr{L}_i, j_i}}\oplus\cdots\oplus\Pi_{M, \mathscr{L}_{n+1}, j_{n+1}}\ar[r]\ar@{=}[d] &0\\
		0\ar[r]&\Pi_{M, \mathscr{L}_i, j_i}\ar[r]& \oplus_{i=1}^n\Pi_{M, \mathscr{L}_i, j_i}\ar[r]&\Pi_{M, \mathscr{L}_1, j_1}\oplus\cdots\oplus\widehat{\Pi_{M, \mathscr{L}_i, j_i}}\oplus\cdots\oplus\Pi_{M, \mathscr{L}_{n+1}, j_{n+1}}\ar[r]&0.
	}$$
	Ainsi, on déduit que $\Pi=\oplus_{i=1}^n\Pi_{M, \mathscr{L}_i, j_i}$, ce que l'on voulait.
\end{proof}
Le corollaire suivant implique que toute somme directe finie de $\Pi_{M, \mathscr{L}, j}$ pour $\mathscr{L}$ différentes est un quotient de $\widehat{\mathrm{LL}(M)}$. On va montrer, à la fin de cette thèse, ces représentations décrivent exhaustivement tous les quotients propres de $\widehat{\mathrm{LL}(M)}$.
\begin{corollary}\label{3.3.1} Soient $\{\mathscr{L}_i\}_{i\in I}$ une famille finie de droites distinctes et $\{j_i\}_{i\in I}$ une famille d'entiers positifs, alors le morphisme
	$$\widehat{\mathrm{LL}(M)}/\cap N_{\mathscr{L}_i, j_i}\to\oplus_{i=1}^n\Pi_{M, \mathscr{L}_i, j_i}$$
	induit par les surjections $\widehat{\mathrm{LL}(M)}\to\Pi_{M, \mathscr{L}_i, j_i}$ est un isomorphisme.
\end{corollary}
\begin{proof}[Preuve] Compte tenu du Lemme \ref{2.2.2}, il suffit de montrer que le morphisme $\mathrm{LL}(M)\to\oplus_{i=1}^n\Pi_{M, \mathscr{L}_i, j_i}$ est d'image dense. Or, l'adhérence $T$ du image de ce morphisme satisfait la condition que toute projection $T\to\Pi_{M, \mathscr{L}_i}$ est non nulle. Il découle du Lemme \ref{3.4.0} que $T=\oplus_{i=1}^n\Pi_{M, \mathscr{L}_i}$, d'où le résultat.
\end{proof}	
Le lemme ci-dessous implique que la représentation $\Pi_{M, \mathscr{L}}^{\oplus n}$ pour $n\geq2$ ne peut pas être un quotient de $\widehat{\mathrm{LL}(M)}$.
\begin{lemma}\label{3.4.2} Soient $\mathscr{L}\in\mathbf{P}^1$ et $\{s_j\}_{j\in J}$ une famille d'entiers positifs où $|J|\geq2$, alors il n'existe pas de surjection $\widehat{\mathrm{LL}(M)}\twoheadrightarrow\oplus_{j\in J}\Pi_{M, \mathscr{L}, j}$.
\end{lemma}
\begin{proof}[Preuve] Supposons qu'il existe une surjection $\widehat{\mathrm{LL}(M)}\twoheadrightarrow\oplus_{j\in J}\Pi_{M, \mathscr{L}, j}$. On note $T$ le noyau, alors on a la suite exacte suivante
	$$0\to T\to\widehat{\mathrm{LL}(M)}\to\oplus_{j\in J}\Pi_{M, \mathscr{L}, j}\to0.$$
	En appliquant le foncteur $\mathrm{Hom}_G(-, \Pi_{M, \mathscr{L}})$, on obtient une suite exacte
	$$0\to\mathrm{Hom}_G(\oplus_{j\in J}\Pi_{M, \mathscr{L}, j}, \Pi_{M, \mathscr{L}})\to\mathrm{Hom}_G(\widehat{\mathrm{LL}(M)}, \Pi_{M, \mathscr{L}})\to\mathrm{Hom}_G(T, \Pi_{M, \mathscr{L}}).$$
	Comme $$\mathrm{Hom}_G(\widehat{\mathrm{LL}(M)}, \Pi_{M, \mathscr{L}})=\mathrm{Hom}_G(\mathrm{LL}(M), \Pi_{M, \mathscr{L}})=\mathrm{Hom}_G(\mathrm{LL}(M), \Pi_{M, \mathscr{L}}^{\mathrm{lisse}})=\mathrm{End}_G(\mathrm{LL}(M))$$
	est de dimension $1$, on obtient une contradiction. Cela permet de conclure.
\end{proof}
\begin{lemma}\label{3.3.3} Soit $\Pi$ une sous-représentation de $\widehat{\mathrm{LL}(M)}$, alors on a $\mathrm{Hom}_G(\widehat{\mathrm{LL}(M)}/\Pi, \Pi_{M, \mathscr{L}})\neq0$ si et seulement si $\Pi\subseteq N_{\mathscr{L}, 1}$.
\end{lemma}
\begin{proof}[Preuve] Si $f\in\mathrm{Hom}_G(\widehat{\mathrm{LL}(M)}/\Pi, \Pi_{M, \mathscr{L}})$ est un morphisme non nul, alors la composée de $f$ et du morphisme naturel $\widehat{\mathrm{LL}(M)}\to\widehat{\mathrm{LL}(M)}/\Pi$ est un morphisme $\widehat{\mathrm{LL}(M)}\to\Pi_{M, \mathscr{L}}$, ce qui est induit par l'inclusion $\mathrm{LL}(M)\to\Pi_{M, \mathscr{L}}$. Cela permet de conclure.
\end{proof}

\begin{lemma}\label{3.3.4} Il existe un nombre fini de blocs $\{\mathfrak{B}_i\}_i$ tels que l'on ait une injection $$\widehat{\mathrm{LL}(M)}\hookrightarrow\prod_i\mathrm{LL}(M)_{\mathfrak{B}_i}.$$
\end{lemma}
\begin{proof}[Preuve] On se ramène au même énoncé pour $\mathrm{ind}_{KZ}^G\sigma_M$. On prends un réseau $\sigma_M^0$ de $\sigma_M$. La réduction de $\sigma_M^0$ modulo $p$ est une extension d'un nombre fini de poids de Serre que l'on note $\{\sigma_i\}_i$. Pour tout $i$, on a $\mathrm{End}_G(\sigma_i)=\kappa_L[T_i]$ et on note $\mathfrak{B}_i$ le bloc correspondant à la représentation $(\mathrm{ind}_{KZ}^G\sigma_i)/P_i(T_i)$, où $P_i(T_i)$ est un polynôme irréductible dans $\kappa_L[T_i]$ et $T_i$ est l'opérateur de Barthel-Livné. Pour chaque bloc $\mathfrak{B}_i$, on a une flèche naturelle $f_{\mathfrak{B}_i}:\mathrm{LL}(M)\to\mathrm{LL}(M)_{\mathfrak{B}_i}$. Le réseau $\mathrm{ind}_{KZ}^G\sigma_M^0$ est minimal  à homothétie près, donc le morphisme naturel $\mathrm{ind}_{KZ}^G\sigma_M^0\to(\mathrm{ind}_{KZ}^G\sigma_M^0)_{\mathfrak{B}_i}$ induit un morphisme $\widehat{\mathrm{LL}(M)}\to\mathrm{LL}(M)_{\mathfrak{B}_i}$.
	
	Il nous reste à montrer que la flèche $(\mathrm{ind}_{KZ}^G\sigma_M^0)/\pi_L\to\prod\limits_i((\mathrm{ind}_{KZ}^G\sigma_M^0)/\pi_L)_{\mathfrak{B}_i}$ est injective. Comme le foncteur de complété $\mathfrak{B}$-adique est exact (Lemme \ref{exact}), on est ramené à montrer que la flèche $\mathrm{ind}_{KZ}^G\sigma_i\to(\mathrm{ind}_{KZ}^G\sigma_i)_{\mathfrak{B}_i}$ est injective pour tout $i$. Or, il suit de la \cite[Proposition 2.5]{cdn2023correspondance} (le cas supersingulier) et de la \cite[Proposition 2.6]{cdn2023correspondance} (le cas non supersingulier) que 
	$$(\mathrm{ind}_{KZ}^G\sigma_i)_{\mathfrak{B}_i}=\varprojlim_n\mathrm{ind}_{KZ}^G\sigma_i/P_i(T_i)^n,$$
	donc on a bien une injection $\mathrm{ind}_{KZ}^G\sigma_i\hookrightarrow(\mathrm{ind}_{KZ}^G\sigma_i)_{\mathfrak{B}_i}$ car $\mathrm{ind}_{KZ}^G\sigma_i$ est libre sur $\kappa_L[T_i]$. Cela permet de conclure.
\end{proof}
\begin{lemma}\label{3.3.5} Soit $\mathfrak{B}$ l'un des blocs donnés par le Lemme \ref{3.3.4}. Si $I\subseteq U_{M, \mathfrak{B}}$ est infini, alors on a une injection	$$\mathrm{LL}(M)_\mathfrak{B}\hookrightarrow\prod_{\mathscr{L}\in I}\Pi_{M, \mathscr{L}}.$$
\end{lemma}
\begin{proof}[Preuve] D'après le \cite[Corollaire 5.5]{cdn2023correspondance}, $R_{M, \mathfrak{B}}$ est réduit et de dimension $1$, et $\rho_{M, \mathfrak{B}}$ est libre sur $R_{M, \mathfrak{B}}$. On en déduit une injection de $R_{M, \mathfrak{B}}$-modules $\rho_{M, \mathfrak{B}}\hookrightarrow\prod\limits_{\mathscr{L}\in I}\rho_{M, \mathfrak{B}}/\mathfrak{m}_{\mathscr{L}}$ où $\mathfrak{m}_{\mathscr{L}}$ est l'idéal maximal de $R_{M, \mathfrak{B}}$ correspondant à $\mathscr{L}$. Cela induit, par fonctorialité, une injection de $R_{M, \mathfrak{B}}$-modules
	$$\mathbf{\Pi}(\rho_{M, \mathfrak{B}})\hookrightarrow\prod\limits_{\mathscr{L}\in I}\mathbf{\Pi}(\rho_{M, \mathfrak{B}})/\mathfrak{m}_{\mathscr{L}}.$$
	Comme $\mathbf{\Pi}(\rho_{M, \mathfrak{B}})\cong\mathrm{LL}(M)_\mathfrak{B}$ et $\mathbf{\Pi}(\rho_{M, \mathfrak{B}})/\mathfrak{m}_{\mathscr{L}}\cong\Pi_{M, \mathscr{L}}$, on obtient une injection
	$$\mathrm{LL}(M)_\mathfrak{B}\hookrightarrow\prod_{\mathscr{L}\in I}\Pi_{M, \mathscr{L}}.$$
	Cela permet de conclure.
\end{proof}
\begin{corollary}\label{3.3.6} Fixons une droite $\mathscr{L}$, alors on a une injection $$\widehat{\mathrm{LL}(M)}\hookrightarrow\prod_{\mathscr{L}'\neq\mathscr{L}}\Pi_{M, \mathscr{L}'}.$$
\end{corollary}	 
\begin{proof}[Preuve] Soit $\mathfrak{B}$ un bloc tel que $\mathscr{L}\in U_{M, \mathfrak{B}}$. D'après la preuve du Lemme \ref{3.3.4}, on peut choisir des blocs $\mathfrak{B}_i\neq\mathfrak{B}$ tels que 
	$$\widehat{\mathrm{LL}(M)}\hookrightarrow\prod_{\mathfrak{B}_i}\mathrm{LL}(M)_{\mathfrak{B}_i}.$$
	Le résultat découle du Lemme \ref{3.3.5}.
\end{proof}
\begin{proposition}\label{3.3.7} Toute sous-représentation fermée non nulle $\Pi$ de $\widehat{\mathrm{LL}(M)}$ est résiduellement de longueur infinie.
\end{proposition}
\begin{proof}[Preuve] Notons $\widehat{\mathrm{LL}(M)}^+$ la boule unité de $\widehat{\mathrm{LL}(M)}$ et posons $\Pi^+:=\Pi\cap\widehat{\mathrm{LL}(M)}^+$. Le morphisme naturel $i: \Pi^+/\pi_L\to\widehat{\mathrm{LL}(M)}^+/\pi_L$ est injectif. Comme $\sigma_M^0/\pi_L$ est une extension des $\sigma_i$ pour $1\leq i\leq n$, la représentation $\widehat{\mathrm{LL}(M)}^+/\pi_L$ est une extension des $\mathrm{ind}_{KZ}^G\sigma_i$. On en déduit que $\Pi^+/\pi_L$ est une extension de sous-représentations des $\mathrm{ind}_{KZ}^G\sigma_i$. Il résulte du \cite[Lemme 4.2]{hp2019on} que toutes les sous-$\kappa_L$-représentations de $\mathrm{ind}_{KZ}^G\sigma_i$ sont de longueur infinie, donc $\Pi^+/\pi_L$ est de longueur infinie, ce que l'on voulait.
\end{proof}
\subsection{Vecteurs lisses et localement analytiques de $\widehat{\Omega^1[M]^{*}}$} 
On note $\Omega^1[M]^\mathrm{b}$ le sous-espace de $\Omega^1[M]$ constitué des vecteurs $G$-bornés, alors on a $\widehat{\Omega^1[M]^*}=\Omega^1[M]^{\mathrm{b}, *}$ en vertu du \cite[Lemme 5.3]{cdn2023factorisation}. La restriction $$\Omega^1[M]^*\to\Omega^1[M]^{\mathrm{b}, *}\xrightarrow{\sim}\widehat{\Omega^1[M]^*}$$
nous donne un morphisme $\Omega^1[M]^*\to\widehat{\Omega^1[M]^*}$. La représentation $\Omega^1[M]^*$ est localement analytique, cela induit un morphisme $\Omega^1[M]^*\to(\widehat{\Omega^1[M]^*})^{\mathrm{an}}$. On va montrer que c'est un isomorphisme. Inspiré par \cite{st2003algebras}, on va d'abord montrer que l'on a un isomorphisme $D(G)\hat{\otimes}_{\Lambda(G)}\Omega^1[M]^b  \cong\Omega^1[M]$, voir le Lemme \ref{3.5.2}. Notons que l'on prends le produit tensoriel complété car la représentation $\Omega^1[M]^{b, *}$ n'est pas admissible en tant que représentation de Banach. Il suffit de prouver l'existence d'un isomorphisme $D(G)\hat{\otimes}_{\Lambda(G)}\Omega^{1, b}(\Sigma_n)  \cong\Omega^1(\Sigma_n)$.

D'après le Lemme \ref{1.2.6}, on peut munir $\Omega^{1, b}(\Sigma_n)$, l'ensemble des vecteurs $G$-bornés dans $\Omega^1(\Sigma_n)$, de la topologie induite par $(\varprojlim_i\Omega(\mathfrak{X}_i))[\frac{1}{p}]$. Tout d'abord, on définit le produit tensoriel complété $\Omega^{1, b}(\Sigma_n)\hat{\otimes}_{\Lambda(G)} D(G)$. Le réseau $\Omega^1(\mathfrak{X}_i)$ de $\Omega^1(\mathfrak{X}_i)[\tfrac{1}{p}]$ induit une norme sur $\Omega^1(\mathfrak{X}_i)[\tfrac{1}{p}]$ définie par $$||\omega||:=\inf_{\omega\in \lambda \Omega^1(\mathfrak{X}_i)}|\lambda|.$$
Comme $\Omega^1(\mathfrak{X}_i)[\tfrac{1}{p}]$ et $D_r(G_i)$ sont des espaces de Banach, on peut munir l'espace $D_r(G_i)\otimes_{\Lambda(G_i)} \Omega^1(\mathfrak{X}_i)[\tfrac{1}{p}]$ de la semi-norme projective définie par
$$||z||:=\inf\{\max_k|x_k|\cdot||y_k||\}$$
où l'infimum est pris parmi toutes les expressions $z=\sum_k x_k\otimes y_k$. 

On en obtient un espace normé $$D_r(G_i)\otimes'_{\Lambda(G_i)}\Omega^1(\mathfrak{X}_i)[\tfrac{1}{p}] :=\frac{D_r(G_i)\otimes_{\Lambda(G_i)}\Omega^1(\mathfrak{X}_i)[\tfrac{1}{p}] }{\{z\mid ||z||=0\}}.$$
Et on définit $D_r(G_i)\hat{\otimes}_{\Lambda(G_i)} \Omega^1(\mathfrak{X}_i)[\tfrac{1}{p}]$ comme la complétion de $D_r(G_i)\otimes'_{\Lambda(G_i)}\Omega^1(\mathfrak{X}_i)[\tfrac{1}{p}]$, ce qui fait de $D_r(G_i)\hat{\otimes}_{\Lambda(G_i)}\Omega^1(\mathfrak{X}_i)[\tfrac{1}{p}]$ un Banach. On dispose d'un morphisme naturel $$\alpha: D_r(G_i)\otimes_{\Lambda(G_i)}\Omega^1(\mathfrak{X}_i)[\tfrac{1}{p}]\rightarrow D_r(G_i)\hat{\otimes}_{\Lambda(G_i)}\Omega^1(\mathfrak{X}_i)[\tfrac{1}{p}].$$
D'après la preuve de \cite[Théorème 3.2]{dl2017revetements}, $D_r(G_i)$ agit continûment sur $\Omega^1(\mathfrak{X}_i)$ pour $r$ suffisamment grand, ce qui nous donne une surjection continue 
$$\beta: D_r(G_i)\otimes_{\Lambda(G_i)}\Omega^1(\mathfrak{X}_i)[\tfrac{1}{p}]\to\Omega^1(\mathfrak{X}_i)[\tfrac{1}{p}].$$
Enfin, on définit 
$$D(G)\hat{\otimes}_{\Lambda(G)}\Omega^{1, b}(\Sigma_n) :=\varprojlim_i\varprojlim_r(D_r(G_i)\hat{\otimes}_{\Lambda(G_i)}\Omega^1(\mathfrak{X}_i)[\tfrac{1}{p}] ),$$ 
alors l'espace $D(G)\hat{\otimes}_{\Lambda(G)}\Omega^{1, b}(\Sigma_n) $ est un espace de Fréchet. D'après la propriété universelle du produit tensoriel algébrique, il y a une flèche naturelle $D(G)\otimes_{\Lambda(G)}\Omega^{1, b}(\Sigma_n) \to D(G)\hat{\otimes}_{\Lambda(G)}\Omega^{1, b}(\Sigma_n)$.
\begin{lemma}\label{3.5.1} Soient $R\to S$ un morphisme d'anneaux non forcément commutatifs et $M$ un $S$-module à gauche, alors le noyau de la surjection naturelle $S\otimes_RM\twoheadrightarrow M$ induite par $s\otimes m\mapsto sm$ est engendré par 
	$$\{s_1\otimes s_2m-s_1s_2\otimes m\mid \forall s_1, s_2\in S,\forall m\in M\}.$$
\end{lemma} 
\begin{proof}[Preuve] D'une part, il est évident que le sous-module engendré par $\{s_1\otimes s_2m-s_1s_2\otimes m\mid \forall s_1, s_2\in S,\forall m\in M\}$ est contenu dans le noyau. Donc on a un morphisme de $S$-modules à gauche
	$$S\otimes_RM/\langle s_1\otimes s_2m-s_1s_2\otimes m\rangle\to M.$$
	D'autre part, on a une flèche $S\times M\to S\otimes_RM/\langle s_1\otimes s_2m-s_1s_2\otimes m\rangle$, ce qui induit une flèche $$M=S\otimes_SM\to S\otimes_RM/\langle s_1\otimes s_2m-s_1s_2\otimes m\rangle$$ en utilisant la propriété universelle du produit tensoriel.
	
	Les deux flèches sont inverses l'une de l'autre, ce qui permet de conclure.
\end{proof}
\begin{lemma}\label{3.5.2} On a un isomorphisme d'espaces de Fréchet $$D(G)\hat{\otimes}_{\Lambda(G)}\Omega^1[M]^b\cong\Omega^1[M].$$
\end{lemma}
\begin{proof}[Preuve] Il suffit de montrer que le morphisme $D(G)\otimes_{\Lambda(G)}\Omega^{1, b}(\Sigma_n)\to\Omega^1(\Sigma_n)$ induit par $\beta$ se factorise par $D(G)\otimes_{\Lambda(G)}\Omega^{1, b}(\Sigma_n) \to D(G)\hat{\otimes}_{\Lambda(G)}\Omega^{1, b}(\Sigma_n)$ et le morphisme 
	$$D(G)\hat{\otimes}_{\Lambda(G)}\Omega^{1, b}(\Sigma_n)\to\Omega^1(\Sigma_n)$$
	est un isomorphisme. 
	
	Comme l'espace $\Omega^1(\mathfrak{X}_i)[\tfrac{1}{p}]$ est un espace de Banach, le morphisme $\beta$ se factorise par $D_r(G_i)\otimes_{\Lambda(G_i)} \Omega^1(\mathfrak{X}_i)[\tfrac{1}{p}]\twoheadrightarrow D_r(G_i)\otimes'_{\Lambda(G_i)}\Omega^1(\mathfrak{X}_i)[\tfrac{1}{p}] \hookrightarrow D_r(G_i)\hat{\otimes}_{\Lambda(G_i)} \Omega^1(\mathfrak{X}_i)[\tfrac{1}{p}]$. On a donc un diagramme commutatif 
	$$\xymatrix@R=5mm@C=4mm{
		D_r(G_i)\otimes_{\Lambda(G_i)} \Omega^1(\mathfrak{X}_i)[\tfrac{1}{p}] \ar[drr]_\beta\ar@{->>}[r]& D_r(G_i)\otimes'_{\Lambda(G_i)}\Omega^1(\mathfrak{X}_i)[\tfrac{1}{p}] \ar@{^{(}->}[r]\ar@{.>}[dr]^\gamma&   D_r(G_i)\hat{\otimes}_{\Lambda(G_i)} \Omega^1(\mathfrak{X}_i)[\tfrac{1}{p}]\ar@{.>}[d]_i \\
		&&      \Omega^1(\mathfrak{X}_i)[\tfrac{1}{p}]
	}.$$
	Il suit du Lemme \ref{3.5.1} que le noyau de $\beta$ est le $D_r(G_i)$-module engendré par $d\omega\otimes d'-\omega\otimes dd'$ où $\omega\in\Omega^1(\mathfrak{X}_i)[\tfrac{1}{p}]$ et $d, d'\in D_r(G_i)$. Comme $\Lambda(G_i)$ est dense dans $D_r(G_i)$, on peut écrire $d$ comme $d=\lim d_i$ où $d_i\in\Lambda(G_i)$, alors dans $D_r(G_i)\otimes'_{\Lambda(G_i)}\Omega^1(\mathfrak{X}_i)[\frac{1}{p}] $, on a 
	$$d\omega\otimes d'-\omega\otimes dd'=\lim(d_i\omega\otimes d'-\omega\otimes d_id')=\lim(\omega\otimes d_id'-\omega\otimes d_id')=0.$$ 
	Cela implique que $\gamma$ est injectif. 
	
	On doit montrer que $\gamma$ est une isométrie. Plus précisément, pour $z\in D_r(G_i)\otimes'_{\Lambda(G_i)}\Omega^1(\mathfrak{X}_i)[\tfrac{1}{p}]$, il s'agit de montrer que $||\gamma(z)||=||z||$.
	
	(i) $||\gamma(z)||\leq||z||$. Fixons une expression $z=\sum_kx_k\otimes y_k$, on doit montrer que $$||\gamma(z)||\leq\max_k|x_k|\cdot||y_k||.$$
	Il nous suffit de montrer que $\lambda=\frac{1}{\max_k|x_k|\cdot||y_k||}$ satisfait la condition que $\gamma(z)=\sum_kx_ky_k\in\lambda\Omega^1(\mathfrak{X}_i)$, ce qui est évident.
	
	(ii)  $||\gamma(z)||\geq||z||$. Fixons $\lambda$ tel que $\gamma(z)\in\lambda\Omega^1(\mathfrak{X}_i)$, on doit montrer que $$|\lambda|\geq\inf\{\max_k|x_k|\cdot||y_k||\}.$$
	Il nous suffit de trouver une expression de $z$ telle que $|\lambda|\geq\max_k|x_k|\cdot||y_k||$. Supposons que $z=\sum_kx_k\otimes y_k$ où $x_k\in D_r(G_i)$ et $y_k\in\Omega^1(\mathfrak{X}_i)[\tfrac{1}{p}]$, alors $z=1\otimes(\sum_kx_ky_k)$. On est ramené à montrer que $|\lambda|\geq||\sum_kx_ky_k||=||\gamma(z)||$, ce qui est immédiat.
	
	L'espace $D_r(G_i)\hat{\otimes}_{\Lambda(G_i)}\Omega^1(\mathfrak{X}_i)[\tfrac{1}{p}]$ peut donc être réalisé comme l'adhérence de $D_r(G_i)\otimes'_{\Lambda(G_i)}\Omega^1(\mathfrak{X}_i)[\tfrac{1}{p}]$ dans $\Omega^1(\mathfrak{X}_i)[\tfrac{1}{p}]$, ce qui implique que le morphisme $i$ est injectif. D'après le théorème de l'application ouverte, le morphisme 
	$$D_r(G_i) \hat{\otimes}_{\Lambda(G_i)}\Omega^1(\mathfrak{X}_i)[\tfrac{1}{p}]\to    \Omega^1(\mathfrak{X}_i)[\tfrac{1}{p}]$$
	est un isomorphisme. On en a obtenu un isomorphisme
	$$D(G)\hat{\otimes}_{\Lambda(G)}\Omega^{1, b}(\Sigma_n) \to\Omega^1(\Sigma_n)$$
	en prenant la limite sur $r$ et $i$. Cela permet de conclure.
\end{proof}
\begin{corollary}\label{3.4.3} Le sous-espace $\Omega^1[M]^b$ est dense dans $\Omega^1[M]$.
\end{corollary} 
\begin{proof}[Preuve] Comme $\Lambda(G)$ est dense dans $D(G)$, le morphisme $\Omega^1[M]^b\to D(G)\hat{\otimes}_{\Lambda(G)}\Omega^1[M]^b$ est d'image dense. Le résultat découle du Lemme \ref{3.5.2}.
\end{proof}
Le lemme suivant généralise le \cite[Theorem 7.1 iii]{st2003algebras}.
\begin{lemma}\label{relation} Soit $\Pi$ une représentation de Banach de $G$ sur $L$, alors on a 
	$$D(G)\hat{\otimes}_{\Lambda(G)}\Pi^*=\Pi^{\mathrm{an}, *}.$$
\end{lemma}
\begin{proof}[Preuve] L'inclusion $\Pi^{\mathrm{an}}\hookrightarrow\Pi$ induit un morphisme $\Pi^*\to\Pi^{\mathrm{an}, *}$ en prenant le dual. La structure de $D(G)$-module sur $\Pi^{\mathrm{an}, *}$ induit un morphisme continu $D(G)\otimes_{\Lambda(G)}\Pi^*\to\Pi^{\mathrm{an}, *}$. Étant donné que $\Pi^{\mathrm{an}, *}$ est un espace de Fréchet, ce morphisme se factorise par $i: D(G)\hat{\otimes}_{\Lambda(G)}\Pi^*\to\Pi^{\mathrm{an}, *}$. Comme $\Lambda(G)$ est dense dans $D(G)$, on a un morphisme naturel $j: \Pi^*\to D(G)\hat{\otimes}_{\Lambda(G)}\Pi^*$ qui est d'image dense, ce qui induit une injection $j^*: (D(G)\hat{\otimes}_{\Lambda(G)}\Pi^*)^*\to\Pi$ en prenant le morphisme dual. La composition $j^*\circ i^*: \Pi^{\mathrm{an}}\to(D(G)\hat{\otimes}_{\Lambda(G)}\Pi^*)^*\hookrightarrow\Pi$ est l'inclusion naturelle $\Pi^{\mathrm{an}}\hookrightarrow\Pi$, ce qui implique que l'image de $j^*$ contient $\Pi^{\mathrm{an}}$. Comme $(D(G)\hat{\otimes}_{\Lambda(G)}\Pi^*)^*$ est une représentation localement analytique, l'image de $j^*$ est contenue dans $\Pi^{\mathrm{an}}$. Le morphisme $j^*$ induit donc une bijection $(D(G)\hat{\otimes}_{\Lambda(G)}\Pi^*)^*\to\Pi^{\mathrm{an}}$. Le résultat suit du théorème de l'application ouverte pour les espaces de Fréchet.
\end{proof}
\begin{corollary}\label{3.4.4} Le morphisme naturel $$\Omega^1[M]^*\to(\widehat{\Omega^1[M]^*})^{\mathrm{an}}$$
	est un isomorphisme.
\end{corollary}
\begin{proof}[Preuve] Il suit du \cite[Lemme 5.3]{cdn2023factorisation} et du Lemme \ref{relation} que l'on a un isomorphisme 
	$$D(G)\hat{\otimes}_{\Lambda(G)}\Omega^1[M]^b=\Omega^1[M]^{b, *, \mathrm{an}, *}.$$
	En appliquant le Lemme \ref{3.5.2}, on obtient un isomorphisme 
	$$\Omega^1[M]^{b, *, \mathrm{an}, *}\cong\Omega^1[M].$$
	Le résultat se déduit en prenant les espaces duals.
\end{proof}
Passons maintenant à la partie lisse de $\widehat{\Omega^1[M]^*}$. 
\begin{proposition}\label{3.4.5} On a $$\Omega^1[M]^{*, {\mathrm{lisse}}}=M_{\mathrm{dR}}\otimes_L\mathrm{LL}(M).$$
\end{proposition} 
\begin{proof}[Preuve] Considérons la suite exacte de (\ref{equation4})
	$$	0 \to M_{\mathrm{dR}}\otimes_L\mathrm{LL}(M)\to\Omega^1[M]^*\to\mathscr{O}[M]^*\to0.$$
	Soit $v\in\Omega^1[M]^{*, \mathrm{lisse}}$, alors l'image de $v$ se trouve dans $\mathscr{O}[M]^{*, \mathrm{lisse}}=0$ (preuve de l'assertion (i) du Lemme \ref{2.3.1}), ce qui implique que $v\in M_{\mathrm{dR}}\otimes_L\mathrm{LL}(M)$. Comme $\mathrm{LL}(M)$ est lisse, on a $\Omega^1[M]^{*, {\mathrm{lisse}}}=M_{\mathrm{dR}}\otimes_L\mathrm{LL}(M)$, ce que l'on voulait.
\end{proof}
\begin{corollary}\label{3.5.6} On a $$(\widehat{\Omega^1[M]^*})^{\mathrm{lisse}}=M_{\mathrm{dR}}\otimes_L\mathrm{LL}(M).$$
\end{corollary}
\begin{proof}[Preuve] Combiner le Corollaire \ref{3.4.4} et la Proposition \ref{3.4.5}.
\end{proof}
\subsection{Finitude résiduelle de $\widehat{\mathrm{LL}(M)}$ et $\widehat{\Omega^1[M]^*}$}
\begin{proposition}\label{3.6.1} Tout quotient de la représentation $\widehat{\mathrm{LL}(M)}$ est résiduellement de présentation finie.
\end{proposition}
\begin{proof}[Preuve] On peut remplacer $\mathrm{LL}(M)$ par $\mathrm{ind}_{KZ}^G\sigma_M$. Prenons un réseau $\sigma_M^0$ de $\sigma_M$ et notons $\{\sigma_i\}_i$ les facteurs de Jordan-Hölder de $\sigma_M^0$, alors la réduction de $\widehat{\mathrm{LL}(M)}$ est l'extension des $\mathrm{ind}_{KZ}^G\sigma_i$. Le résultat se déduit du Corollaire \ref{3.1.3}.
\end{proof}
\begin{proposition}\label{3.6.2} La représentation $\widehat{\Omega^1[M]^*}$ est résiduellement de présentation finie.
\end{proposition}
\begin{proof}[Preuve] Considérons la suite exacte 
	$$0\to\mathscr{L}\otimes_L\mathrm{LL}(M)\to \Omega^1[M]^*\to \Pi_{M, \mathscr{L}}^{\mathrm{an}}\to0.$$
	Soit $\Lambda$ un réseau ouvert, $G$-stable et minimal (à homothétie près) de $\Omega^1[M]^*$, alors $\Lambda_1:= \mathrm{LL}(M)\cap\Lambda$ est un réseau ouvert de $\mathrm{LL}(M)$, stable par $G$. Notons $\Lambda_2$ l'image de $\Lambda$, alors on a une suite exacte 
	$$0\to\Lambda_1/\pi_L\to\Lambda/\pi_L\to\Lambda_2/\pi_L\to0.$$
	Il suit de la \cite[Corollary 1.6]{cdp2014the} et du \cite[Remark 2.5.15]{deg2023localization} que $\Lambda_2/\pi_L$ est de présentation finie. Comme $\Lambda_1/\pi_L$ est également de présentation finie, il découle du Corollaire \ref{3.1.3} que la représentation $\Lambda/\pi_L$ est de présentation finie. Cela permet de conclure.
\end{proof}
\subsection{Preuve du Théorème \ref{0.0.1}}
Dans cette section, on va montrer le Théorème \ref{0.0.1}, qui est le premier résultat principal dans cette thèse. Cela généralise un résultat de Lue Pan \cite{pan2017first}. Tout d'abord, on a besoin du lemme suivant
\begin{lemma}\label{3.7.1} Supposons que $\mathscr{L}\neq\mathscr{L}'$, alors la composée $\mathscr{L}\otimes_L\widehat{\mathrm{LL}(M)}\to \widehat{\Omega^1[M]^*}\to\Pi_{M, \mathscr{L}'}$ est non nulle.
\end{lemma}
\begin{proof}[Preuve] Supposons que la composée est nulle, alors la composée  $\mathscr{L}\otimes_L\mathrm{LL}(M)\to \Omega^1[M]^*\to\Pi_{M, \mathscr{L}'}^{\mathrm{an}}$ est également nulle en passant aux parties localement analytiques. Cela conduit à une contradiction car le noyau de $\Omega^1[M]^*\to\Pi_{M, \mathscr{L}'}^{\mathrm{an}}$ est $\mathscr{L}'\otimes_L\mathrm{LL}(M)$ (Corollaire \ref{2.1.3}). Cela permet de conclure.
\end{proof}
\begin{theorem}\label{3.7.2} Le choix de $\mathscr{L}\in\mathbf{P}^1$ induit la suite exacte non scindée de représentations de Banach de $G$ suivante $$0\to\widehat{\mathrm{LL}(M)}\to \Omega^1[M]^{\mathrm{b}, *}\to\Pi_{M, \mathscr{L}}\to0.$$
\end{theorem}
\begin{proof}[Preuve] Considérons tout d'abord la suite exacte suivante
	$$0\to\mathscr{L}\otimes_L\mathrm{LL}(M)\to \Omega^1[M]^*\to \Pi_{M, \mathscr{L}}^{\mathrm{an}}\to0.$$
	D'après la Proposition \ref{1.3.3}, l'image de $\widehat{\mathrm{LL}(M)}\to \Omega^1[M]^{\mathrm{b}, *}$ est dense dans le noyau $\mathrm{Ker}f$ de $f: \Omega^1[M]^{\mathrm{b}, *}\to\Pi_{M, \mathscr{L}}$, et le morphisme $\widehat{\Omega^1[M]^*}\to\Pi_{M, \mathscr{L}}$ est surjectif. 
	
	On montre ensuite que le morphisme $\mathscr{L}\otimes_L\widehat{\mathrm{LL}(M)}\to \Omega^1[M]^{\mathrm{b}, *}$ est injectif. Or, il suit du Lemme \ref{3.7.1} que la composée $$\mathscr{L}\otimes_L\widehat{\mathrm{LL}(M)}\to \Omega^1[M]^{\mathrm{b}, *}\to\prod\limits_{\mathscr{L}'\neq\mathscr{L}}\Pi_{M, \mathscr{L}'}$$
	est non nulle sur chaque composante, donc est injective en vertu du Corollaire \ref{3.3.6}, ce qui implique l'injectivité de $\widehat{\mathrm{LL}(M)}\to \Omega^1[M]^{\mathrm{b}, *}$.
	
	On montre maintenant que l'image de $\widehat{\mathrm{LL}(M)}\to \Omega^1[M]^{\mathrm{b}, *}$ est exactement $\mathrm{Ker}f$. D'après ce qui précède, $\widehat{\mathrm{LL}(M)}$ est dense dans $\mathrm{Ker}f$, donc la composée $\mathrm{LL}(M)\hookrightarrow \widehat{\mathrm{LL}(M)}\hookrightarrow\mathrm{Ker}f$ est aussi d'image dense. La suite exacte 
	$$0\to\mathrm{Ker}f\to\widehat{\Omega^1[M]^*}\to\Pi_{M, \mathscr{L}}\to0$$
	induit une suite exacte
	$$0\to\mathrm{Ker}f^+/\pi_L\to\widehat{\Omega^1[M]^*}^+/\pi_L\to\Pi_{M, \mathscr{L}}^+/\pi_L\to0$$	
	où $+$ désigne la boule unité.
	Il suit de la Proposition \ref{3.6.1} et de la Proposition \ref{3.6.2} que les représentations $\widehat{\Omega^1[M]^*}^+/\pi_L$ et $\Pi_{M, \mathscr{L}}^+/\pi_L$ sont toutes de présentation finie, donc le noyau $\mathrm{Ker}f$ est résiduellement de présentation finie en vertu du Corollaire \ref{3.1.3}. Le Lemme \ref{2.2.2} implique que la suite ci-dessous est exacte
	$$\mathscr{L}\otimes_L\widehat{\mathrm{LL}(M)}\to \Omega^1[M]^{\mathrm{b}, *}\to\Pi_{M, \mathscr{L}}\to0.$$
	
	Il nous reste à montrer que cette suite est non scindée. Supposons qu'il y a une section $\Omega^1[M]^{\mathrm{b}, *}\to\widehat{\mathrm{LL}(M)}$. Comme $\Omega^1[M]^{\mathrm{b}, *}=\widehat{\Omega^1[M]^*}$, un morphisme non nul $\Omega^1[M]^{\mathrm{b}, *}\to\widehat{\mathrm{LL}(M)}$ est induit par un morphisme non nul $\Omega^1[M]^*\to\mathrm{LL}(M)$ car $\widehat{\mathrm{LL}(M)}^{\mathrm{an}}=\mathrm{LL}(M)$, ce qui contredit le Lemme \ref{2.3.4} (ii). Cela permet de conclure.
\end{proof}
\begin{corollary}\label{2.6.3} Les $N_{\mathscr{L}, 1}$ sont tous isomorphes à $H^1(\mathfrak{X}, \mathscr{O})[\tfrac{1}{p}]^*[M]$. En particulier, on a une suite exacte
	$$0\to H^1(\mathfrak{X}, \mathscr{O})[\tfrac{1}{p}]^*[M]\to\widehat{\mathrm{LL}(M)}\to\Pi_{M, \mathscr{L}} \to0$$
	pour toute $\mathscr{L}$.
\end{corollary}
\begin{proof}[Preuve] On déduit du Théorème \ref{3.7.2} un diagramme commutatif à lignes exactes
	$$	\xymatrix@R=5mm@C=4mm{
		0 \ar[r] & \mathscr{L} \otimes \widehat{\mathrm{LL}(M)} \ar[r] \ar@{=}[d] & M_{\mathrm{dR}} \otimes \widehat{\mathrm{LL}(M)} \ar[r] \ar[d] & (M_{\mathrm{dR}}/\mathscr{L})\otimes \widehat{\mathrm{LL}(M)} \ar[r] \ar@{->>}[d] & 0 \\
		0 \ar[r] &  \mathscr{L} \otimes\widehat{\mathrm{LL}(M)} \ar[r] & \Omega^1[M]^{\mathrm{b}, *} \ar[r]  & \Pi_{M, \mathscr{L}} \ar[r] & 0.
	}$$	
	Cela implique que le noyau de la flèche $M_{\mathrm{dR}} \otimes \widehat{\mathrm{LL}(M)}\to\Omega^1[M]^{\mathrm{b}, *}$ est isomorphe à celui de la flèche $(M_{\mathrm{dR}}/\mathscr{L})\otimes \widehat{\mathrm{LL}(M)} \to\Pi_{M, \mathscr{L}} $. Or, d'après le Corollaire \ref{1.3.6}, on dispose d'une suite exacte
	$$0\to H^1(\mathfrak{X}, \mathscr{O})[\tfrac{1}{ p}]^*[M]\to \widehat{H_{\mathrm{dR}}^1(\Sigma_n)^*}[M]\to \widehat{\Omega^1(\Sigma_n)^*}[M]\to0,$$
	ce qui implique que le noyau de $M_{\mathrm{dR}} \otimes \widehat{\mathrm{LL}(M)}\to\Omega^1[M]^{\mathrm{b}, *}$ est isomorphe à $H^1(\mathfrak{X}, \mathscr{O})[\tfrac{1}{p}]^*[M]$. Le résultat découle directement de la définition de $N_{\mathscr{L}, 1}$.
\end{proof}
\subsection{Entrelacements entre $\widehat{\mathrm{LL}(M)}$, $\Pi_{M, \mathscr{L}}$ et $\widehat{\Omega^1[M]^*}$} 
Rappelons que l'on a montré que la catégorie des représentations de Banach résiduellement de présentation finie est abélienne dans le Corollaire \ref{3.1.3}. Comme les représentations de Banach $\widehat{\mathrm{LL}(M)}$, $\Pi_{M, \mathscr{L}}$ et $\widehat{\Omega^1[M]^*}$ sont toutes résiduellement de présentation finie, on va calculer les entrelacements entre elles dans cette catégorie. Fixons un caractère central $\psi$ et notons $\mathrm{Ext}^1_{G, \psi}$ le groupe d'extensions à caractère central $\psi$ fixé.
\subsubsection{Homomorphismes}
\begin{proposition}\label{3.8.1} On a 
	
	(i) $\mathrm{End}_G(\widehat{\mathrm{LL}(M)})=L$.
	
	(ii) $\mathrm{Hom}_G(\widehat{\mathrm{LL}(M)}, \widehat{\Omega^1[M]^*})=M_{\mathrm{dR}}$.
\end{proposition}
\begin{proof}[Preuve] (i) Comme la représentation $\mathrm{LL}(M)$ est lisse, il suit de la propriété universelle de $\widehat{\mathrm{LL}(M)}$ et de la Proposition \ref{3.2.2} que l'on a
	$$\mathrm{End}_G(\widehat{\mathrm{LL}(M)})=\mathrm{Hom}_G(\mathrm{LL}(M), \widehat{\mathrm{LL}(M)})=\mathrm{Hom}_G(\mathrm{LL}(M), \widehat{\mathrm{LL}(M)}^{\mathrm{lisse}})=\mathrm{End}_G(\mathrm{LL}(M)).$$
	
	(ii) Comme la représentation $\mathrm{LL}(M)$ est lisse, il suit de la propriété universelle de $\widehat{\mathrm{LL}(M)}$ et du Corollaire \ref{3.5.6} que l'on a
	\begin{align*}\mathrm{Hom}_G(\widehat{\mathrm{LL}(M)}, \widehat{\Omega^1[M]^*})&=\mathrm{Hom}_G(\mathrm{LL}(M), \widehat{\Omega^1[M]^*})=\mathrm{Hom}_G(\mathrm{LL}(M), \widehat{\Omega^1[M]^*}^{\mathrm{lisse}})\\&=\mathrm{Hom}_G(\mathrm{LL}(M), \mathrm{LL}(M)\otimes M_{\mathrm{dR}})=M_{\mathrm{dR}}.\end{align*}	
	Cela permet de conclure.	
\end{proof}
\begin{corollary}\label{cor2} Soit $f\in \mathrm{Hom}_G(\widehat{\mathrm{LL}(M)}, \widehat{\Omega^1[M]^*})$ un morphisme non nul, alors le conoyau de $f$ est isomorphe à $\Pi_{M, \mathscr{L}}$ pour une certaine $\mathscr{L}$.
\end{corollary}
\begin{proof}[Preuve] La Proposition \ref{3.8.1} (ii) implique que $f$ est donné par l'inclusion $\mathscr{L}\otimes_L\widehat{\mathrm{LL}(M)}\hookrightarrow\widehat{\Omega^1[M]^*}$, alors le résultat suit du Théorème \ref{3.7.2}.
\end{proof}
\begin{proposition} On a 
	
	(i) $\mathrm{Hom}_G(\Pi_{M, \mathscr{L}}, \widehat{\mathrm{LL}(M)})=0$.
	
	(ii) $\mathrm{Hom}_G(\Pi_{M, \mathscr{L}}, \widehat{\Omega^1[M]^*})=0$.
\end{proposition}
\begin{proof}[Preuve] (i) Soit $f$ un morphisme $\Pi_{M, \mathscr{L}}\to\widehat{\mathrm{LL}(M)}$, il suit de la Proposition \ref{3.2.3} que $f$ induit un morphisme $f^{\mathrm{an}}: \Pi_{M, \mathscr{L}}^{\mathrm{an}}\to\widehat{\mathrm{LL}(M)}^{\mathrm{an}}=\mathrm{LL}(M)$. Le Lemme \ref{2.3.4} (iv) implique que $f^{\mathrm{an}}=0$. Comme $f$ est continu et $\Pi_{M, \mathscr{L}}^{\mathrm{an}}$ est dense dans $\Pi_{M, \mathscr{L}}$, on déduit que $f=0$. 
	
	(ii) La preuve est similaire à celle de (i), sauf qu'ici nous utilisons le Corollaire \ref{3.4.4} et la Proposition \ref{2.3.6}.
\end{proof}
\begin{proposition} On a $$\mathrm{End}_G(\widehat{\Omega^1[M]^*})=L.$$
\end{proposition}
\begin{proof}[Preuve] Comme la représentation $\Omega^1[M]^*$ est localement analytique, il suit de la propriété universelle de $\widehat{\Omega^1[M]^*}$ et du Corollaire \ref{3.4.4} que l'on a
	$$\mathrm{End}_G(\widehat{\Omega^1[M]^*})=\mathrm{Hom}_G(\Omega^1[M]^*, \widehat{\Omega^1[M]^*})=\mathrm{Hom}_G(\Omega^1[M]^*, \widehat{\Omega^1[M]^*}^{\mathrm{an}})=\mathrm{End}_G(\Omega^1[M]^*).$$
	Ainsi, le résultat se déduit de la Proposition \ref{2.3.8}. Cela permet de conclure.
\end{proof}
\begin{lemma}\label{3.7.7} On a $$\mathrm{Hom}_G(H^1(\mathfrak{X}, \mathscr{O})[\tfrac{1}{p}]^*[M], \Pi_{M, \mathscr{L}})=L.$$
\end{lemma}
\begin{proof}[Preuve] En appliquant le foncteur $\mathrm{Hom}_G(-, \Pi_{M, \mathscr{L}})$ à la suite exacte
	$$0\to H^1(\mathfrak{X}, \mathscr{O})[\tfrac{1}{p}]^*[M]\to\widehat{\mathrm{LL}(M)}\to\Pi_{M, \mathscr{L}'}\to0$$
	où $\mathscr{L}'\neq\mathscr{L}$, on a une suite exacte
	\begin{align*}0&\to\mathrm{Hom}_G(\Pi_{M, \mathscr{L}'}, \Pi_{M, \mathscr{L}})\to\mathrm{Hom}_G(\widehat{\mathrm{LL}(M)}, \Pi_{M, \mathscr{L}})\to\mathrm{Hom}_G(H^1(\mathfrak{X}, \mathscr{O})[\tfrac{1}{p}]^*[M], \Pi_{M, \mathscr{L}})\\& \to\mathrm{Ext}^1_G(\Pi_{M, \mathscr{L}'}, \Pi_{M, \mathscr{L}}).	
	\end{align*}
	Le résultat découle du fait que $\mathrm{Hom}_G(\widehat{\mathrm{LL}(M)}, \Pi_{M, \mathscr{L}})=L$, du Théorème \ref{1.1.4} et de la Proposition \ref{2.1.1}.
\end{proof}
\subsubsection{Extensions}
Dans le lemme suivant, on suppose que $G$ est un groupe quelconque.
\begin{lemma}\label{3.7.4} Soient $H\supseteq H'$ deux sous-groupes d'un groupe $G$, $W$ une représentation de $H$ et $\gamma\in G$. 
	
	(i) On a un isomorphisme naturel de représentations de $H$
	$$\mathrm{ind}_H^{H\gamma H}W\cong\mathrm{ind}_{H^\gamma\cap H}^HW^\gamma,$$
	où $H^\gamma=\gamma H\gamma^{-1}$, et $W^\gamma$ est la représentation de $H^\gamma\cap H'$ obtenue en faisant agir $g$ sur $W$ par 
	$$g\cdot_\gamma v=\gamma^{-1}g\gamma\cdot v.$$
	
	(ii)(Formule de Mackey) On a un isomorphisme naturel de représentations de $H'$
	$$\mathrm{ind}_H^{H\gamma H}W\cong\bigoplus_{s\in H'\backslash H/(H^\gamma\cap H)}\mathrm{ind}_{H^{s\gamma}\cap H'}^{H'}W^{s\gamma},$$
	où $H^{s\gamma}=s\gamma H\gamma^{-1}s^{-1}$, et $W^{s\gamma}$ est la représentation de $H^{s\gamma}\cap H'$ obtenue en faisant agir $g$ sur $W$ par 
	$$g\cdot_{s\gamma}v=\gamma^{-1}s^{-1}gs\gamma\cdot v.$$
	
	(iii)(Lemme de Shapiro) On a $H^1(H', \mathrm{ind}_{H^{s\gamma}\cap H'}^{H'}W^{s\gamma})=H^1(H^{s\gamma}\cap H', W^{s\gamma})$.
\end{lemma}
\begin{proof}[Preuve] (i) Soit $\phi\in\mathrm{ind}_H^{H\gamma H}W$, on définit une flèche $\phi_\gamma: H\to W$ par $\phi_\gamma(x)=\phi(x\gamma)$, alors, si $h\in H^\gamma\cap H$ (en particulier, $\gamma^{-1}h\gamma\in H$) et $x\in H$, on a 
	$$\phi_\gamma(xh^{-1})=\phi(xh^{-1}\gamma)=\phi(x\gamma\gamma^{-1}h^{-1}\gamma)=\gamma^{-1}h\gamma\cdot\phi(x\gamma)=\gamma^{-1}h\gamma\cdot\phi_\gamma(x),$$
	ce qui prouve que $\phi_\gamma\in\mathrm{ind}_{H^\gamma\cap H}^HW^\gamma$.
	
	Soit $\psi\in\mathrm{ind}_{H^\gamma\cap H}^HW^\gamma$, on définit une flèche $\psi_\gamma: H\gamma H\to W$ par $\psi_\gamma(h_1\gamma h_2)= h_2^{-1}\cdot\psi(h_1)$, alors on a 
	$$\psi_r(h_1\gamma h_2h_3^{-1})= h_3h_2^{-1}\cdot\psi(h_1) =h_3\cdot\psi_\gamma(h_1\gamma h_2),$$
	ce qui prouve que $\psi_r\in\mathrm{ind}_H^{H\gamma H}W$.
	
	Maintenant on vérifie sans mal que les deux flèches sont inverses l'une de l'autre, ce qui permet de conclure l'assertion (i).
	
	(ii) On déduit de l'assertion (i) et la formule de Mackey \cite[Section 7.3]{serre1977linear} les isomorphismes suivants de représentations de $H'$
	$$\mathrm{ind}_H^{H\gamma H}W\cong\mathrm{Res}_{H'}^H\mathrm{ind}_{H^\gamma\cap H}^HW^\gamma\cong\bigoplus_{s\in H'\backslash H/(H^\gamma\cap H)}\mathrm{ind}_{(H^\gamma\cap H)^s\cap H'}^{H'}(W^{\gamma})^{s}\cong\bigoplus_{s\in H'\backslash H/(H^\gamma\cap H)}\mathrm{ind}_{H^{s\gamma}\cap H'}^{H'}(W^{\gamma})^{s}$$
	où $(W^\gamma)^s$ est la représentation de $(H^\gamma\cap H)^s\cap H'$ obtenue en faisant agir $g$ sur $W^\gamma$ par $g\cdot_sv=s^{-1}gs\cdot_\gamma v$. On vérifie facile que $(W^\gamma)^s=W^{s\gamma}$, ce qui permet de conclure l'assertion (ii).
	
	(iii) Voir la \cite[Proposition 8]{stix2010trading}.
\end{proof}
\begin{lemma}\label{ouvert} (i) Soient $t\in G=\mathrm{GL}_2(\mathbb{Q}_p)$ et $H=\mathrm{SL}_2(\z_p)$, alors $H^t$ contient un sous-groupe de congruence principal de $H$.
	
	(ii)  Supposons de plus que $H'$ est un sous-groupe ouvert de $H$, alors $H'^t\cap H'$ est ouvert dans $H$.
\end{lemma}
\begin{proof}[Preuve] (i) Soit $d$ un nombre pair suffisamment grand tel que $p^{\frac{d}{2}}t, p^{\frac{d}{2}}t^{-1}\in M_2(\z_p)$. On note $\Gamma_d:=(I+p^dM_2(\z_p))\cap H$ qui est aussi le noyau de la morphisme de réduction $H\to\mathrm{SL}_2(\z_p/p^d\z_p)$, alors on a $t^{-1}\Gamma_dt\subseteq H$. En effet, soit $I+p^dk\in\Gamma_d$ avec $k\in M_2(\z_p)$, alors on a  $$t^{-1}(I+p^dk)t=I+p^dt^{-1} kt\in M_2(\z_p)$$
	d'après notre choix de $d$. De plus, on a 
	$$|t^{-1}(I+p^dk)t|=|I+p^dk|=1$$
	car $I+p^dk\in H$. On a donc $t^{-1}(I+p^dk)t\in M_2(\z_p)\cap\mathrm{SL}_2(\mathbb{Q}_p)=H$, ce qui implique que $\Gamma_d\subseteq H^t$.

	(ii) On a besoin de montrer que $H'^t\cap H'=H'^t\cap H\cap H'$ est ouvert dans $H$, il suffit donc de montrer que $H'^t\cap H$ est ouvert dans $H$. Comme $H'$ est ouvert dans $H$, $H'^t\cap H$ est ouvert dans $H^t\cap H$, on est donc ramené à montrer que $H^t\cap H$ est ouvert dans $H$. D'après l'assertion (i), le sous-groupe $H^t\cap H$ de $H$ contient $\Gamma_d$, qui est ouvert dans $H$, donc $H^t\cap H$ est ouvert dans $H$. Cela permet de conclure.
\end{proof}
\begin{proposition}\label{3.7.5} Soit $W$ une $\mathscr{O}_L$-représentation lisse de $KZ$, libre de rang fini et de caractère central $\psi$, alors $\mathrm{Ext}^1_{G, \psi}(\mathrm{Ind}_{KZ}^GW, \mathrm{Ind}_{KZ}^GW)$ est tué par une puissance de $p$.
\end{proposition}
\begin{proof}[Preuve] Par la réciprocité de Frobenius, on a 
	$$\mathrm{Ext}^1_{G, \psi}(\mathrm{Ind}_{KZ}^GW, \mathrm{Ind}_{KZ}^GW)\cong\mathrm{Ext}^1_{KZ, \psi}(\mathrm{Ind}_{KZ}^GW, W)\cong\mathrm{Ext}^1_{KZ, \psi}(W^*, \mathrm{ind}_{KZ}^GW^*),$$
	où $W^*$ est le $\mathscr{O}_L$-dual de $W$. De plus, si $\gamma_n=\begin{psmallmatrix}
		p^n&0\\
		0&1
	\end{psmallmatrix}$, alors on a un isomorphisme de $KZ$-module
	$$\mathrm{ind}_{KZ}^GW^*=\widehat{\bigoplus}_{n\geq0}\mathrm{ind}_{KZ}^{KZ\gamma_nKZ}W^*,$$
	donc il suffit de montrer que $\mathrm{Ext}^1_{KZ, \delta}(W^*, \mathrm{ind}_{KZ}^{KZ\gamma_nKZ}W^*)$ est tué par une puissance de $p$ indépendante de $n$. Comme on a fixé le caractère central, on peut remplacer $KZ$ par $H=\mathrm{SL}_2(\mathbb{Z}_p)$. Comme $W$ est lisse, quelque sous-groupe ouvert $H'$ de $H$ agit trivialement sur $W$. Comme $H$ est compact, $H'$ est d'indice fini et on peut supposer que $H'$ est distingué dans $H$.
	
	Comme $H/H'$ est fini, le groupe $H^1(H/H', (W\otimes\mathrm{ind}_H^{H\gamma_nH}W^*)^{H'})$ est tué par $[H:H']$. En utilisant la suite d'inflation-restriction
	$$0\to H^1(H/H', (W\otimes\mathrm{ind}_H^{H\gamma_nH}W^*)^{H'})\to H^1(H, W\otimes\mathrm{ind}_H^{H\gamma_nH}W^*)\to H^1(H', W\otimes\mathrm{ind}_H^{H\gamma_nH}W^*)^{H/H'}\to\cdot\cdot\cdot,$$
	on est ramené à montrer que $H^1(H', \mathrm{ind}_H^{H\gamma_nH}W^*)$ est tué par une puissance de $p$ indépendante de $n$ (notons que $H'$ agit trivialement sur $W$, donc $W\otimes\mathrm{ind}_H^{H\gamma_nH}W^*$ est isomorphe à une somme directe d'un nombre fini de copies de $\mathrm{ind}_H^{H\gamma_nH}W^*$ en tant que $H'$-représentation). D'après les points (ii) et (iii) du Lemme \ref{3.7.4}, on a 
	\begin{align*}
		H^1(H', \mathrm{ind}_H^{H\gamma_n H}W^*)&=H^1(H', \bigoplus_{s\in H'\backslash H/(H^{\gamma_n}\cap H)}\mathrm{ind}_{H^{s\gamma_n}\cap H'}^{H'}W^{*s\gamma_n})=\bigoplus_{s\in H'\backslash H/(H^{\gamma_n}\cap H)}H^1(H', \mathrm{ind}_{H^{s\gamma_n}\cap H'}^{H'}W^{*s\gamma_n})\\&=\bigoplus_{s\in H'\backslash H/(H^{\gamma_n}\cap H)}H^1(H^{s\gamma_n}\cap H', W^{*s\gamma_n})=\bigoplus_{s\in H'\backslash H/(H^{\gamma_n}\cap H)}H^1(H\cap H'^{\gamma_n^{-1}s^{-1}}, W^*).
	\end{align*}
	Comme $H'$ agit trivialement sur $W$, on a $H^1(H'\cap H'^{\gamma_n^{-1}s^{-1}}, W^*)=\mathrm{Hom}_{\mathrm{cont}}(H'\cap H'^{\gamma_n^{-1}s^{-1}}, W^*)$. 
	
	D'après le Lemme \ref{ouvert} (ii), le groupe $H'\cap H'^{\gamma_n^{-1}s^{-1}}$ est ouvert dans $H$. Puisque $W^*$ est sans torsion, et qu'il résulte du Corollaire \ref{C.0.7} que l'abélianisé de tout sous-groupe ouvert de $\mathrm{SL}_2(\mathbb{Z}_p)$ est fini, on a $$H^1(H'\cap H'^{\gamma_n^{-1}s^{-1}}, W^*)=0.$$
	
	En utilisant à nouveau la suite d'inflation-restriction
	\begin{align*}
		0&\to H^1((H\cap H'^{\gamma_n^{-1}s^{-1}})/(H'\cap H'^{\gamma_n^{-1}s^{-1}}), W^*)\to H^1(H\cap H'^{\gamma_n^{-1}s^{-1}}, W^*)\\&\to H^1(H'\cap H'^{\gamma_n^{-1}s^{-1}}, W^*)^{(H\cap H'^{\gamma_n^{-1}s^{-1}})/(H'\cap H'^{\gamma_n^{-1}s^{-1}})},
	\end{align*}
	on déduit un isomorphisme
	$$H^1(H\cap H'^{\gamma_n^{-1}s^{-1}}, W^*)\cong H^1((H\cap H'^{\gamma_n^{-1}s^{-1}})/(H'\cap H'^{\gamma_n^{-1}s^{-1}}), W^*).$$
	Cela implique que $H^1(H\cap H'^{\gamma_n^{-1}s^{-1}}, W^*)$ est tué par $[(H\cap H'^{\gamma_n^{-1}s^{-1}}):(H'\cap H'^{\gamma_n^{-1}s^{-1}})]$, et donc aussi par $[H: H']$, ce qui achève la preuve.
\end{proof}
\begin{remark} La même preuve que celle de la Proposition \ref{3.7.5} implique que $\mathrm{Ext}^i_{G, \psi}(\mathrm{Ind}_{KZ}^GW, \mathrm{Ind}_{KZ}^GW)$ est tué par une puissance de $p$ pour tout $i\geq1$.
\end{remark}
\begin{corollary}\label{3.7.6} On a $\mathrm{Ext}^1_{G, \psi}(\widehat{\mathrm{LL}(M)}^*, \widehat{\mathrm{LL}(M)}^*)=0$ et $\mathrm{Ext}^1_{G, \psi}(\widehat{\mathrm{LL}(M)}, \widehat{\mathrm{LL}(M)})=0$.
\end{corollary}
\begin{proof}[Preuve] D'après le Corollaire \ref{classsuper}, la représentation $\widehat{\mathrm{LL}(M)}^*=\mathrm{LL}(M)^*$ est soit isomorphe à, soit un facteur direct de $L\otimes_{\mathscr{O}_L}\mathrm{Ind}_{KZ}^G\sigma_M^{0, *}$, et il résulte de la Proposition \ref{3.7.5} que $L\otimes_{\mathscr{O}_L}\mathrm{Ext}^1_{G, \psi}(\mathrm{Ind}_{KZ}^G\sigma_M^{0, *}, \mathrm{Ind}_{KZ}^G\sigma_M^{0, *})=0$. On en déduit le résultat.
\end{proof} 
\begin{lemma}\label{3.7.8} Soit $0\to A\to B\to C\to0$ une suite exacte courte de $\mathscr{O}_L$-modules. Supposons que $A$ est $p$-adiquement séparé et $C$ est de $\pi_L^N$-torsion, alors $B$ est de $\pi_L$-adiquement séparé.
\end{lemma}
\begin{proof}[Preuve] Soit $b\in\cap_{i\geq1}\pi_L^iB$, alors on peut écrire $b=\pi_L^ib_i$ avec $b_i\in B$. Comme $C$ est de $\pi_L^N$-torsion, on a $\pi_L^Nb=\pi_L^{i+N}b_i=\pi_L^i(\pi_L^Nb_i)\in A$. Comme $\pi_L^Nb_i\in A$, on en déduit que $\pi_L^Nb\in\cap_{i\geq1}\pi_L^iA=\{0\}$. On a donc $\cap_{i\geq1}\pi_L^iB=\pi_L^N(\cap_{i\geq1}\pi_L^iB)=\{0\}$, ce qui achève de démontrer le lemme.
\end{proof}
\begin{lemma}\label{3.7.9} Soient $\Pi_1, \Pi_2$ deux représentations de Banach de $G$ dont les boules unité sont $\Pi_1^+$ et $\Pi_2^+$. Supposons que
	\begin{enumerate}[label=(\roman*)]
		\item l'espace $\mathrm{Ext}_{G, \psi}^1(\overline{\Pi}_1, \overline{\Pi}_2)$ est de dimension finie.
		
		\item $\mathrm{Ext}_{G, \psi}^1(\Pi_1^+, \Pi_2^+)$ est $\pi_L$-adiquement séparé.
	\end{enumerate}
	où $\overline{\Pi}_1:=\Pi_1^+/\pi_L$ et $\overline{\Pi}_2:=\Pi_2^+/\pi_L$, alors l'espace $\mathrm{Ext}_{G, \psi}^1(\Pi_1, \Pi_2)$ est de dimension finie.
\end{lemma}
\begin{proof}[Preuve] En appliquant le foncteur $\mathrm{Hom}(\overline{\Pi}_1, -)$ à la suite
	$$0\to \Pi_2^+\xrightarrow{\times \pi_L}\Pi_2^+\to\overline{\Pi}_2\to0,$$
	on obtient une suite exacte
	$$0\to\mathrm{Ext}_{G, \psi}^1(\overline{\Pi}_1, \Pi_2^+)\to\mathrm{Ext}_{G, \psi}^1(\overline{\Pi}_1, \overline{\Pi}_2)\to\mathrm{Ext}_{G, \psi}^2(\overline{\Pi}_1, \Pi_2^+)\to0,$$
	donc les espaces $\mathrm{Ext}_{G, \psi}^i(\overline{\Pi}_1, \Pi_2^+)$ sont de dimension finie pour $i=1, 2$.
	
	En appliquant le foncteur $\mathrm{Hom}_G(-, \Pi_2^+)$ à la suite
	$$0\to \Pi_1^+\xrightarrow{\times \pi_L}\Pi_1^+\to\overline{\Pi}_1\to0,$$
	on obtient une suite exacte
	$$\mathrm{Ext}_{G, \psi}^1(\overline{\Pi}_1, \Pi_2^+)\to\mathrm{Ext}_{G, \psi}^1(\Pi_1^+, \Pi_2^+)\to\mathrm{Ext}_{G, \psi}^1(\Pi_1^+, \Pi_2^+)\to\mathrm{Ext}_{G, \psi}^2(\overline{\Pi}_1, \Pi_2^+),$$
	l'espace $\kappa_L\otimes_{\mathscr{O}_L}\mathrm{Ext}_{G, \psi}^1(\Pi_1^+, \Pi_2^+)$ est donc de dimension finie. Comme $\mathrm{Ext}_{G, \psi}^1(\Pi_1^+, \Pi_2^+)$ est $\pi_L$-adiquement séparé, il suit d'une version du lemme de Nakayama \cite[Theorem 8.4]{matsumura1989commutative} que l'espace $\mathrm{Ext}_{G, \psi}^1(\Pi_1^+, \Pi_2^+)$ est de type fini. Comme l'application $L$-linéaire $$L\otimes_{\mathscr{O}_L} \mathrm{Ext}_{G, \psi}^1(\Pi_1^+, \Pi_2^+)\to\mathrm{Ext}_{G, \psi}^1(\Pi_1, \Pi_2)$$ est surjective, on déduit que l'espace $\mathrm{Ext}_{G, \psi}^1(\Pi_1, \Pi_2)$ est de dimension finie, comme on voulait.
\end{proof}
\begin{corollary}\label{3.8.10} L'espace $\mathrm{Ext}^1_{G, \psi}(\Pi_ {M, \mathscr{L}, j}, \widehat{\mathrm{LL}(M)})$ est de dimension finie sur $L$.
\end{corollary}
\begin{proof}[Preuve] Il suffit de vérifier les conditions (i) et (ii) du Lemme \ref{3.7.9} pour $\Pi_1:=\Pi_{M, \mathscr{L}, j}$ et $\Pi_2:=\widehat{\mathrm{LL}(M)}$. La condition (i) est garanti par la \cite[Proposition 4.2.2]{deg2023localization}, il nous reste donc à montrer (ii). La suite exacte $0\to N_{\mathscr{L}, j}\to\widehat{\mathrm{LL}(M)}\to\Pi_{M, \mathscr{L}, j}\to0$ nous donne une suite exacte de leurs boules unité.
	$$0\to N_{\mathscr{L}, j}^+\to\widehat{\mathrm{LL}(M)}^+\to\Pi_{M, \mathscr{L}, j}^+\to0.$$
	En appliquant le foncteur $\mathrm{Hom}_G(-, \widehat{\mathrm{LL}(M)}^+)$, on obtient
	\begin{align*}0\to&\mathrm{Hom}_G(\Pi_{M, \mathscr{L}, j}^+, \widehat{\mathrm{LL}(M)}^+)\to\mathrm{Hom}_G(\widehat{\mathrm{LL}(M)}^+, \widehat{\mathrm{LL}(M)}^+)\to\mathrm{Hom}_G(N_{\mathscr{L}, j}^+, \widehat{\mathrm{LL}(M)}^+)\\ \to&\mathrm{Ext}_{G, \psi}^1(\Pi_{M, \mathscr{L}, j}^+, \widehat{\mathrm{LL}(M)}^+)\to\mathrm{Ext}_{G, \psi}^1(\widehat{\mathrm{LL}(M)}^+, \widehat{\mathrm{LL}(M)}^+). 
	\end{align*}
	Comme $\mathrm{Ext}_{G, \psi}^1(\widehat{\mathrm{LL}(M)}^+, \widehat{\mathrm{LL}(M)}^+)$ est de $p^N$-torsion et $\mathrm{Hom}_G(N_{\mathscr{L}, j}^+, \widehat{\mathrm{LL}(M)}^+)$ est $\pi_L$-adiquement séparé, il suit du Lemme \ref{3.7.8} que l'espace $\mathrm{Ext}_{G, \psi}^1(\Pi_{M, \mathscr{L}, j}^+, \widehat{\mathrm{LL}(M)}^+)$ est $\pi_L$-adiquement séparé, ce que l'on voulait.
\end{proof}
 			\section{Foncteurs de Langlands catégoriques}  
Désormais, on va utiliser le produit tensoriel complété $\hat{\otimes}$, comme défini dans l'annexe \ref{D}. Remarquons que l'espace sous-jacent à une représentation localement analytique de $G$ est de type compact, et un tel espace est en particulier un espace LB. De plus, son dual fort est un espace de Fréchet.
\subsection{Foncteur localement analytique} 
\subsubsection{Construction} \label{section4.1.1}
Dans cette section, on va définir un "ersatz" du foncteur de Langlands catégorique en version localement analytique. Le choix d'une base $e_1, e_2$ de $M_{\mathrm{dR}}$ fournit un isomorphisme $\mathbf{P}(M_{\mathrm{dR}})\cong\mathbf{P}^1$. La droite correspondante à $z\neq\infty$ est $\mathscr{L}(z):=L\cdot(e_1+ze_2)$. La droite correspondante à $z=\infty$ est $\mathscr{L}(\infty):=Le_2$. En vertu du Lemme \ref{2.3.4} (iii), on a $\mathrm{Hom}_G(\mathrm{LL}(M), \Omega^1[M]^*)=M_{\mathrm{dR}}$, donc $e_1, e_2$ déterminent deux morphismes $\mathrm{LL}(M)\to\Omega^1[M]^*$ que l'on note abusivement encore $e_1, e_2$. La Proposition \ref{3.8.1} (ii) nous permet de définir les morphismes $e_1, e_2$ dans le cas de Banach.

\begin{definition} On définit un faisceau $\underline{\Omega}$ sur $\mathbf{P}^1$. Soit $U\subseteq\mathbf{P}^1$ un ouvert affinoïde, on note $\underline{\Omega}(U)$ le sous-$\mathscr{O}_{\mathbf{P}^1}(U)$-module de $\mathscr{O}_{\mathbf{P}^1}(U)\hat{\otimes}_L\Omega^1[M]$ défini par 
	$$\{\omega\in\mathscr{O}_{\mathbf{P}^1}(U)\hat{\otimes}_L\Omega^1[M]\mid\pi_{\mathrm{dR}}(\omega(z))\in \mathscr{L}(z)^\perp\hat{\otimes}_L\mathrm{LL}(M)^*,\ \forall z\in U\},$$
	où $\pi_{\mathrm{dR}}$ est le morphisme $\Omega^1[M]\to H^1_{\mathrm{dR}}[M]\cong M_{\mathrm{dR}}^*\otimes_L\mathrm{LL}(M)^*$. 
\end{definition}
\begin{remark} Dans la définition précédente, on peut interpréter un élément $\omega$ de $\mathscr{O}_{\mathbf{P}^1}(U)\hat{\otimes}_L\Omega^1[M]$ comme une fonction définie sur $U$ à valeur dans $\Omega^1[M]$. L'espace $\mathscr{O}_{\mathbf{P}^1}(U)\hat{\otimes}_L\Omega^1[M]$ est muni d'une action de $G$ induite par l'action de $G$ sur $\Omega^1[M]$, et on vérifie sans mal que $\underline{\Omega}(U)$ est une sous-$G$-représentation de $\mathscr{O}_{\mathbf{P}^1}(U)\hat{\otimes}_L\Omega^1[M]$.
\end{remark}
Si $U$ est un ouvert affinoïde de $\mathbf{P}^1$, on peut choisir la base $e_1, e_2$ de telle sorte que l'image de $U$ dans $\mathbf{P}^1$ ne contienne pas $\infty$. Si $\mathscr{L}(U)$ est le sous-$\mathscr{O}_{\mathbf{P}^1}(U)$-module de $\mathscr{O}_{\mathbf{P}^1}(U)\otimes_L M_{\mathrm{dR}}$ engendré par $e_1+ze_2$, alors le quotient $(\mathscr{O}_{\mathbf{P}^1}(U)\otimes_L M_{\mathrm{dR}})/\mathscr{L}(U)$ est isomorphe à $\mathscr{O}_{\mathbf{P}^1}(U)$ par $\lambda e_2\mapsto\lambda$. De plus, on note $\mathscr{L}(U)^\perp$ le sous-$\mathscr{O}_{\mathbf{P}^1}(U)$-module de $\mathscr{O}_{\mathbf{P}^1}(U)\otimes_L M_{\mathrm{dR}}$ engendré par $(e_1+ze_2)^\perp$.

Par définition, on a un diagramme commutatif à lignes exactes
\begin{equation*} \xymatrix@R=5mm@C=4mm{
		0\ar[r]& \underline{\Omega}(U)\ar[r]\ar[d]&\mathscr{O}_{\mathbf{P}^1}(U)\hat{\otimes}_L\Omega^1[M]\ar[r]\ar[d]& (\mathscr{O}_{\mathbf{P}^1}(U)\otimes_L M_{\mathrm{dR}}^*/\mathscr{L}(U)^\perp)\hat{\otimes}_L\mathrm{LL}(M)^*\ar[r]\ar@{=}[d]&0\\
		0 \ar[r] & \mathscr{L}(U)^\perp\hat{\otimes}_L\mathrm{LL}(M)^*\ar[r]  &(\mathscr{O}_{\mathbf{P}^1}(U)\otimes_L M_{\mathrm{dR}}^*)\hat{\otimes}_L\mathrm{LL}(M)^*\ar[r] &(\mathscr{O}_{\mathbf{P}^1}(U)\otimes_L M_{\mathrm{dR}}^*/\mathscr{L}(U)^\perp)\hat{\otimes}_L\mathrm{LL}(M)^*\ar[r] &0 .
	}
\end{equation*}
Soit $\pi$ une représentation de Fréchet de $G$ sur $L$, alors on a les isomorphismes
\begin{align*}\mathrm{Hom}_{\mathscr{O}_{\mathbf{P}^1}(U)}(\mathscr{O}_{\mathbf{P}^1}(U)\hat{\otimes}_L\pi, \mathscr{O}_{\mathbf{P}^1}(U))&=\mathrm{Hom}_{\mathscr{O}_{\mathbf{P}^1}(U)}(\mathscr{O}_{\mathbf{P}^1}(U), \mathrm{Hom}_L(\pi, \mathscr{O}_{\mathbf{P}^1}(U)))\\&=\mathrm{Hom}_L(\pi, \mathscr{O}_{\mathbf{P}^1}(U))\\&=\mathscr{O}_{\mathbf{P}^1}(U)\hat{\otimes}_L\pi^*.
\end{align*}
Il en résulte que, en appliquant $\mathrm{Hom}_{\mathscr{O}_{\mathbf{P}^1}(U)}(-, \mathscr{O}_{\mathbf{P}^1}(U))$ au diagramme commutatif ci-dessus, on obtient le diagramme commutatif à lignes exactes suivant
\begin{equation}\label{equation1}  \xymatrix@R=5mm@C=4mm{
		0& \underline{\Omega}(U)^\diamond\ar[l]&\mathscr{O}_{\mathbf{P}^1}(U)\hat{\otimes}_L\Omega^1[M]^*\ar[l]& \mathscr{L}(U)\hat{\otimes}_L\mathrm{LL}(M)\ar[l]&0\ar[l]\\
		0  &(\mathscr{O}_{\mathbf{P}^1}(U)\otimes_L M_{\mathrm{dR}}/\mathscr{L}(U))\hat{\otimes}_L\mathrm{LL}(M)\ar[l] \ar[u] &(\mathscr{O}_{\mathbf{P}^1}(U)\otimes_L M_{\mathrm{dR}})\hat{\otimes}_L\mathrm{LL}(M)\ar[l]\ar[u] &\mathscr{L}(U)\hat{\otimes}_L\mathrm{LL}(M)\ar[l] \ar@{=}[u]&0. \ar[l]   
	} 
\end{equation}
En utilisant le diagramme commutatif (\ref{equation4}), le diagramme commutatif ci-dessus induit un diagramme commutatif à lignes exactes
\begin{equation}\label{equation10}  \xymatrix@R=5mm@C=4mm{
		0\ar[r]& (\mathscr{O}_{\mathbf{P}^1}(U)\otimes_LM_{\mathrm{dR}})\hat{\otimes}_L\mathrm{LL}(M)\ar[r]\ar[d]&\mathscr{O}_{\mathbf{P}^1}(U)\hat{\otimes}_L\Omega^1[M]^*\ar[r]\ar[d]& \mathscr{O}_{\mathbf{P}^1}(U)\hat{\otimes}_L\mathscr{O}[M]^*\ar[r]\ar@{=}[d]&0\\
		0 \ar[r] & ((\mathscr{O}_{\mathbf{P}^1}(U)\otimes_LM_{\mathrm{dR}})/\mathscr{L}(U))\hat{\otimes}_L\mathrm{LL}(M)\ar[r]  &\underline{\Omega}(U)^\diamond\ar[r] &\mathscr{O}_{\mathbf{P}^1}(U)\hat{\otimes}_L\mathscr{O}[M]^*\ar[r] &0,
	}
\end{equation}
où $\underline{\Omega}(U)^\diamond$ est le $\mathscr{O}_{\mathbf{P}^1}(U)$-dual de $\underline{\Omega}(U)$. 

Fixons $\mathscr{L}\in U$. En tensorisant la première ligne du diagramme (\ref{equation1}) avec $\mathscr{O}_{\mathbf{P}^1}(U)/(z-z(\mathscr{L}))$ où $z(\mathscr{L})$ est la constante associée à $\mathscr{L}$, on obtient une suite exacte
$$0\to\mathscr{L}\hat{\otimes}_L\mathrm{LL}(M)\to\Omega^1[M]^*\to\mathscr{O}_{\mathbf{P}^1}(U)/(z-z(\mathscr{L}))\hat{\otimes}_L\underline{\Omega}(U)^\diamond\to0,$$
ce qui nous donne une suite exacte  
$$0\to\underline{\Omega}(U)^\diamond\xrightarrow{z-z(\mathscr{L})}\underline{\Omega}(U)^\diamond\to\Pi_{M, \mathscr{L}}^{\mathrm{an}}\to0.$$
En passant aux complétés unitaires universels, on obtient une surjection
$\widehat{\underline{\Omega}(U)^\diamond}\to\Pi_{M, \mathscr{L}}\to0$. En particulier, on a $\widehat{\underline{\Omega}(U)^\diamond}\neq0$. 
\begin{lemma}\label{torsionfree1} Si $U$ est un ouvert affinoïde de $\mathbf{P}^1$, alors le $\mathscr{O}_{\mathbf{P}^1}(U)$-module $\underline{\Omega}(U)^\diamond$ est sans torsion.
\end{lemma}
\begin{proof}[Preuve] La première ligne du diagramme (\ref{equation1}) implique que $\underline{\Omega}(U)^\diamond$ est le conoyau du morphisme $ze_1-e_2: \mathscr{L}(U)\hat{\otimes}_L\mathrm{LL}(M)\to\mathscr{O}_{\mathbf{P}^1}(U)\hat{\otimes}_L\Omega^1[M]^*$. Il s'agit donc de prouver que si $\lambda\in\bar{L}$, et si $x\in\mathscr{L}(U)\hat{\otimes}_L\mathrm{LL}(M)$ tel qu'il existe $v\in\mathscr{O}_{\mathbf{P}^1}(U)\hat{\otimes}_L\Omega^1[M]^*$ avec $(z-\lambda)v=(ze_1-e_2)x$, alors il existe $x'\in\mathscr{L}(U)\hat{\otimes}_L\mathrm{LL}(M)$ tel que $x=(z-\lambda)x'$, ce qui implique que $v=(ze_1-e_2)x'\in\mathscr{L}(U)\hat{\otimes}_L\mathrm{LL}(M)$.
	
	On peut écrire $x$ sous la forme $x_0+(z-\lambda)x_1$ avec $x_0, x_1\in\mathrm{LL}(M)$, et on peut supposer que $x=x_0\in\mathrm{LL}(M)$. Dans ce cas, on a $(z-\lambda)v=(ze_1-e_2)x_0$. On évalue $z$ en $\lambda$, alors on a $(\lambda e_1-e_2)x_0=0$. Comme le morphisme $\lambda e_1-e_2: \mathrm{LL}(M)\to\Omega^1[M]^*$ est injectif (Corollaire \ref{2.1.3}), on déduit que $x_0=0$. Cela permet de conclure.
\end{proof}
En tensorisant la deuxième ligne du diagramme (\ref{equation10}) avec $\mathscr{O}_{\mathbf{P}^1}(U)^*$ au-dessus de $\mathscr{O}_{\mathbf{P}^1}(U)$ et en utilisant l'isomorphisme
$$\mathscr{O}_{\mathbf{P}^1}(U)^*\hat{\otimes}_{\mathscr{O}_{\mathbf{P}^1}(U)}\underline{\Omega}(U)^\diamond\cong\underline{\Omega}(U)^*$$ 
donné par $f\otimes g\mapsto g\circ f$ où $g\in \mathscr{O}_{\mathbf{P}^1}(U)^*$ et $f\in\underline{\Omega}(U)^\diamond$, on déduit une suite exacte 
\begin{equation}	0\to\mathscr{O}_{\mathbf{P}^1}(U)^*\otimes_{\mathscr{O}_{\mathbf{P}^1}(U)}((\mathscr{O}_{\mathbf{P}^1}(U)\otimes_LM_{\mathrm{dR}})/\mathscr{L}(U))\hat{\otimes}_L\mathrm{LL}(M)\to\underline{\Omega}(U)^*\to\mathscr{O}_{\mathbf{P}^1}(U)^*\hat{\otimes}_L\mathscr{O}[M]^*\to0. \label{equation2}
\end{equation}

Soit $\pi$ un $D(G)$-module topologique. Pour tout $i\geq0$, on définit le faisceau $\mathbf{m}^i(\pi)$ sur $\mathbf{P}^1$ par $$\mathbf{m}^i(\pi)(U):=\mathrm{Ext}_{G, \psi}^i(\pi, \underline{\Omega}(U)^*)^*,$$
où $U$ est un ouvert affinoïde dans $\mathbf{P}^1$. Notons que, pour tout $i$, $\mathbf{m}^i(\pi)$ est complètement décrit par les restrictions au cercle unité et aux boules unité centrées en $0$ et en $\infty$.
\begin{lemma} \label{nouveau2} Si $U$ est un ouvert affinoïde de $\mathbf{P}^1$ et $\mathscr{L}\in U$, alors on a une suite exacte
	\begin{equation*}
		0\to\Pi_{M, \mathscr{L}}^{\mathrm{an}}\to\underline{\Omega}(U)^*\xrightarrow{z-z(\mathscr{L})}\underline{\Omega}(U)^*\to0. 
	\end{equation*}
\end{lemma} 
\begin{proof}[Preuve] La suite exacte (\ref{equation2}) induit, en passant aux duaux, un diagramme commutatif à lignes exactes
	\begin{equation*} \xymatrix@R=5mm@C=4mm{
			0\ar[r]& \mathscr{O}_{\mathbf{P}^1}(U)\hat{\otimes}_L\mathscr{O}[M]\ar[r]\ar[d]_{z-z(\mathscr{L})}&\underline{\Omega}(U)\ar[r]\ar[d]_{z-z(\mathscr{L})}& \mathscr{L}(U)^\perp\hat{\otimes}_L\mathrm{LL}(M)^*\ar[r]\ar[d]_{z-z(\mathscr{L})}&0\\
			0 \ar[r] & \mathscr{O}_{\mathbf{P}^1}(U)\hat{\otimes}_L\mathscr{O}[M]\ar[r]  &\underline{\Omega}(U)\ar[r] &\mathscr{L}(U)^\perp\hat{\otimes}_L\mathrm{LL}(M)^*\ar[r] &0 .
		}
	\end{equation*}
	Comme le morphisme $z-z(\mathscr{L})$ est injectif et d'image fermée sur $\mathscr{O}_{\mathbf{P}^1}(U)$ et que l'on a $$\mathscr{O}_{\mathbf{P}^1}(U)/(z-z(\mathscr{L}))=L,$$ 
	il suit du lemme de serpent et du Théorème \ref{2.1.2} que l'on a 
	$$(\underline{\Omega}(U)^*[\mathfrak{m}_{\mathscr{L}}])^*=\underline{\Omega}(U)/\mathfrak{m}_{\mathscr{L}}=\Pi_{M, \mathscr{L}}^{\mathrm{an}, *},$$
	d'où le résultat. 	 
\end{proof}
Pour tout $j$, on a une suite exacte $$0\to\underline{\Omega}(U)/\mathfrak{m}_{\mathscr{L}}^j\xrightarrow{z-z(\mathscr{L})}\underline{\Omega}(U)/\mathfrak{m}_{\mathscr{L}}^{j+1}\to\underline{\Omega}(U)/\mathfrak{m}_{\mathscr{L}}\to0.$$ 
On en déduit que, pour tout $j$, $\underline{\Omega}(U)^*[\mathfrak{m}_{\mathscr{L}}^j]$ est une extension successive de $\Pi_{M, \mathscr{L}}^{\mathrm{an}}$ de longueur $j$.
\begin{lemma}\label{split} Si $U$ est un ouvert affinoïde de $\mathbf{P}^1$ et $\mathscr{L}\in U$, alors la représentation $\underline{\Omega}(U)/\mathfrak{m}_{\mathscr{L}}^2$ est une extension non scindée de $\Pi_{M, \mathscr{L}}^{\mathrm{an}, *}$ par $\Pi_{M, \mathscr{L}}^{\mathrm{an}, *}$.
\end{lemma}
\begin{proof}[Preuve] Soit $\epsilon=z-z(\mathscr{L})$ le paramètre local en $\mathscr{L}$ associé à la base $e_1, e_2$ de $M_{\mathrm{dR}}$. Notons $x_1, x_2: \Omega^1[M]\to\mathrm{LL}(M)^*$ les coordonnées de $\pi_{\mathrm{dR}}(\omega)$ dans la base $e_1, e_2$, alors $\underline{\Omega}(U)/\mathfrak{m}_{\mathscr{L}}^2$ est le noyau du morphisme $L[\epsilon]/\epsilon^2\otimes_L\Omega^1[M]\to(L[\epsilon]/\epsilon^2)\otimes_L\mathrm{LL}(M)^*$ donné par $$\omega\mapsto x_1(\omega)-z(\mathscr{L})x_2(\omega)-\epsilon x_2(\omega).$$
	Supposons que l'extension soit scindée, alors $\underline{\Omega}(U)/\mathfrak{m}_{\mathscr{L}}^2$ est le sous-espace $(L[\epsilon]/\epsilon^2)\otimes_L\Pi_{{M, \mathscr{L}}}^{\mathrm{an}, *}$ puisque $\mathrm{Hom}_G(\Pi_{{M, \mathscr{L}}}^{\mathrm{an}, *}, \Omega^1[M])$ est de dimension $1$ (Proposition \ref{2.3.7}). Comme $\Pi_{{M, \mathscr{L}}}^{\mathrm{an}, *}\subseteq\Omega^1[M]$ est le noyau de $\omega\mapsto x_1(\omega)-z(\mathscr{L})x_2(\omega)$, on a $(x_1-zx_2)(\omega)=-\epsilon x_2(\omega)$. Pour obtenir une contradiction, il nous suffit donc de prouver que $x_2(\omega)$ n'est pas identiquement nul sur $\Pi_{{M, \mathscr{L}}}^{\mathrm{an}, *}$, ce qui est clair car, sinon $x_1(\omega)=x_2(\omega)=0$, et on aurait $\Pi_{{M, \mathscr{L}}}^{\mathrm{an}, *}\subseteq\mathscr{O}[M]$, ce qui n'est pas. Cela permet de conclure.
\end{proof}
\begin{corollary}\label{analytique} Si $U$ est la boule unité de centre $0$ ou $\infty$ ou le cercle unité et $\mathscr{L}\in U$, alors on a une suite exacte $$0\to\Pi_{M, \mathscr{L}, j}^{\mathrm{an}}\to\underline{\Omega}(U)^*\xrightarrow{(z-z(\mathscr{L}))^j}\underline{\Omega}(U)^*\to0.$$
\end{corollary}
\begin{proof}[Preuve] La discussion sous le Corollaire \ref{4.2.2} ci-dessous montre qu'il existe une injection
	$$\underline{\Omega}(U)^*[\mathfrak{m}_{\mathscr{L}}^j]\hookrightarrow(\mathscr{O}_{\mathbf{P}^1}(U)^*\hat{\otimes}_{\mathscr{O}_{\mathbf{P}^1}(U)}\widehat{\underline{\Omega}(U)^\diamond})[\mathfrak{m}_{\mathscr{L}}^j]^{\mathrm{an}}.$$
	Le Corollaire \ref{Banach} ci-dessous implique que $(\mathscr{O}_{\mathbf{P}^1}(U)^*\hat{\otimes}_{\mathscr{O}_{\mathbf{P}^1}(U)}\widehat{\underline{\Omega}(U)^\diamond})[\mathfrak{m}_{\mathscr{L}}^j]^{\mathrm{an}}=\Pi_{M, \mathscr{L}, j}^{\mathrm{an}}$. Comme les deux côtés de l'injection ont la même longueur, on a en fait un isomorphisme
	$$\underline{\Omega}(U)^*[\mathfrak{m}_{\mathscr{L}}^j]=\Pi_{M, \mathscr{L}, j}^{\mathrm{an}}.$$
	Cela permet de conclure.
\end{proof}
\begin{corollary}\label{gratte1} Le faisceau $\mathbf{m}^i(\Pi_{M, \mathscr{L}, j}^{\mathrm{an}})$ est supporté en $\mathscr{L}$ pour tout $i\geq0$ et $j\geq1$.
\end{corollary}
\begin{proof}[Preuve] Il suit du Corollaire \ref{analytique} que l'on a une suite exacte
	$$0\to\mathbf{m}^{i+1}(\underline{\Omega}(U)^*)/(z-z(\mathscr{L}))^j\to\mathbf{m}^i(\Pi_{M, \mathscr{L}, j}^{\mathrm{an}})\to\mathbf{m}^i(\underline{\Omega}(U)^*)[(z-z(\mathscr{L}))^j]\to0.$$
	Cela permet de conclure car les deux termes à droite et à gauche sont supportés en $\mathscr{L}$. 
\end{proof}
\subsubsection{Faisceaux associés à $\Pi_ {M, \mathscr{L},j}^{\mathrm{an}}$, $\Omega^1[M]^*$, $\mathrm{LL}(M)$ et $\mathscr{O}[M]^*$}
\begin{proposition}\label{4.1.4} Pour tout $j$, le faisceau $\mathbf{m}^0(\Pi_ {M, \mathscr{L}, j}^{\mathrm{an}})$ est isomorphe au faisceau gratte-ciel $\mathscr{O}_{\mathbf{P}^1, \mathscr{L}}/\mathfrak{m}_{\mathscr{L}}^j$ concentré en le point correspondant à $\mathscr{L}$.
\end{proposition}
\begin{proof}[Preuve] En appliquant le foncteur $\mathrm{Hom}_G(\Pi_ {M, \mathscr{L}}^{\mathrm{an}},-)^*$ à la suite (\ref{equation2}), on obtient une suite exacte 
	\begin{align*}&\mathrm{Hom}_G(\Pi_ {M, \mathscr{L}}^{\mathrm{an}},\mathscr{O}_{\mathbf{P}^1}(U)^*\hat{\otimes}_L\mathscr{O}[M]^*)^*\to\mathbf{m}^0(\Pi_ {M, \mathscr{L}}^{\mathrm{an}})(U)\to
		\\ &\mathrm{Hom}_G(\Pi_ {M, \mathscr{L}}^{\mathrm{an}}, \mathscr{O}_{\mathbf{P}^1}(U)^*\otimes_{\mathscr{O}_{\mathbf{P}^1}(U)} (\mathscr{O}_{\mathbf{P}^1}(U)\otimes_LM_{\mathrm{dR}})/\mathscr{L}(U))\hat{\otimes}_L\mathrm{LL}(M))^*\to0.
	\end{align*}
	Il découle du l'isomorphisme $\mathrm{Hom}_G(\Pi_ {M, \mathscr{L}}^{\mathrm{an}}, \mathscr{O}[M]^*)=L$ et du Lemme \ref{2.3.4} (iv) que l'on a une suite exacte	$$\mathscr{O}_{\mathbf{P}^1}(U)\to\mathbf{m}^0(\Pi_ {M, \mathscr{L}}^{\mathrm{an}})(U)\to0,$$
	ce qui implique que $\mathbf{m}^0(\Pi_ {M, \mathscr{L}}^{\mathrm{an}})(U)$ est un quotient de $\mathscr{O}_{\mathbf{P}^1}(U)$. Notons que l'on peut extraire $\mathscr{O}_{\mathbf{P}^1}(U)^*\otimes_{\mathscr{O}_{\mathbf{P}^1}(U)}((\mathscr{O}_{\mathbf{P}^1}(U)\otimes_LM_{\mathrm{dR}})/\mathscr{L}(U))$ et $\mathscr{O}_{\mathbf{P}^1}(U)^*$ de $\mathrm{Hom}_G(-, -)$ car l'action de $G$ est triviale sur ces espaces. D'après le Corollaire \ref{gratte1}, le faisceau $\mathbf{m}^0(\Pi_ {M, \mathscr{L}}^{\mathrm{an}})$ est tué par $\mathfrak{m}_{\mathscr{L}}^s$ pour $s\gg0$. Comme $\Pi_ {M, \mathscr{L}, j}^{\mathrm{an}}$ est une extension successive de $\Pi_ {M, \mathscr{L}}^{\mathrm{an}}$, le faisceau $\mathbf{m}^0(\Pi_ {M, \mathscr{L}, j}^{\mathrm{an}})$ est une extension successive de $\mathscr{O}_{\mathbf{P}^1, \mathscr{L}}/\mathfrak{m}_{\mathscr{L}}$. D'une part, en appliquant $\mathrm{Hom}_G(\Pi_{M, \mathscr{L}, j}^{\mathrm{an}}, -)^*$ à la suite du Lemme \ref{nouveau2}, on obtient une suite exacte
	$$\mathbf{m}^0(\Pi_ {M, \mathscr{L}, j}^{\mathrm{an}})\xrightarrow{z-z(\mathscr{L})}\mathbf{m}^0(\Pi_ {M, \mathscr{L}, j}^{\mathrm{an}})\to L\to0,$$ 
	ce qui montre que $\mathbf{m}^0(\Pi_ {M, \mathscr{L}, j}^{\mathrm{an}})\cong\mathscr{O}_{\mathbf{P}^1, \mathscr{L}}/\mathfrak{m}_{\mathscr{L}}^k$ pour un certain $k$. D'autre part, en appliquant $\mathrm{Hom}_G(\Pi_{M, \mathscr{L}, j}^{\mathrm{an}}, -)^*$ à la suite du Corollaire \ref{analytique} et en utilisant le Lemme \ref{2.3.4} (v), on a 
	$$\mathbf{m}^0(\Pi_ {M, \mathscr{L}, j}^{\mathrm{an}})\xrightarrow{(z-z(\mathscr{L}))^n}\mathbf{m}^0(\Pi_ {M, \mathscr{L}, j}^{\mathrm{an}})\to L[T]/T^{\min(n, j)}\to0$$
	pour tout $n\geq1$. On en conclut que $k=j$, ce qui achève la démonstration.
\end{proof}    
\begin{proposition}\label{4.1.5} On a $\mathbf{m}^0(\Omega^1[M]^*)=\mathscr{O}_{\mathbf{P}^1}$.
\end{proposition}
\begin{proof}[Preuve] En utilisant le foncteur $\mathrm{Hom}_G(\Omega^1[M]^*,-)^*$ à la suite (\ref{equation2}), on obtient une suite exacte
	\begin{align*}	&\mathrm{Ext}_{G, \psi}^1(\Omega^1[M]^*,  \mathscr{O}_{\mathbf{P}^1}(U)^*\otimes_{\mathscr{O}_{\mathbf{P}^1}(U)}(\mathscr{O}_{\mathbf{P}^1}(U)\otimes_LM_{\mathrm{dR}}/\mathscr{L}(U))\hat{\otimes}_L\mathrm{LL}(M))^*\to\mathrm{Hom}_G(\Omega^1[M]^*, \mathscr{O}_{\mathbf{P}^1}(U)^*\hat{\otimes}_L\mathscr{O}[M]^*)^* \\&\to\mathbf{m}^0(\Omega^1[M]^*) \to\mathrm{Hom}_G(\Omega^1[M]^*, \mathscr{O}_{\mathbf{P}^1}(U)^*\otimes_{\mathscr{O}_{\mathbf{P}^1}(U)} (\mathscr{O}_{\mathbf{P}^1}(U)\otimes_LM_{\mathrm{dR}}/\mathscr{L}(U))\hat{\otimes}_L\mathrm{LL}(M))^*\to0.
	\end{align*}
	D'après le Lemme \ref{2.3.1} (v), le Lemme \ref{2.3.4} (ii) et le Corollaire \ref{2.3.5}, on obtient une suite exacte
	$$0\to\mathscr{O}_{\mathbf{P}^1}(U)\to\mathbf{m}^0(\Omega^1[M]^*)\to0\to0.$$
	On en déduit que $\mathbf{m}^0(\Omega^1[M]^*)=\mathscr{O}_{\mathbf{P}^1}$, ce qui permet de conclure.
\end{proof}
\begin{proposition}\label{4.1.6} On a $\mathbf{m}^0(\mathrm{LL}(M))=\mathscr{O}_{\mathbf{P}^1}(-1)$.
\end{proposition}
\begin{proof}[Preuve] En utilisant le foncteur $\mathrm{Hom}_G(\mathrm{LL}(M),-)^*$ à la suite (\ref{equation2}), on obtient une suite exacte
	\begin{align*}&	\mathrm{Hom}_G(\mathrm{LL}(M),\mathscr{O}_{\mathbf{P}^1}(U)^*\hat{\otimes}_L\mathscr{O}[M]^*)^*\to\mathbf{m}^0(\mathrm{LL}(M)) \to\\&  \mathrm{Hom}_G(\mathrm{LL}(M),\mathscr{O}_{\mathbf{P}^1}(U)^*\otimes_{\mathscr{O}_{\mathbf{P}^1}(U)}(\mathscr{O}_{\mathbf{P}^1}(U)\otimes_LM_{\mathrm{dR}}/\mathscr{L}(U))\hat{\otimes}_L\mathrm{LL}(M))^*\to0.
	\end{align*}
	D'après le Lemme \ref{2.3.1} (i), on obtient une suite exacte
	$$0\to\mathbf{m}^0(\mathrm{LL}(M))\to\mathscr{L}(U)^\perp\to0.$$
	Donc $\mathbf{m}^0(\mathrm{LL}(M))=\mathscr{L}(U)^\perp$.  
	
	On note $U_0\subseteq\mathbf{P}^1$ la boule unité de centre $0$ et $U_\infty\subseteq\mathbf{P}^1$ la boule unité de centre $\infty$. Alors on a $\mathscr{L}(U_0)=\mathscr{O}_{\mathbf{P}^1}(U_0)(e_1+ze_2)$ et $\mathscr{L}(U_\infty)=\mathscr{O}_{\mathbf{P}^1}(U_0)(z^{-1}e_1+e_2)$. Sur $U_0\cap U_\infty$, on a $e_1+ze_2=z(z^{-1}e_1+e_2)$ qui a un pôle simple en $\infty$. De plus, le faisceau tordu de Serre $\mathscr{O}_{\mathbf{P}^1}(-1)$ est défini en recollant $U_0$ et $U_\infty$ le long de $U_0\cap U_\infty$ via le morphisme $f\mapsto \frac{f}{z}$. On en déduit que $\mathbf{m}^0(\mathrm{LL}(M))=\mathscr{O}_{\mathbf{P}^1}(-1)$, c.q.f.d.
\end{proof}
\begin{proposition}\label{4.1.7} On a $\mathbf{m}^0(\mathscr{O}[M]^*)=0$.
\end{proposition}
\begin{proof}[Preuve] L'exactitude à droite de $\mathbf{m}^0$ nous donne une suite exacte
	$$\mathbf{m}^0(\mathrm{LL}(M))\to\mathbf{m}^0(\Pi^{\mathrm{an}}_{M, \mathscr{L}})\to\mathbf{m}^0(\mathscr{O}[M]^*)\to0.$$ 
	D'après la Proposition \ref{4.1.4} et la Proposition \ref{4.1.6}, le morphisme $\mathbf{m}^0(\mathrm{LL}(M))\to\mathbf{m}^0(\Pi^{\mathrm{an}}_{M, \mathscr{L}})$ est soit surjectif, soit nul. Supposons que ce morphisme soit nul, alors on a $\mathbf{m}^0(\mathscr{O}[M]^*)=\mathscr{O}_{\mathbf{P}^1, \mathscr{L}}/\mathfrak{m}_{\mathscr{L}}$ pour toute $\mathscr{L}$, ce qui est absurde. On en déduit que le morphisme $\mathbf{m}^0(\mathrm{LL}(M))\to\mathbf{m}^0(\Pi^{\mathrm{an}}_{M, \mathscr{L}})$ est surjectif, d'où le résultat cherché.
\end{proof}
\subsection{Foncteur de Banach}
\subsubsection{Construction}
Dans cette section, on va définir un "ersatz" du foncteur de Langlands catégorique en version de Banach. Soit $\Pi$ un $\Lambda(G)$-module topologique. Pour tout $i\geq0$, on définit le faisceau $\hat{\mathbf{m}}(\Pi)$ sur $\mathbf{P}^1$ par $$\hat{\mathbf{m}}^i(\Pi)(U):=\mathrm{Ext}_{G, \psi}^i(\Pi, \mathscr{O}_{\mathbf{P}^1}(U)^*\hat{\otimes}_{\mathscr{O}_{\mathbf{P}^1}(U)}\widehat{\underline{\Omega}(U)^\diamond})^*,$$
où $U$ est un ouvert affinoïde dans $\mathbf{P}^1$. Notons que, pour tout $i$, $\hat{\mathbf{m}}^i(\Pi)$ est complètement décrit par les restrictions au cercle unité et aux boules unité centrées en $0$ et en $\infty$.

Le diagramme commutatif (\ref{equation1}) induit, en prenant les complétés unitaires universels, un diagramme commutatif à lignes exactes
\begin{equation*} \xymatrix@R=5mm@C=4mm{
		0& \widehat{\underline{\Omega}(U)^\diamond}\ar[l]&\mathscr{O}_{\mathbf{P}^1}(U)\hat{\otimes}_L\widehat{\Omega^1[M]^*}\ar[l]& Z_1(U)\ar[l]&0\ar[l]\\
		0  & (\mathscr{O}_{\mathbf{P}^1}(U)\otimes_L M_{\mathrm{dR}}/\mathscr{L}(U))\hat{\otimes}_L\widehat{\mathrm{LL}(M)}\ar[l] \ar[u] &(\mathscr{O}_{\mathbf{P}^1}(U)\otimes_L M_{\mathrm{dR}})\hat{\otimes}_L\widehat{\mathrm{LL}(M)}\ar[l]\ar[u] &\mathscr{L}(U)\hat{\otimes}_L\widehat{\mathrm{LL}(M)}\ar[l] \ar[u] &0\ar[l]
	}
\end{equation*}
où on note $Z_1(U)$ le noyau de la flèche $\mathscr{O}_{\mathbf{P}^1}(U)\hat{\otimes}\widehat{\Omega^1[M]^*}\to\widehat{\underline{\Omega}(U)^\diamond}$. Il suit de la Proposition \ref{1.3.3} que la flèche $\mathscr{L}(U)\hat{\otimes}_L\widehat{\mathrm{LL}(M)}\to Z_1(U)$ est d'image dense.
\begin{lemma}\label{torsionfree2} Si $U$ est la boule unité de centre $0$ ou $\infty$ ou le cercle unité, alors le $\mathscr{O}_{\mathbf{P}^1}(U)$-module $\widehat{\underline{\Omega}(U)^\diamond}$ est sans torsion.
\end{lemma}
\begin{proof}[Preuve] La démonstration est pareille à celle du Lemme \ref{torsionfree1}. D'après le Lemme \ref{4.2.1} ci-dessous, $\widehat{\underline{\Omega}(U)^\diamond}$ est le conoyau du morphisme $ze_1-e_2: \mathscr{L}(U)\hat{\otimes}_L\widehat{\mathrm{LL}(M)} \to\mathscr{O}_{\mathbf{P}^1}(U)\hat{\otimes}_L\widehat{\Omega^1[M]^*}$. Il s'agit donc de prouver que si $\lambda\in\bar{L}$, et si $x\in\mathscr{L}(U)\hat{\otimes}_L\widehat{\mathrm{LL}(M)}$ tel qu'il existe $v\in\mathscr{O}_{\mathbf{P}^1}(U)\hat{\otimes}_L\widehat{\Omega^1[M]^*}$ avec $(z-\lambda)v=(ze_1-e_2)x$, alors il existe $x'\in\mathscr{L}(U)\hat{\otimes}_L\widehat{\mathrm{LL}(M)}$ tel que $x=(z-\lambda)x'$, ce qui implique que $v=(ze_1-e_2)x'\in\mathscr{L}(U)\hat{\otimes}_L\widehat{\mathrm{LL}(M)}$.
	
	On peut écrire $x$ sous la forme $x_0+(z-\lambda)x_1$ avec $x_0, x_1\in\widehat{\mathrm{LL}(M)}$, et on peut supposer que $x=x_0\in\widehat{\mathrm{LL}(M)}$. Dans ce cas, on a $(z-\lambda)v=(ze_1-e_2)x_0$. On évalue $z$ en $\lambda$, alors on a $(\lambda e_1-e_2)x_0=0$. Comme le morphisme $\lambda e_1-e_2: \widehat{\mathrm{LL}(M)}\to\widehat{\Omega^1[M]^*}$ est injectif (Théorème \ref{3.7.2}), on en déduit que $x_0=0$. Cela permet de conclure.
\end{proof}
\begin{lemma}\label{4.2.1} Si $U$ est la boule unité de centre $0$ ou $\infty$ ou le cercle unité, la flèche $\mathscr{L}(U)\hat{\otimes}_L\widehat{\mathrm{LL}(M)}\to Z_1(U)$ est un isomorphisme. En particulier, on a la suite exacte suivante
	$$0\to\mathscr{L}(U)\hat{\otimes}_L\widehat{\mathrm{LL}(M)}\to\mathscr{O}_{\mathbf{P}^1}(U)\hat{\otimes}_L\widehat{\Omega^1[M]^*}\to\widehat{\underline{\Omega}(U)^\diamond}\to0.$$
\end{lemma}
\begin{proof}[Preuve] Il suffit de montrer que la flèche $$ze_1-e_2: \mathscr{L}(U)\hat{\otimes}_L\widehat{\mathrm{LL}(M)} \to\mathscr{O}_{\mathbf{P}^1}(U)\hat{\otimes}_L\widehat{\Omega^1[M]^*}$$ est une isométrie car cela prouve que l'image est fermée. On note $\widehat{\mathrm{LL}(M)}^+$ la préimage de la boule unité $\widehat{\Omega^1[M]^*}^+$ de $\widehat{\Omega^1[M]^*}$ par $e_1$, alors la réduction $\bar{e}_1: \widehat{\mathrm{LL}(M)}^+/\pi_L\to\widehat{\Omega^1[M]^*}^+/\pi_L$ est injective. On peut multiplier $e_2$ par $\pi^n$ pour que l'image de $\widehat{\mathrm{LL}(M)}^+$ par $e_2$ soit contenue dans $\widehat{\Omega^1[M]^*}^+$.
	
	On est ramené à montrer que la flèche $z\bar{e}_1-\bar{e}_2$ est injective sur $$\kappa_L[z]\otimes_{\kappa_L} \widehat{\mathrm{LL}(M)}^+/\pi_L\cong(\widehat{\mathrm{LL}(M)}^+/\pi_L)[z]$$
	pour $U$ la boule unité de centre $0$ (pour celle de centre $\infty$, il suffit d'échanger $e_1$ et $e_2$)
	ou $$\kappa_L[T^{\pm1}]\otimes_{\kappa_L} \widehat{\mathrm{LL}(M)}^+/\pi_L\cong(\widehat{\mathrm{LL}(M)}^+/\pi_L)[T^{\pm1}]$$
	pour le cercle unité. On procède en supposant que $U$ est la boule unité de centre $0$, les preuves pour les autres cas étant similaires. Supposons que $\sum_{i=0}^k m_iz^i\in (\widehat{\mathrm{LL}(M)}^+/\pi_L)[z]$ tel que $(z\bar{e}_1-\bar{e}_2)(\sum_{i=0}^k m_iz^i)=0$, alors il nous faut montrer que $m_i=0$ pour tout $i\geq0$. Puisque
	$$\sum_{i=0}^k(\bar{e}_1(m_i)z-\bar{e}_2(m_i))z^i=0,$$
	on a 
	$$\sum_{i=0}^k\bar{e}_1(m_i)z^{i+1}=\sum_{i=0}^k\bar{e}_2(m_i)z^i.$$
	On en déduit que $\bar{e}_1(m_k)=0$ et que $\bar{e}_1(m_i)=\bar{e}_2(m_{i+1})$ pour tout $0\leq i\leq k-1$. Comme le morphisme $\bar{e}_1$ est injectif, on en déduit que $m_i=0$ pour tout $i$, ce que l'on voulait.
\end{proof}
\begin{corollary}\label{4.2.2} Si $U$ est la boule unité de centre $0$ ou $\infty$ ou le cercle unité, alors on a une suite exacte 
	\begin{equation}0\to\mathscr{O}_{\mathbf{P}^1}(U)^*\otimes_{\mathscr{O}_{\mathbf{P}^1}(U)}\mathscr{L}(U)\hat{\otimes}_L\widehat{\mathrm{LL}(M)}\to\mathscr{O}_{\mathbf{P}^1}(U)^*\hat{\otimes}_L\widehat{\Omega^1[M]^*}\to\mathscr{O}_{\mathbf{P}^1}(U)^*\hat{\otimes}_{\mathscr{O}_{\mathbf{P}^1}(U)}\widehat{\underline{\Omega}(U)^\diamond}\to0. \label{equation3}
	\end{equation}
\end{corollary}
\begin{proof}[Preuve] On déduit du Lemme \ref{torsionfree2} que $\widehat{\underline{\Omega}(U)^\diamond}$ est plat sur $\mathscr{O}_{\mathbf{P}^1}(U)$, alors le résultat suit du Lemme \ref{4.2.1}.
\end{proof}
Si $U$ est la boule unité de centre $0$ ou $\infty$ ou le cercle unité, alors on déduit du Corollaire \ref{4.2.2} un diagramme commutatif
\begin{equation*} \xymatrix@R=5mm@C=4mm{
		0\ar[r]& \mathscr{O}_{\mathbf{P}^1}(U)^*\otimes_{\mathscr{O}_{\mathbf{P}^1}(U)}\mathscr{L}(U)\hat{\otimes}_L\mathrm{LL}(M)\ar[r]\ar[d] &\mathscr{O}_{\mathbf{P}^1}(U)^*\hat{\otimes}_L\Omega^1[M]^*\ar[r]\ar[d] & \mathscr{O}_{\mathbf{P}^1}(U)^*\hat{\otimes}_{\mathscr{O}_{\mathbf{P}^1}(U)}\underline{\Omega}(U)^\diamond\ar[r]\ar[d] &0\\
		0 \ar[r] & \mathscr{O}_{\mathbf{P}^1}(U)^*\otimes_{\mathscr{O}_{\mathbf{P}^1}(U)}\mathscr{L}(U)\hat{\otimes}_L\widehat{\mathrm{LL}(M)}^{\mathrm{an}}\ar[r]  &\mathscr{O}_{\mathbf{P}^1}(U)^*\hat{\otimes}_L\widehat{\Omega^1[M]^*}^{\mathrm{an}}\ar[r] &\mathscr{O}_{\mathbf{P}^1}(U)^*\hat{\otimes}_{\mathscr{O}_{\mathbf{P}^1}(U)}\widehat{\underline{\Omega}(U)^\diamond}^{\mathrm{an}}.&
	}
\end{equation*}
Il suit de la Proposition \ref{3.2.3} et du Corollaire \ref{3.4.4} que les deux flèches verticales de gauche sont des isomorphismes, on a donc une injection $$\underline{\Omega}(U)^*=\mathscr{O}_{\mathbf{P}^1}(U)^*\hat{\otimes}_{\mathscr{O}_{\mathbf{P}^1}(U)}\underline{\Omega}(U)^\diamond\hookrightarrow\mathscr{O}_{\mathbf{P}^1}(U)^*\hat{\otimes}_{\mathscr{O}_{\mathbf{P}^1}(U)}\widehat{\underline{\Omega}(U)^\diamond}^{\mathrm{an}}.$$
\begin{lemma}\label{nouveau1} Si $U$ est la boule unité de centre $0$ ou $\infty$ ou le cercle unité et $\mathscr{L}\in U$, alors on a une suite exacte
	\begin{equation*}0\to\Pi_{M, \mathscr{L}}\to\mathscr{O}_{\mathbf{P}^1}(U)^*\hat{\otimes}_{\mathscr{O}_{\mathbf{P}^1}(U)}\widehat{\underline{\Omega}(U)^\diamond}\xrightarrow{z-z(\mathscr{L})}\mathscr{O}_{\mathbf{P}^1}(U)^*\hat{\otimes}_{\mathscr{O}_{\mathbf{P}^1}(U)}\widehat{\underline{\Omega}(U)^\diamond}\to0. \end{equation*}
\end{lemma}
\begin{proof}[Preuve] La suite exacte (\ref{equation3}) induit le diagramme commutatif
	\begin{equation*} \xymatrix@R=5mm@C=4mm{
			0\ar[r]& \mathscr{O}_{\mathbf{P}^1}(U)^*\otimes_{\mathscr{O}_{\mathbf{P}^1}(U)}\mathscr{L}(U)\hat{\otimes}_L\widehat{\mathrm{LL}(M)}\ar[r]\ar[d]_{z-z(\mathscr{L})}&\mathscr{O}_{\mathbf{P}^1}(U)^*\hat{\otimes}_L\widehat{\Omega^1[M]^*}\ar[r]\ar[d]_{z-z(\mathscr{L})}& \mathscr{O}_{\mathbf{P}^1}(U)^*\hat{\otimes}_{\mathscr{O}_{\mathbf{P}^1}(U)}\widehat{\underline{\Omega}(U)^\diamond}\ar[r]\ar[d]_{z-z(\mathscr{L})}&0\\
			0 \ar[r] & \mathscr{O}_{\mathbf{P}^1}(U)^*\otimes_{\mathscr{O}_{\mathbf{P}^1}(U)}\mathscr{L}(U)\hat{\otimes}_L\widehat{\mathrm{LL}(M)}\ar[r]  &\mathscr{O}_{\mathbf{P}^1}(U)^*\hat{\otimes}_L\widehat{\Omega^1[M]^*}\ar[r] &\mathscr{O}_{\mathbf{P}^1}(U)^*\hat{\otimes}_{\mathscr{O}_{\mathbf{P}^1}(U)}\widehat{\underline{\Omega}(U)^\diamond}\ar[r] &0.
		}
	\end{equation*}
	Le morphisme $z-z(\mathscr{L})$ est injectif et d'image fermée sur $\mathscr{O}_{\mathbf{P}^1}(U)$, donc $z-z(\mathscr{L})$ est surjectif sur $\mathscr{O}_{\mathbf{P}^1}(U)^*$. Comme on a $$\mathscr{O}_{\mathbf{P}^1}(U)^*[z-z(\mathscr{L})]=(\mathscr{O}_{\mathbf{P}^1}(U)/(z-z(\mathscr{L})))^*=L,$$
	le lemme du serpent nous donne une suite exacte
	$$0\to\mathscr{L}\otimes_L\widehat{\mathrm{LL}(M)}\to\widehat{\Omega^1[M]^*}\to(\mathscr{O}_{\mathbf{P}^1}(U)^*\hat{\otimes}_{\mathscr{O}_{\mathbf{P}^1}(U)}\widehat{\underline{\Omega}(U)^\diamond})[\mathfrak{m}_{\mathscr{L}}]\to0.$$
	Le Corollaire \ref{cor2} implique que $(\mathscr{O}_{\mathbf{P}^1}(U)^*\hat{\otimes}_{\mathscr{O}_{\mathbf{P}^1}(U)}\widehat{\underline{\Omega}(U)^\diamond})[\mathfrak{m}_{\mathscr{L}}]=\Pi_{M, \mathscr{L}'}$ pour une certaine $\mathscr{L}'$. Compte tenu de l'injection $\underline{\Omega}(U)^*[\mathfrak{m}_{\mathscr{L}}]\hookrightarrow(\mathscr{O}_{\mathbf{P}^1}(U)^*\hat{\otimes}_{\mathscr{O}_{\mathbf{P}^1}(U)}\widehat{\underline{\Omega}(U)^\diamond})[\mathfrak{m}_{\mathscr{L}}]$, le Lemme \ref{nouveau2} et le Lemme \ref{2.3.4} (i) impliquent que $\mathscr{L}=\mathscr{L}'$. Cela permet de conclure. 
\end{proof}
\begin{corollary}\label{Banach} Soient $U$ la boule unité de centre $0$ ou $\infty$ ou le cercle unité et $\mathscr{L}\in U$, alors on a une suite exacte $$0\to\Pi_{M, \mathscr{L}, j}\to\mathscr{O}_{\mathbf{P}^1}(U)^*\hat{\otimes}_{\mathscr{O}_{\mathbf{P}^1}(U)}\widehat{\underline{\Omega}(U)^\diamond}\xrightarrow{(z-z(\mathscr{L}))^j}\mathscr{O}_{\mathbf{P}^1}(U)^*\hat{\otimes}_{\mathscr{O}_{\mathbf{P}^1}(U)}\widehat{\underline{\Omega}(U)^\diamond}\to0.$$
\end{corollary}
\begin{proof}[Preuve] Considérons le diagramme commutatif
	\begin{equation*} \xymatrix@R=5mm@C=4mm{
			0\ar[r]& \mathscr{O}_{\mathbf{P}^1}(U)^*\otimes_{\mathscr{O}_{\mathbf{P}^1}(U)}\mathscr{L}(U)\hat{\otimes}_L\widehat{\mathrm{LL}(M)}\ar[r]\ar[d]_{(z-z(\mathscr{L}))^j}&\mathscr{O}_{\mathbf{P}^1}(U)^*\hat{\otimes}_L\widehat{\Omega^1[M]^*}\ar[r]\ar[d]_{(z-z(\mathscr{L}))^j}& \mathscr{O}_{\mathbf{P}^1}(U)^*\hat{\otimes}_{\mathscr{O}_{\mathbf{P}^1}(U)}\widehat{\underline{\Omega}(U)^\diamond}\ar[r]\ar[d]_{(z-z(\mathscr{L}))^j}&0\\
			0 \ar[r] & \mathscr{O}_{\mathbf{P}^1}(U)^*\otimes_{\mathscr{O}_{\mathbf{P}^1}(U)}\mathscr{L}(U)\hat{\otimes}_L\widehat{\mathrm{LL}(M)}\ar[r]  &\mathscr{O}_{\mathbf{P}^1}(U)^*\hat{\otimes}_L\widehat{\Omega^1[M]^*}\ar[r] &\mathscr{O}_{\mathbf{P}^1}(U)^*\hat{\otimes}_{\mathscr{O}_{\mathbf{P}^1}(U)}\widehat{\underline{\Omega}(U)^\diamond}\ar[r] &0.
		}
	\end{equation*}
	On déduit du lemme de serpent une suite exacte
	\begin{align*}0&\to(\mathscr{O}_{\mathbf{P}^1}(U)^*\otimes_{\mathscr{O}_{\mathbf{P}^1}(U)}\mathscr{L}(U))[\mathfrak{m}_{\mathscr{L}}^j]\otimes_L\widehat{\mathrm{LL}(M)}\to\mathscr{O}_{\mathbf{P}^1}(U)^*[\mathfrak{m}_{\mathscr{L}}^j]\otimes_L\widehat{\Omega^1[M]^*}\\&\xrightarrow{f}(\mathscr{O}_{\mathbf{P}^1}(U)^*\hat{\otimes}_{\mathscr{O}_{\mathbf{P}^1}(U)}\widehat{\underline{\Omega}(U)^\diamond})[\mathfrak{m}_{\mathscr{L}}^j]\to0.
	\end{align*}
	Or, il suit du Théorème \ref{3.7.2} que, pour toute autre droite $\mathscr{L}'$, on a une suite exacte
	$$0\to\mathscr{O}_{\mathbf{P}^1}(U)^*[\mathfrak{m}_{\mathscr{L}}^j]\otimes_L\widehat{\mathrm{LL}(M)}\xrightarrow{g} \mathscr{O}_{\mathbf{P}^1}(U)^*[\mathfrak{m}_{\mathscr{L}}^j]\otimes_L\Omega^1[M]^{\mathrm{b}, *}\to\mathscr{O}_{\mathbf{P}^1}(U)^*[\mathfrak{m}_{\mathscr{L}}^j]\otimes_L\Pi_{M, \mathscr{L}'}\to0.$$
	La surjection $f$ induit une surjection 
	$$\mathscr{O}_{\mathbf{P}^1}(U)^*[\mathfrak{m}_{\mathscr{L}}^j]\otimes_L\Pi_{M, \mathscr{L}'}=(\mathscr{O}_{\mathbf{P}^1}(U)^*[\mathfrak{m}_{\mathscr{L}}^j]\otimes_L\widehat{\Omega^1[M]^*})/\mathrm{Im}(g)\overset{\bar{f}}{\twoheadrightarrow}(\mathscr{O}_{\mathbf{P}^1}(U)^*\hat{\otimes}_{\mathscr{O}_{\mathbf{P}^1}(U)}\widehat{\underline{\Omega}(U)^\diamond})[\mathfrak{m}_{\mathscr{L}}^j]/\mathrm{Im}(f\circ g).$$
	Or, il suit du Lemme \ref{nouveau1} que $(\mathscr{O}_{\mathbf{P}^1}(U)^*\hat{\otimes}_{\mathscr{O}_{\mathbf{P}^1}(U)}\widehat{\underline{\Omega}(U)^\diamond})[\mathfrak{m}_{\mathscr{L}}^j]$ est une extension successive de $\Pi_{M, \mathscr{L}}$ de longueur $j$, ce qui implique que $(\mathscr{O}_{\mathbf{P}^1}(U)^*\hat{\otimes}_{\mathscr{O}_{\mathbf{P}^1}(U)}\widehat{\underline{\Omega}(U)^\diamond})[\mathfrak{m}_{\mathscr{L}}^j]/\mathrm{Im}(f\circ g)=0$. En d'autre termes, on a une surjection 
	$$f\circ g: \mathscr{O}_{\mathbf{P}^1}(U)^*[\mathfrak{m}_{\mathscr{L}}^j]\otimes_L\widehat{\mathrm{LL}(M)}\twoheadrightarrow(\mathscr{O}_{\mathbf{P}^1}(U)^*\hat{\otimes}_{\mathscr{O}_{\mathbf{P}^1}(U)}\widehat{\underline{\Omega}(U)^\diamond})[\mathfrak{m}_{\mathscr{L}}^j].$$
	Il suit du \cite[Théorème 4.1]{cdn2023correspondance} que $(\mathscr{O}_{\mathbf{P}^1}(U)^*\hat{\otimes}_{\mathscr{O}_{\mathbf{P}^1}(U)}\widehat{\underline{\Omega}(U)^\diamond})[\mathfrak{m}_{\mathscr{L}}^j]$ est isomorphe à une somme directe finie de $\Pi_{M, \mathscr{L}, k}$. 
	
	Montrons que $(\mathscr{O}_{\mathbf{P}^1}(U)^*\hat{\otimes}_{\mathscr{O}_{\mathbf{P}^1}(U)}\widehat{\underline{\Omega}(U)^\diamond})[\mathfrak{m}_{\mathscr{L}}^j]=\Pi_{M, \mathscr{L}, j}$ pour tout $j$ par récurrence. Pour $j=2$, supposons que $(\mathscr{O}_{\mathbf{P}^1}(U)^*\hat{\otimes}_{\mathscr{O}_{\mathbf{P}^1}(U)}\widehat{\underline{\Omega}(U)^\diamond})[\mathfrak{m}_{\mathscr{L}}^2]=\Pi_{M, \mathscr{L}}\oplus\Pi_{M, \mathscr{L}}$. Comme on a une injection
	$$\underline{\Omega}(U)^*=\mathscr{O}_{\mathbf{P}^1}(U)^*\hat{\otimes}_{\mathscr{O}_{\mathbf{P}^1}(U)}\underline{\Omega}(U)^\diamond\hookrightarrow\mathscr{O}_{\mathbf{P}^1}(U)^*\hat{\otimes}_{\mathscr{O}_{\mathbf{P}^1}(U)}\widehat{\underline{\Omega}(U)^\diamond}^{\mathrm{an}},$$
	on en déduit que $\underline{\Omega}(U)^*[\mathfrak{m}_{\mathscr{L}}^2]\hookrightarrow\Pi_{M, \mathscr{L}}\oplus\Pi_{M, \mathscr{L}}$. D'après le Lemme \ref{nouveau2}, $\underline{\Omega}(U)^*[\mathfrak{m}_{\mathscr{L}}^2]$ est une extension de $\Pi_{M, \mathscr{L}}^{\mathrm{an}}$ par $\Pi_{M, \mathscr{L}}^{\mathrm{an}}$, on a donc $\underline{\Omega}(U)^*[\mathfrak{m}_{\mathscr{L}}^2]=\Pi_{M, \mathscr{L}}^{\mathrm{an}}\oplus\Pi_{M, \mathscr{L}}^{\mathrm{an}}$, ce qui contredit le Lemme \ref{split}. Il en résulte que $(\mathscr{O}_{\mathbf{P}^1}(U)^*\hat{\otimes}_{\mathscr{O}_{\mathbf{P}^1}(U)}\widehat{\underline{\Omega}(U)^\diamond})[\mathfrak{m}_{\mathscr{L}}^2]=\Pi_{M, \mathscr{L}, 2}$.
	
	Supposons que l'isomorphisme est vrai pour un certain $j\geq2$. Comme on a une inclusion $$(\mathscr{O}_{\mathbf{P}^1}(U)^*\hat{\otimes}_{\mathscr{O}_{\mathbf{P}^1}(U)}\widehat{\underline{\Omega}(U)^\diamond})[\mathfrak{m}_{\mathscr{L}}^j]\hookrightarrow(\mathscr{O}_{\mathbf{P}^1}(U)^*\hat{\otimes}_{\mathscr{O}_{\mathbf{P}^1}(U)}\widehat{\underline{\Omega}(U)^\diamond})[\mathfrak{m}_{\mathscr{L}}^{j+1}],$$
	on en déduit que $(\mathscr{O}_{\mathbf{P}^1}(U)^*\hat{\otimes}_{\mathscr{O}_{\mathbf{P}^1}(U)}\widehat{\underline{\Omega}(U)^\diamond})[\mathfrak{m}_{\mathscr{L}}^{j+1}]=\Pi_{M, \mathscr{L}, j+1}$ ou $(\mathscr{O}_{\mathbf{P}^1}(U)^*\hat{\otimes}_{\mathscr{O}_{\mathbf{P}^1}(U)}\widehat{\underline{\Omega}(U)^\diamond})[\mathfrak{m}_{\mathscr{L}}^{j+1}]=\Pi_{M, \mathscr{L}, j}\oplus\Pi_{M, \mathscr{L}}$. Cependant, la composée $$(\mathscr{O}_{\mathbf{P}^1}(U)^*\hat{\otimes}_{\mathscr{O}_{\mathbf{P}^1}(U)}\widehat{\underline{\Omega}(U)^\diamond})[\mathfrak{m}_{\mathscr{L}}^j]\hookrightarrow(\mathscr{O}_{\mathbf{P}^1}(U)^*\hat{\otimes}_{\mathscr{O}_{\mathbf{P}^1}(U)}\widehat{\underline{\Omega}(U)^\diamond})[\mathfrak{m}_{\mathscr{L}}^{j+1}]\twoheadrightarrow\frac{\mathscr{O}_{\mathbf{P}^1}(U)^*\hat{\otimes}_{\mathscr{O}_{\mathbf{P}^1}(U)}\widehat{\underline{\Omega}(U)^\diamond})[\mathfrak{m}_{\mathscr{L}}^{j+1}]}{\mathscr{O}_{\mathbf{P}^1}(U)^*\hat{\otimes}_{\mathscr{O}_{\mathbf{P}^1}(U)}\widehat{\underline{\Omega}(U)^\diamond})[\mathfrak{m}_{\mathscr{L}}]}$$ n'est ni injective ni surjective, donc on a $$(\mathscr{O}_{\mathbf{P}^1}(U)^*\hat{\otimes}_{\mathscr{O}_{\mathbf{P}^1}(U)}\widehat{\underline{\Omega}(U)^\diamond})[\mathfrak{m}_{\mathscr{L}}^{j+1}]=\Pi_{M, \mathscr{L}, j+1}$$
	en remarquant que $\frac{\mathscr{O}_{\mathbf{P}^1}(U)^*\hat{\otimes}_{\mathscr{O}_{\mathbf{P}^1}(U)}\widehat{\underline{\Omega}(U)^\diamond})[\mathfrak{m}_{\mathscr{L}}^{j+1}]}{\mathscr{O}_{\mathbf{P}^1}(U)^*\hat{\otimes}_{\mathscr{O}_{\mathbf{P}^1}(U)}\widehat{\underline{\Omega}(U)^\diamond})[\mathfrak{m}_{\mathscr{L}}]}\cong(\mathscr{O}_{\mathbf{P}^1}(U)^*\hat{\otimes}_{\mathscr{O}_{\mathbf{P}^1}(U)}\widehat{\underline{\Omega}(U)^\diamond})[\mathfrak{m}_{\mathscr{L}}^j]$. Cela permet de conclure.
\end{proof}
\begin{corollary}\label{supporte2} Le faisceau $\hat{\mathbf{m}}^i(\Pi_{M, \mathscr{L}, j})$ est supporté en $\mathscr{L}$ pour tout $i\geq0$ et $j\geq1$.
\end{corollary}
\begin{proof}[Preuve] Il suit du Corollaire \ref{Banach} que l'on a une suite exacte
	$$0\to\hat{\mathbf{m}}^{i+1}(\mathscr{O}_{\mathbf{P}^1}(U)^*\hat{\otimes}_{\mathscr{O}_{\mathbf{P}^1}(U)}\widehat{\underline{\Omega}(U)^\diamond})/\mathfrak{m}_{\mathscr{L}}^j\to\hat{\mathbf{m}}^i(\Pi_{M, \mathscr{L}, j})\to\hat{\mathbf{m}}^i(\mathscr{O}_{\mathbf{P}^1}(U)^*\hat{\otimes}_{\mathscr{O}_{\mathbf{P}^1}(U)}\widehat{\underline{\Omega}(U)^\diamond})[\mathfrak{m}_{\mathscr{L}}^j]\to0.$$
	Le résultat se déduit du fait que les deux termes à droite et à gauche sont supportés en $\mathscr{L}$.
\end{proof}
\subsubsection{Faisceaux associés à $\widehat{\mathrm{LL}(M)}$, $\Pi_{M, \mathscr{L}, j}$ et $\widehat{\Omega^1[M]^*}$}
\begin{proposition}\label{4.2.4} On a $\hat{\mathbf{m}}^0(\widehat{\mathrm{LL}(M)})=\mathscr{O}_{\mathbf{P}^1}(-1)$.
\end{proposition}
\begin{proof}[Preuve] En combinant le Corollaire \ref{3.7.6} avec l'application du foncteur $\mathrm{Hom}_G(\widehat{\mathrm{LL}(M)},-)$ à la suite (\ref{equation3}) pour $U=U_0, U_\infty$ ou $U_0\cap U_\infty$, on obtient une suite exacte  
	$$0\to\mathscr{O}_{\mathbf{P}^1}(U)^*\hat{\otimes}_{\mathscr{O}_{\mathbf{P}^1}(U)}\mathscr{L}(U)\to\mathscr{O}_{\mathbf{P}^1}(U)^*\otimes_L M_{\mathrm{dR}}\to(\hat{\mathbf{m}}^0(\widehat{\mathrm{LL}(M)})(U))^*\to0.$$	 
	En dualisant, cela fournit la suite exacte
	$$0\to\hat{\mathbf{m}}^0(\widehat{\mathrm{LL}(M)}(U))\to\mathscr{O}_{\mathbf{P}^1}(U)\otimes_L M_{\mathrm{dR}}^*\to(\mathscr{O}_{\mathbf{P}^1}(U)\otimes_L M_{\mathrm{dR}}^*)/\mathscr{L}(U)^\perp\to0.$$
	On en déduit que $\hat{\mathbf{m}}^0(\widehat{\mathrm{LL}(M)}(U))=\mathscr{L}(U)^\perp$, ce qui implique que $\hat{\mathbf{m}}^0(\widehat{\mathrm{LL}(M)})=\mathscr{O}_{\mathbf{P}^1}(-1)$ d'après la preuve de la Proposition \ref{4.1.6}. Cela permet de conclure.
\end{proof}
\begin{proposition}\label{4.2.5} Pour tout $j$, le faisceau $\hat{\mathbf{m}}^0(\Pi_{M, \mathscr{L}, j})$ est isomorphe au faisceau gratte-ciel $\mathscr{O}_{\mathbf{P}^1, \mathscr{L}}/\mathfrak{m}_{\mathscr{L}}^j$ qui est concentré en le point correspondant à $\mathscr{L}$.
\end{proposition}
\begin{proof}[Preuve] En utilisant le foncteur $\mathrm{Hom}_G(\Pi_{M, \mathscr{L}},-)$ à la suite (\ref{equation3}) pour $U=U_0, U_\infty$ ou $U_0\cap U_\infty$, on obtient une suite exacte
	\begin{align*} 0&\to\hat{\mathbf{m}}^0(\Pi_{M, \mathscr{L}})(U)^*\to\mathscr{O}_{\mathbf{P}^1}(U)^*\otimes_{\mathscr{O}_{\mathbf{P}^1}(U)}\mathscr{L}(U)\otimes_L\mathrm{Ext}_{G, \psi}^1(\Pi_{M, \mathscr{L}}, \widehat{\mathrm{LL}(M)})\\&\to\mathscr{O}_{\mathbf{P}^1}(U)^*\otimes_L\mathrm{Ext}_{G, \psi}^1(\Pi_{M, \mathscr{L}}, \widehat{\Omega^1[M]^*}).
	\end{align*}
	Il suit du Corollaire \ref{3.8.10} que l'espace $\mathrm{Ext}_{G, \psi}^1(\Pi_{M, \mathscr{L}}, \widehat{\mathrm{LL}(M)})$ est de dimension finie, alors $\hat{\mathbf{m}}^0(\Pi_{M, \mathscr{L}})(U)$ est un quotient de $\mathscr{O}_{\mathbf{P}^1}(U)^{\oplus d}$. Il est donc supporté en un nombre fini de points et de degré fini, ou bien il est supporté sur $U$ tout entier.
	
	D'après le Corollaire \ref{supporte2}, le faisceau $\hat{\mathbf{m}}^0(\Pi_ {M, \mathscr{L}})$ est tué par $\mathfrak{m}_{\mathscr{L}}^s$ pour $s\gg0$. 
	Comme $\Pi_ {M, \mathscr{L}, j}$ est une extension successive de $\Pi_ {M, \mathscr{L}}$, le faisceau $\hat{\mathbf{m}}^0(\Pi_ {M, \mathscr{L}, j})$ est une extension successive de $\mathscr{O}_{\mathbf{P}^1, \mathscr{L}}/\mathfrak{m}_{\mathscr{L}}$. D'une part,en appliquant $\mathrm{Hom}_G(\Pi_{M, \mathscr{L}, j}^{\mathrm{an}}, -)^*$ à la suite du Lemme \ref{nouveau1}, on a 
	$$\hat{\mathbf{m}}^0(\Pi_ {M, \mathscr{L}, j})\xrightarrow{z-z(\mathscr{L})}\hat{\mathbf{m}}^0(\Pi_ {M, \mathscr{L}, j})\to L\to0,$$ ce qui implique que $\hat{\mathbf{m}}^0(\Pi_ {M, \mathscr{L}, j})=\mathscr{O}_{\mathbf{P}^1, \mathscr{L}}/\mathfrak{m}_{\mathscr{L}}^k$ pour un certain $k$. D'autre part, en appliquant $\mathrm{Hom}_G(\Pi_{M, \mathscr{L}, j}, -)^*$ à la suite du Corollaire \ref{Banach} et en utilisant la Proposition \ref{finalbanach}, on obtient
	$$\hat{\mathbf{m}}^0(\Pi_ {M, \mathscr{L}, j})\xrightarrow{(z-z(\mathscr{L}))^n}\hat{\mathbf{m}}^0(\Pi_ {M, \mathscr{L}, j})\to L[T]/T^{\min(n, j)}\to0$$
	pour tout $n\geq1$, ce qui entraîne que $k=j$, d'où le résultat.
\end{proof}
\begin{corollary}\label{4.2.6} On a  
	
	(i) $\hat{\mathbf{m}}^0(\widehat{\Omega^1[M]^*})=\mathscr{O}_{\mathbf{P}^1}$. 
	
	(ii) $\hat{\mathbf{m}}^0(H^1(\mathfrak{X}, \mathscr{O})[\tfrac{1}{p}]^*[M])=\mathscr{O}_{\mathbf{P}^1}(-2)$.
\end{corollary}
\begin{proof}[Preuve] (i) En appliquant le foncteur $\hat{\mathbf{m}}$ à la suite du Théorème \ref{3.7.2}, on obtient une suite exacte
	$$\hat{\mathbf{m}}^0(\widehat{\mathrm{LL}(M)})\to\hat{\mathbf{m}}^0(\widehat{\Omega^1[M]^*})\to\hat{\mathbf{m}}^0(\Pi_{M, \mathscr{L}})\to0.$$
	Comme le faisceau $\hat{\mathbf{m}}^0(\widehat{\Omega^1[M]^*})$ se surjecte sur $\hat{\mathbf{m}}^0(\Pi_{M, \mathscr{L}})$ pour un nombre infini de $\mathscr{L}$, on en déduit que $\hat{\mathbf{m}}^0(\widehat{\Omega^1[M]^*})$ est supporté en un nombre infini de points (Proposition \ref{4.2.5}). La Proposition \ref{4.2.4} nous permet donc d'obtenir la suite exacte suivante pour toute $\mathscr{L}$
	$$0\to\hat{\mathbf{m}}^0(\widehat{\mathrm{LL}(M)})\to\hat{\mathbf{m}}^0(\widehat{\Omega^1[M]^*})\to\hat{\mathbf{m}}^0(\Pi_{M, \mathscr{L}})\to0.$$
	Il existe au moins une droite $\mathscr{L}$ telle que la suite ci-dessus est non scindée, d'où le résultat.
	
	(ii) On déduit du Corollaire \ref{2.6.3} une suite exacte pour toute $\mathscr{L}$
	$$\hat{\mathbf{m}}^1(\Pi_{M, \mathscr{L}})\to\hat{\mathbf{m}}^0(H^1(\mathfrak{X}, \mathscr{O})[\tfrac{1}{p}]^*[M])\to\hat{\mathbf{m}}^0(\widehat{\mathrm{LL}(M)})\to\hat{\mathbf{m}}^0(\Pi_{M, \mathscr{L}})\to0.$$
	Supposons que le morphisme $$\hat{\mathbf{m}}^1(\Pi_{M, \mathscr{L}})\to\hat{\mathbf{m}}^0(H^1(\mathfrak{X}, \mathscr{O})[\tfrac{1}{p}]^*[M])$$ est non nul, alors le faisceau $\hat{\mathbf{m}}^0(H^1(\mathfrak{X}, \mathscr{O})[\tfrac{1}{p}]^*[M])$ est une somme directe d'un faisceau supporté en $\mathscr{L}$ et un fibré en droites (Proposition \ref{4.2.4}, Proposition \ref{4.2.5}). Comme ceci est vrai pour toute $\mathscr{L}$, on obtient une contradiction, ce qui implique que le morphisme $\hat{\mathbf{m}}^1(\Pi_{M, \mathscr{L}})\to\hat{\mathbf{m}}^0(H^1(\mathfrak{X}, \mathscr{O})[\tfrac{1}{p}]^*[M])$ est nul. Cela permet de conclure.
\end{proof}
\begin{remark} La preuve du Corollaire \ref{2.6.3} nous donne une suite exacte
	$$0\to H^1(\mathfrak{X}, \mathscr{O})[\tfrac{1}{p}]^*[M]\to M_{\mathrm{dR}} \otimes \widehat{\mathrm{LL}(M)}\to\Omega^1[M]^{\mathrm{b}, *} \to0.$$
	En appliquant $\hat{\mathbf{m}}^0$, on obtient la suite exacte classique
	$$0\to\mathscr{O}_{\mathbf{P}^1}(-2)\to\mathscr{O}_{\mathbf{P}^1}(-1)\oplus\mathscr{O}_{\mathbf{P}^1}(-1)\to\mathscr{O}_{\mathbf{P}^1}\to0.$$
\end{remark}
Pour une extension successive $\Pi$ de $\Pi_{M, \mathscr{L}}$ qui n'est pas de de Rham, on peut aussi calculer $\hat{\mathbf{m}}^0(\Pi)$. Pour simplifier, on suppose que $\Pi$ est de longueur $2$, alors on a le résultat suivant
\begin{proposition} Soit $\Pi$ une extension de $\Pi_{M, \mathscr{L}}$ par $\Pi_{M, \mathscr{L}}$ qui n'est pas de de Rham, alors on a $\hat{\mathbf{m}}^0(\Pi)=\mathscr{O}_{\mathbf{P}^1, \mathscr{L}}/\mathfrak{m}_{\mathscr{L}}$.
\end{proposition}
\begin{proof}[Preuve] D'après la preuve de la Proposition \ref{4.2.5}, $\hat{\mathbf{m}}^0(\Pi)$ est isomorphe à $\mathscr{O}_{\mathbf{P}^1, \mathscr{L}}/\mathfrak{m}_{\mathscr{L}}$ ou $\mathscr{O}_{\mathbf{P}^1, \mathscr{L}}/\mathfrak{m}_{\mathscr{L}}^2$ ou $\mathscr{O}_{\mathbf{P}^1, \mathscr{L}}/\mathfrak{m}_{\mathscr{L}}\oplus\mathscr{O}_{\mathbf{P}^1, \mathscr{L}}/\mathfrak{m}_{\mathscr{L}}$. D'une part, en appliquant $\mathrm{Hom}_G(\Pi, -)^*$ à la suite du Lemme \ref{nouveau1}, on a 
	$$\hat{\mathbf{m}}^0(\Pi)\xrightarrow{z-z(\mathscr{L})}\hat{\mathbf{m}}^0(\Pi)\to L\to0,$$ ce qui implique que $\hat{\mathbf{m}}^0(\Pi_ {M, \mathscr{L}, j})=\mathscr{O}_{\mathbf{P}^1, \mathscr{L}}/\mathfrak{m}_{\mathscr{L}}$ ou $\mathscr{O}_{\mathbf{P}^1, \mathscr{L}}/\mathfrak{m}_{\mathscr{L}}^2$. D'autre part, en appliquant $\mathrm{Hom}_G(\Pi, -)^*$ à la suite du Corollaire \ref{Banach}, on a 
	$$\hat{\mathbf{m}}^0(\Pi)\xrightarrow{(z-z(\mathscr{L}))^2}\hat{\mathbf{m}}^0(\Pi)\to L\to0,$$
	ce qui implique que $\hat{\mathbf{m}}^0(\Pi)=\mathscr{O}_{\mathbf{P}^1, \mathscr{L}}/\mathfrak{m}_{\mathscr{L}}$, d'où le résultat.
	\end{proof}
\section{Application}
\subsection{Complétions $\mathfrak{B}$-adiques des quotients de $\widehat{\mathrm{LL}(M)}$}
Soient $\mathfrak{B}$ un bloc et $\Pi_2$ un quotient de $\widehat{\mathrm{LL}(M)}$ par une sous-représentation propre fermée $\Pi_1$. Supposons que $\mathrm{LL}(M)=\mathrm{ind}_{KZ}^G\sigma_M$. On note $\sigma_M^0$ un réseau de $\sigma_M$. Comme $\mathrm{ind}_{KZ}^G\sigma_M^0$ est un réseau de type fini de $\mathrm{LL}(M)$, la complétion p-adique de $\mathrm{ind}_{KZ}^G\sigma_M^0$ est un réseau de $\widehat{\mathrm{LL}(M)}$. On note $\Pi_1^+$ la préimage et $\Pi_2^+$ l'image de la complétion p-adique de $\mathrm{ind}_{KZ}^G\sigma_M^0$ dans $\Pi_1$ et $\Pi_2$ respectivement, alors $\Pi_1^+, \Pi_2^+$ sont les boules unités de $\Pi_1$ et $\Pi_2$. Pour tout $k$, $\Pi_1^+/\pi_L^k$ et $\Pi_2^+/\pi_L^k$ sont des représentations lisses de type fini (voir \cite[Theorem 2.4.1]{deg2023localization} pour $\Pi_1^+/\pi_L^k$). D'après la section \ref{3.3}, on dispose des complétions $\mathfrak{B}$-adiques $(\Pi_1^+/\pi_L^k)_\mathfrak{B}$ et $(\Pi_2^+/\pi_L^k)_\mathfrak{B}$, et on définit la complétion $\mathfrak{B}$-adique de $\Pi_i$ par 
$$\Pi_{i, \mathfrak{B}}:=(\varprojlim_k(\Pi_i^+/\pi_L^k)_\mathfrak{B})[\tfrac{1}{\pi_L}]$$
pour $i=1,2$. Le Lemme \ref{exact} nous donne une suite exacte 
$$0\to\Pi_{1, \mathfrak{B}}\to\mathrm{LL}(M)_\mathfrak{B}\to\Pi_{2, \mathfrak{B}}\to0.$$
\begin{lemma}\label{quotient complétion} Soient $\sigma$ un poids de Serre et $W$ un quotient de $\mathrm{ind}_{KZ}^G\sigma$, alors $W$ est un $\kappa_L[T]$-module où $T$ est l'opérateur de Barthel-Livné. De plus, si $\mathfrak{B}$ est un bloc correspondant à un polynôme irréductible $P(T)\in\kappa_L[T]$, alors la complétion $\mathfrak{B}$-adique $W_\mathfrak{B}$ de $W$ est $\varprojlim\limits_iW/P(T)^i$. En particulier, le morphisme naturel $W\to W_\mathfrak{B}$ est injectif.
\end{lemma}
\begin{proof}[Preuve] Le premier énoncé suit du \cite[Corollary 2.1.4]{deg2023localization}. En remarquant que tout quotient de longueur finie de $W$ est un quotient de $\mathrm{ind}_{KZ}^G\sigma/f(T)$ pour un certain polynôme $f(T)\in\kappa_L[T]$ (\cite[Corollary 2.1.4]{deg2023localization}), la preuve reste est similaire à celle de \cite[Proposition 2.5]{cdn2023correspondance} dans le cas supersingulier et à celle de \cite[Proposition 2.6]{cdn2023correspondance} dans le cas non supersingulier.
\end{proof}
\begin{lemma}\label{5.1.2} (i) Soit $\Pi$ un quotient propre de $\widehat{\mathrm{LL}(M)}$, alors on a $\Pi\hookrightarrow\prod\limits_i\Pi_{\mathfrak{B}_i}$ pour un nombre fini de blocs $\mathfrak{B}_i$.
	
	(ii) Soient $\Pi$ un quotient propre de $\widehat{\mathrm{LL}(M)}$ et $\mathfrak{B}$ un bloc donné dans (i), alors on a $\Pi_\mathfrak{B} \hookrightarrow \prod\limits_{i, j} \Pi_{M, \mathscr{L}_i, s(i, j)}$ où $0\leq s(i, j)\leq j$.
\end{lemma}
\begin{proof}[Preuve] (i) On peut supposer que $\mathrm{LL}(M)=\mathrm{ind}_{KZ}^G\sigma_M$. La réduction modulo $p$ de $\sigma_M$ est une extension d'un nombre fini de poids de Serre que l'on note $\{\sigma_i\}_i$. Pour tout $i$, on a $\mathrm{End}_G(\sigma_i)=\kappa_L[T_i]$ et on note $\mathfrak{B}_i$ le bloc correspondant à la représentation $(\mathrm{ind}_{KZ}^G\sigma_i)/P_i(T_i)$, où $P_i(T_i)$ est un polynôme irréductible dans $\kappa_L[T_i]$. Il nous reste à montrer que toutes les flèches $$\Pi^+/\pi_L^k\to\prod_i(\Pi^+/\pi_L^k)_{\mathfrak{B}_i}$$ 
	sont injectives pour $k\geq1$. Par dévissage, il suffit de traiter le cas $k=1$. Il est clair que $\Pi^+/\pi_L$ est l'extension successive des quotients $W_i$ de $\mathrm{ind}_{KZ}^G\sigma_i$. Comme le foncteur de passage à la complétion $\mathfrak{B}$-adique est exact (Lemme \ref{exact}), on est ramené à montrer que le morphisme $W_i\to(W_i)_{\mathfrak{B}_i}$ est injectif pour tout $i$. Le Lemme \ref{quotient complétion} nous permet de conclure.
	
	(ii) D'après le \cite[Corollaire 5.5]{cdn2023correspondance}, $R_{M, \mathfrak{B}}$ est un produit fini d'anneaux principaux, on a donc une injection de $R_{M, \mathfrak{B}}$-modules $	\mathbf{V}(\Pi_{\mathfrak{B}})\hookrightarrow\prod_i\mathbf{V}(\Pi_{\mathfrak{B}})_{\mathfrak{m}_i}\hookrightarrow\prod_{i, j}	\mathbf{V}(\Pi_{\mathfrak{B}})/\mathfrak{m}_i^{j}$ où $\mathfrak{m}_i$ sont les idéaux maximals de $R_{M, \mathfrak{B}}$. Cela induit une injection de $R_{M, \mathfrak{B}}$-modules par fonctorialité 
	$$\Pi_{\mathfrak{B}}\hookrightarrow\prod_{i, j}\Pi_{\mathfrak{B}}/\mathfrak{m}_i^j.$$
	Comme $\Pi$ est un quotient de $\widehat{\mathrm{LL}(M)}$, on a une surjection $\Pi_{M, \mathscr{L}_i, j}\cong\mathrm{LL}(M)_{\mathfrak{B}}/\mathfrak{m}_i^j\twoheadrightarrow\Pi_{\mathfrak{B}}/\mathfrak{m}_i^j$, ce qui implique que $\Pi_{\mathfrak{B}}/\mathfrak{m}_i^j$ soit isomorphe à $\Pi_{M, \mathscr{L}_i, s(i, j)}$ pour $s(i, j)\leq j$, soit nulle. Cela permet de conclure.
\end{proof}
On déduit directement du Lemme \ref{5.1.2} le résultat suivant.
\begin{corollary}\label{5.1.3} Soit $\Pi$ un quotient propre de $\widehat{\mathrm{LL}(M)}$, alors on a une injection
	$$\Pi\hookrightarrow\prod_{i\in I}\prod_{j\in J(i)}\Pi_{M, \mathscr{L}_i, j}$$
	telle que la composée $\Pi\hookrightarrow\prod_{i\in I}\prod_{j\in J(i)}\Pi_{M, \mathscr{L}_i, j}\twoheadrightarrow \Pi_{M, \mathscr{L}_i, j}$ est surjective pour tout $i, j$. 
\end{corollary}
\begin{proof}[Preuve] Considérons le diagramme commutative suivant
	\begin{equation*} \xymatrix@R=5mm@C=4mm{
			\widehat{\mathrm{LL}(M)}\ar@{->>}[r]\ar@{->>}[d]&\Pi\ar[d]\\
			\mathrm{LL}(M)_\mathfrak{B}/\mathfrak{m}^i\ar@{->>}[r]  &\Pi_\mathfrak{B}/\mathfrak{m}^i ,
		}
	\end{equation*}
	on en déduit que le morphisme $\Pi\to\Pi_\mathfrak{B}/\mathfrak{m}^i$ est surjectif pour tout $i$.
\end{proof}
\begin{remark}\label{5.1.4} Si $J(i)$ est fini et non vide pour un certain $i$, alors on peut supposer que $|J(i)|=1$ en prenant $\max_{j\in J(i)}j$. 
\end{remark}
\subsection{Finitude des quotients propres de $\widehat{\mathrm{LL}(M)}$}
\begin{proposition}\label{5.2.2} Soit $\Pi_2$ un quotient de $\widehat{\mathrm{LL}(M)}$ par une sous-représentation propre fermée $\Pi_1$.
	
	(i) Supposons que $\Pi_1\neq0$, alors on a $\hat{\mathbf{m}}^0(\Pi_2)\neq\mathscr{O}_{\mathbf{P}^1}(-1)$.
	
	(ii) La représentation $\Pi_2$ admet un quotient $\Pi_{M, \mathscr{L}}$ pour une certaine $\mathscr{L}$. En particulier, on a $\hat{\mathbf{m}}^0(\Pi_2)\neq0$. 
\end{proposition}
\begin{proof}[Preuve] (i) Il suffit de montrer que $\hat{\mathbf{m}}^0(\Pi_2)\neq\hat{\mathbf{m}}^0(\widehat{\mathrm{LL}(M)})$. Par définition, il suffit de montrer que 
	$$\mathrm{Hom}_{G}(\Pi_2, \mathscr{O}_{\mathbf{P}^1}(U)^*\hat{\otimes}_{\mathscr{O}_{\mathbf{P}^1}(U)}\widehat{\underline{\Omega}(U)^\diamond})\neq\mathrm{Hom}_{G}(\widehat{\mathrm{LL}(M)}, \mathscr{O}_{\mathbf{P}^1}(U)^*\hat{\otimes}_{\mathscr{O}_{\mathbf{P}^1}(U)}\widehat{\underline{\Omega}(U)^\diamond})$$
	pour un certain ouvert $U$. On est ramené à montrer qu'il existe un morphisme dans $$\mathrm{Hom}_{G}(\widehat{\mathrm{LL}(M)}, \mathscr{O}_{\mathbf{P}^1}(U)^*\hat{\otimes}_{\mathscr{O}_{\mathbf{P}^1}(U)}\widehat{\underline{\Omega}(U)^\diamond})$$ qui ne se factorise pas par $\widehat{\mathrm{LL}(M)}\twoheadrightarrow\Pi_2$. Comme $\Pi_1\neq0$, il existe une droite $\mathscr{L}\in U_{M, \mathfrak{B}}$ telle que $\Pi_1\nsubseteq N_{\mathscr{L}, 1}$. La composée $\widehat{\mathrm{LL}(M)}\twoheadrightarrow\Pi_{M, \mathscr{L}}\hookrightarrow \mathscr{O}_{\mathbf{P}^1}(U)^*\hat{\otimes}_{\mathscr{O}_{\mathbf{P}^1}(U)}\widehat{\underline{\Omega}(U)^\diamond}$ nous donne le morphisme voulu car le noyau de la composée est $N_{\mathscr{L}, 1}$.
	
	(ii) Il découle du Corollaire \ref{5.1.3} qu'il existe une surjection $\Pi_2\to\Pi_{M, \mathscr{L}, j}$ pour une certaine $\mathscr{L}$ et un certain $j$. La Proposition \ref{4.2.5} et l'exactitude à droite de $\hat{\mathbf{m}}^0$ nous donnent le résultat voulu.
\end{proof}
\begin{corollary}\label{5.2.3} Soit $\Pi$ un quotient propre de $\widehat{\mathrm{LL}(M)}$, alors le faisceau $\hat{\mathbf{m}}^0(\Pi)$ est concentré en un nombre fini de points. 
\end{corollary}
\begin{proof}[Preuve] Comme le support d'un quotient propre de $\mathscr{O}_{\mathbf{P}^1}(-1)$ est un ensemble fini, le résultat se déduit de la Proposition \ref{5.2.2}.
\end{proof}
\begin{corollary}\label{intersection} Fixons $\mathscr{L}\in\mathbf{P}^1$. Supposons que $J(\mathscr{L})\subseteq\mathbb{N}^*$ est un ensemble infini, alors le morphisme naturel 
	$$\widehat{\mathrm{LL}(M)}\to\prod_{j\in J(\mathscr{L})}\Pi_{M, \mathscr{L}, j}$$
	est injectif.
\end{corollary}
\begin{proof}[Preuve] On note $N$ le noyau de $\widehat{\mathrm{LL}(M)}\hookrightarrow\prod\limits_{j\in J(\mathscr{L})}\Pi_{M, \mathscr{L}, j}$, alors on dispose des surjections $$\widehat{\mathrm{LL}(M)}\twoheadrightarrow\widehat{\mathrm{LL}(M)}/N\twoheadrightarrow\Pi_{M, \mathscr{L}, j}$$ pour tout $j\in J(\mathscr{L})$. Il suit de la Proposition \ref{4.2.4} et de la Proposition \ref{4.2.5} que l'on a des surjections
	$$\mathscr{O}_{\mathbf{P}^1}(-1)\twoheadrightarrow\hat{\mathbf{m}}^0(\widehat{\mathrm{LL}(M)}/N)\twoheadrightarrow\mathscr{O}_{\mathbf{P}^1, \mathscr{L}}/\mathfrak{m}_{\mathscr{L}}^j$$
	pour tout $j\in  J(\mathscr{L})$.
	Or cela implique que $\hat{\mathbf{m}}^0(\widehat{\mathrm{LL}(M)}/N)=\mathscr{O}_{\mathbf{P}^1}(-1)$. Compte-tenu de la Proposition \ref{5.2.2} (i), on a $N=0$, ce qui permet de conclure.
\end{proof}
\begin{lemma}\label{5.2.4} Soit $\Pi$ un quotient propre de $\widehat{\mathrm{LL}(M)}$. Si l'on a une injection $\Pi\hookrightarrow\prod\limits_{i\in I}\prod\limits_{j\in J(i)}\Pi_{M, \mathscr{L}_i, j}$ où $i\in I$ est un ensemble d'indices fini et $|J(i)|$ est fini pour tout $i\in I$ , alors $\Pi$ est fermée dans $\prod\limits_{i\in I}\prod\limits_{j\in J(i)}\Pi_{M, \mathscr{L}_i, j}$. En particulier, $\Pi$ est de longueur finie.
\end{lemma}
\begin{proof}[Preuve] Il est clair que l'adhérence $T$ de l'image de $\Pi$ dans $\prod\limits_{i\in I}\prod\limits_{j\in J(i)}\Pi_{M, \mathscr{L}_i, j}$ est résiduellement de longueur finie. Comme la composée $\widehat{\mathrm{LL}(M)}\twoheadrightarrow\Pi\to T$ est d'image dense, on déduit du Lemme \ref{2.2.2} que $\Pi$ est une sous-représentation fermée de $\prod\limits_{i\in I}\prod\limits_{j\in J(i)}\Pi_{M, \mathscr{L}_i, j}$. Cela permet de conclure.
\end{proof}
\begin{corollary}\label{5.2.5} Si $\Pi$ est un quotient propre de $\widehat{\mathrm{LL}(M)}$ de longueur infinie, alors $\hat{\mathbf{m}}^0(\Pi)$ est supporté en un nombre infini de points.
\end{corollary}
\begin{proof}[Preuve] Il suit du Corollaire \ref{5.1.3}, du Corollaire \ref{intersection} et du Lemme \ref{5.2.4} qu'il existe un nombre infini de droites $\mathscr{L}_i$ telles que l'on a une injection
	$$\Pi\hookrightarrow\prod_{i\in I}\prod_{j\in J(i)}\Pi_{M, \mathscr{L}_i, j}$$
	et que toute composée $\Pi\to \Pi_{M, \mathscr{L}_i, j}$ est surjective pour tout $i, j$. La Proposition \ref{4.2.5} et l'exactitude à droite de $\hat{\mathbf{m}}^0$ nous donnent le résultat voulu.
\end{proof}
\begin{corollary}\label{5.2.7} Tous les quotients propres de $\widehat{\mathrm{LL}(M)}$ sont de longueur finie.
\end{corollary}
\begin{proof}[Preuve] Combiner le Corollaire \ref{5.2.3} et le Corollaire \ref{5.2.5}.
\end{proof}
\begin{corollary}\label{5.2.8} Soit $\Pi$ un quotient propre de $\widehat{\mathrm{LL}(M)}$, alors on a 
	$$\Pi\cong\oplus_{i\in I}\Pi_{M, \mathscr{L}_i, j_i}$$
	où $I$ est un ensemble d'indices fini.
\end{corollary}
\begin{proof}[Preuve] Il suit du Corollaire \ref{5.1.3} et du Lemme \ref{5.2.4} qu'il existe une injection
	$$\Pi\hookrightarrow\prod\limits_{i\in I}\prod\limits_{j\in J(i)}\Pi_{M, \mathscr{L}_i, j}$$
	et que la composée $\Pi\hookrightarrow\prod\limits_{i\in I}\prod\limits_{j\in J(i)}\Pi_{M, \mathscr{L}_i, j}\twoheadrightarrow \Pi_{M, \mathscr{L}_i, j}$ est surjective pour tout $i, j$. D'après le Corollaire \ref{5.2.7}, $I$ est fini. Il résulte du Corollaire \ref{intersection} que $J(i)$ est fini pour tout $i\in I$. 
	
	Il suit du Lemme \ref{5.2.4} que $\Pi$ est fermée dans $\prod\limits_{i\in I}\prod\limits_{j\in J(i)}\Pi_{M, \mathscr{L}_i, j}$. D'après la Remarque \ref{5.1.4}, $\Pi$ est une sous-représentation fermée de $\prod\limits_{i\in I}\Pi_{M, \mathscr{L}_i, j_i}$. On déduit du Lemme \ref{3.4.0} que l'on a un isomorphisme $\Pi\cong\prod\limits_{i\in I}\Pi_{M, \mathscr{L}_i, j_i}$. Cela permet de conclure.
\end{proof}
Notons que le Corollaire \ref{5.2.8} est compatible avec le Lemme \ref{3.4.2}.
\begin{appendices}
\section{Courbes complètes}
On suppose que $\mathfrak{X}$ est un modèle semi-stable, $G$-équivariant d'une courbe rigide sur l'anneau des entiers $\mathscr{O}_K$ d'une extension finie $K$ de $\mathbb{Q}_p$. On note $\pi$ une uniformisante de $K$ et $\kappa$ son corps résiduel. Quitte à faire une extension totalement ramifiée de degré fini (degré $4$ suffit) de $K$ et éclater les points singuliers le nombre de fois qu'il faut, on peut supposer que les composantes irréductibles de la fibre spéciale $\mathfrak{X}_\kappa$ sont lisses et que deux composantes irréductibles s'intersectent en au plus un point. 

Soit $\Gamma$ le graphe dual de $\mathfrak{X}_\kappa$: l'ensemble $S$ de ses sommets est en bijection $s\mapsto \mathfrak{Y}_s$ avec celui des composantes irréductibles de $\mathfrak{X}_\kappa$, et l'ensemble $A$ de ses arêtes est en bijection $a\mapsto P_a$ avec celui des points singuliers de $\mathfrak{X}_\kappa$, l'arête $a\in A$ joint les sommets $s_1$ et $s_2$ si $P_a=\mathfrak{Y}_{s_1}\cap \mathfrak{Y}_{s_2}$. Pour le graphe $\Gamma$, on dispose des groupe de cohomologie $H^i(\Gamma, \kappa)$ pour $i=0, 1$, et de cohomologie à support compact $H^i_c(\Gamma, \kappa)$ pour $i=0, 1$. Si $X=A, S$, alors on note $\kappa^X$ l'espace des fonctions $\phi: X\to\kappa$ et $\kappa^{(X)}$ le sous-espace des fonctions à support fini. On dispose de deux applications
$\partial: \kappa^{(S)}\to\kappa^{(A)}, \partial^*: \kappa^A\to\kappa^S$. Alors on a $H^1(\Gamma, \kappa)=\mathrm{Coker}(\partial: \kappa^S\to\kappa^A)$ et $H_c^1(\Gamma, \kappa)=\mathrm{Coker}(\partial: \kappa^{(S)}\to\kappa^{(A)})$ et $H_c^1(\Gamma, \kappa)^*=\mathrm{Ker}(\partial^*: \kappa^A\to\kappa^S)$.

On munit $\Gamma$ d'une métrique et d'une orientation en munissant chaque arête d'un homéomorphisme sur $[0, \frac{1}{e}]$, où $e$ est l'indice de ramification absolu de $K$. Cela munit $S$ d'une distance: si $s_1, s_2\in S$, alors $ed(s_1, s_2)$ est le minimum de $n\geq0$ tels qu'il existe une chaîne $s=s_0, s_1,..., s_n=s'$ telle que $\mathfrak{Y}_{s_i}\cap \mathfrak{Y}_{s_{i+1}}\neq\emptyset$ pour tout $i$. Le groupe $G$ agit sur $S, A$ et $\Gamma$ de manière isométrique, et $G\backslash A, G\backslash S$ sont des ensembles finis. L'espace topologique $\Gamma$ est contractile, ce qui est équivalent à ce que $\Gamma$ est un arbre, en particulier, $H^1(\Gamma, \kappa)=0$. Par contre, l'espace $H_c^1(\Gamma, \kappa)^*$ n'est pas nul. Pour plus de détails, le lecteur se reportera à \cite{cdn2022cohomologie}.

Si la fibre générique $\mathfrak{X}_K$ est un affinoïde, on définit une suite décroissante $\mathfrak{X}=\mathfrak{X}^{(0)}\supset \mathfrak{X}^{(1)}\supset...$ d'ouverts de $\mathfrak{X}$ où $\mathfrak{X}^{(i+1)}$ est obtenu en retirant le bord $\partial \mathfrak{X}^{(i)}$ de $\mathfrak{X}^{(i)}$. Le graphe $\Gamma^{(i)}$ de $\mathfrak{X}^{(i)}$ est un sous-graphe de $\Gamma$ et $\Gamma^{(i)}=\{s\in\Gamma\mid d(s, \partial \mathfrak{X})\geq i\}$.
\begin{lemma} Soit $f\in\mathscr{O}(\mathfrak{X})$. Si $f$ s'annule sur chacune des composantes connexes de $\mathfrak{X}^{(i)}_\kappa$, alors $f$ est divisible par $\pi^i$ sur $\mathfrak{X}^{(i)}$.
\end{lemma} 
\begin{proof}[Preuve] Par récurrence, on est ramené au cas $i=1$. Soit $\bar{\mathfrak{X}}_\kappa^{(1)}$ l'adhérence de $\mathfrak{X}_\kappa^{(1)}$ dans $\mathfrak{X}_\kappa$. Alors $\bar{\mathfrak{X}}_\kappa^{(1)}$ est une réunion de composantes irréductibles compactes de $\mathfrak{X}_\kappa$. Il s'ensuit que $f\in\mathscr{O}(\mathfrak{X})$ est constante sur chacune des composantes connexes de $\bar{\mathfrak{X}}_\kappa^{(1)}$ et si $f$ s'annule sur chacune des composantes connexes de $\mathfrak{X}^{(1)}_\kappa$, alors $f$ est identiquement nulle sur $\bar{\mathfrak{X}}_\kappa^{(1)}$, et donc est divisible par $\pi$ sur $\mathfrak{X}^{(1)}$.
\end{proof} 
On peut associer à $X$ un graphe, à savoir son squelette adique $\Gamma^{\mathrm{ad}}$. On renvoie le lecteur à \cite[2.3.4]{cdn2022cohomologie} pour la définition d'un squelette adique.
\begin{definition} On dit que $X$ est complète si $\Gamma^{\mathrm{ad}}$ est un espace métrique complet.
\end{definition}
\begin{lemma}\label{1.3.2} Si $\mathfrak{X}$ est complète, alors
	
	(i) $\mathscr{O}(\mathfrak{X}_\kappa)=\kappa$.
	
	(ii) $\mathscr{O}(\mathfrak{X})=\mathscr{O}_K$.
	
	(ii) $H^1(\mathfrak{X}, \mathscr{O})$ est sans $\pi$-torsion.
\end{lemma}
\begin{proof}[Preuve] (i) Si $\mathfrak{X}$ est complète, les composantes irréductibles de la fibre spéciale sont des courbes compactes. Une fonction sur $\mathfrak{X}_\kappa$ est donc constante sur chacune de ces composantes, et donc constante en vertu de la connectivité de $\Gamma$. 
	
	(ii) Soit $f\in\mathscr{O}(\mathfrak{X})$, alors la réduction $\bar{f}$ de $f$ est une constante d'après (i). Si $c\in\mathscr{O}_K$ est un relèvement de $\bar{f}$, alors on peut appliquer ce qui précède à $\pi^{-1}(f-c)$, et réitérer pour en déduire que $f$ est constante modulo $\pi^n$ pour tout $n$. Comme $\mathscr{O}(\mathfrak{X})$ est complet, on a $\mathscr{O}(\mathfrak{X})=\mathscr{O}_K$.
	
	(iii) La suite exacte courte $$0\to\mathscr{O}\xrightarrow{\times \pi}\mathscr{O}\to\mathscr{O}/\pi\to0$$ induit une suite exacte longue
	$$0\to\mathscr{O}(\mathfrak{X})\to\mathscr{O}(\mathfrak{X})\to\mathscr{O}(\mathfrak{X}_\kappa)\to H^1(\mathfrak{X}, \mathscr{O})\to H^1(\mathfrak{X}, \mathscr{O}).$$
	Compte tenu des assertions (i) et (ii), la flèche $\mathscr{O}(\mathfrak{X})\to\mathscr{O}(\mathfrak{X}_\kappa)$ est surjective, donc $H^1(\mathfrak{X}, \mathscr{O})$ est sans $\pi$-torsion, ce que l'on voulait.
\end{proof} 
\begin{lemma} On a un isomorphisme $$\Omega^1(\mathfrak{X})/\pi^n\xrightarrow{\sim}H^0(\mathfrak{X}, \Omega^1/\pi^n)$$
	pour tout $n\geq1$. En particulier, l'application naturelle $\Omega^1(\mathfrak{X})/\pi\xrightarrow{\sim}\Omega^1(\mathfrak{X}_\kappa)$ est un isomorphisme en prenant $n=1$.
\end{lemma}
\begin{proof}[Preuve] La suite exacte courte
	$$0\to\Omega^1\xrightarrow{\times \pi^n}\Omega^1\to\Omega^1/\pi^n\to0$$
	induit une suite exacte longue
	$$0\to\Omega^1(\mathfrak{X})\to\Omega^1(\mathfrak{X})\to H^0(\mathfrak{X}, \Omega^1/\pi^n)\to H^1(\mathfrak{X}, \Omega^1)\to H^1(\mathfrak{X}, \Omega^1).$$
	Il suffit donc de prouver que $H^1(\mathfrak{X}, \Omega^1)$ n'a pas de $\pi^n$-torsion. Or, ce groupe est nul sauf si $\mathfrak{X}_K$ est compact, auquel cas il est isomorphe à $\mathscr{O}_K$. Cela permet de conclure.
\end{proof}
\begin{proposition} Si $\mathfrak{X}_K$ est une courbe complète, alors $\omega\mapsto\mathrm{Res}(\omega)$ induit une surjection $\Omega^1(\mathfrak{X})\to H_c^1(\Gamma, \mathscr{O}_K)^*$.
\end{proposition}
\begin{proof}[Preuve] Il suffit de prouver le résultat modulo $\pi$, et on est ramené à prouver que $\Omega^1(\mathfrak{X}_\kappa)\to H_c^1(\Gamma, \kappa)^*$ est surjectif, ce qui découle des définitions de l'application résidu et du groupe $H_c^1(\Gamma, \mathscr{O}_K)^*$.
\end{proof}
\begin{lemma}\label{A.0.7} On a 
	
	(i) $H^1(\mathfrak{X}, \Omega^1)=0$.
	
	(ii) La suite $0\to\Omega^1(\mathfrak{X})\to H_{\mathrm{dR}}^1(\mathfrak{X})\to H^1(\mathfrak{X}, \mathscr{O})\to0$ est exacte.
	
	(iii) $H_{\mathrm{dR}}^1(\mathfrak{X})$ est sans $\pi$-torsion.
\end{lemma}
\begin{proof}[Preuve] (i) Par dévissage, on est ramené à montrer que $H^1(\mathfrak{X}, \Omega^1/\pi)=0$. On peut munir $\mathfrak{X}$ avec la structure logarithmique induite par $\mathfrak{X}_\kappa$, ce qui induit une structure logarithmique sur $\mathfrak{X}_\kappa$. De plus, on note $Y_s^\times$ la courbe $Y_s$ munie de la structure logarithmique induite par $\mathfrak{X}_\kappa$. Alors on a une suite exacte
	$$0\to H^0(\mathfrak{X}, \Omega^1/\pi)\to\prod_{s\in S}H^0(Y_s^\times, \Omega^1/\pi)\to\prod_{a\in A}\kappa\to H^1(\mathfrak{X}, \Omega^1/\pi)\to\prod_{s\in S}H^1(Y_s^\times, \Omega^1/\pi).$$ 
	
	D'un coté, le conoyau de la flèche $H^0(\mathfrak{X}, \Omega^1/\pi)\to\prod_{s\in S}H^0(Y_s^\times, \Omega^1/\pi)$ est $\kappa^S$. Comme $\Gamma$ est un arbre, on a $H^1(\Gamma, \kappa)=0$. On en déduit que la flèche $H^1(\mathfrak{X}, \Omega^1/\pi)\to\prod_{s\in S}H^1(Y_s^\times, \Omega^1/\pi)$ est injective.
	
	D'autre coté, on a 
	$$H^1(Y_s^\times, \Omega^1/\pi)=H^1(U_s, \Omega^1/\pi)=0,$$ 
	où $U_s$ est l'ouvert de $Y_s$ sur lequel la log-structure est triviale (l'ouvert de lissité de $\mathfrak{X}$ intersecté avec $Y_s$).
	
	Par conséquent, on a $H^1(\mathfrak{X}, \Omega^1)=0$.
	
	(ii) On a une suite exacte
	$$0\to\Omega^1(\mathfrak{X})\to H_{\mathrm{dR}}^1(\mathfrak{X})\to H^1(\mathfrak{X}, \mathscr{O})\to H^1(\mathfrak{X}, \Omega^1).$$
	Or l'assertion (i) montre que le dernier terme est $0$, ce qui achève la démonstration.
	
	(iii) Au vu de la suite exacte de (ii), il suffit de montrer que $\Omega^1(\mathfrak{X})$ et $H^1(\mathfrak{X}, \mathscr{O})$ sont sans torsion. L'espace $\Omega^1(\mathfrak{X})$ est évidemment sans torsion, et le résultat pour $H^1(\mathfrak{X}, \mathscr{O})$ est montré dans le Lemme \ref{1.3.2}. Cela permet de conclure.
\end{proof}
Maintenant on suppose que $\mathfrak{X}$ est un modèle semi-stable, $G$-équivariant de $\Sigma_n$ sur l'anneau des entiers $\mathscr{O}_K$ d'une extension finie $K$ de $\mathbb{Q}_p$, alors on a le résultat suivant
\begin{corollary}\label{1.3.6} On a une suite exacte courte
	$$0\to H^1(\mathfrak{X}, \mathscr{O})[\tfrac{1}{ p}]^*\to \widehat{H_{\mathrm{dR}}^1(\Sigma_n)^*}\to \widehat{\Omega^1(\Sigma_n)^*}\to0,$$
	où $\widehat{H_{\mathrm{dR}}^1(\Sigma_n)^*}$ (resp. $\widehat{\Omega^1(\Sigma)^*}$) est le complété unitaire universel de $H_{\mathrm{dR}}^1(\Sigma_n)^*$ (resp. $\Omega^1(\Sigma_n)^*$).
\end{corollary}
\begin{proof}[Preuve] La suite du Lemme \ref{A.0.7} (ii) induit une suite exacte
	$$0\to H^1(\mathfrak{X}, \mathscr{O})[\tfrac{1}{ p}]^*\to H_{\mathrm{dR}}^1(\mathfrak{X})[\tfrac{1}{p}]^*\to \Omega^1(\mathfrak{X})[\tfrac{1}{p}]^*\to0.$$
	Le résultat se déduit du Lemme \ref{1.2.6} et de la Proposition \ref{1.2.5}.
\end{proof}
  \section{Représentations supercuspidales}\label{B}
Dans cette section, on considère la structure des représentations supercuspidales de $G$. Notons $I=\begin{psmallmatrix}
	\z_p^\times&\z_p\\
	p\z_p&\z_p^\times
\end{psmallmatrix}$ le sous-groupe d'Iwahori de $G$, alors le normalisateur $N(I)$ de $I$ dans $G$ est engendré par $I$ et $w_p=\begin{psmallmatrix}
	0&1\\
	p&0
\end{psmallmatrix}$. Il s'ensuit que le groupe $I$ est d'indice $2$ dans $N(I)$.

Pour les représentations supercuspidales irréductibles de $G$, on a le théorème de structure suivant
\begin{theorem}[{\cite[15.5]{bh2006the}}]\label{B.0.1} Soit $\pi$ une représentation supercuspidale admissible irréductible de $G$, alors il existe un sous groupe compact-modulo-centre $J$ de $G$ et une représentation irréductible de dimension finie $\sigma$ de $J$ tels que
	$$\pi\cong\mathrm{ind}_J^G\sigma.$$ 
\end{theorem}
À conjugaison près, il y a deux sous-groupes maximals compacts-modulo-centre de $G$, l'un, appelé non ramifié, est $KZ$, l'autre est $N(I)Z$, appelé ramifié. Ainsi on obtient le résultat suivant
\begin{corollary}\label{0.4} Toute représentation supercuspidale admissible irréductible de $G$ est l'induite soit de $KZ$, soit de $N(I)Z$.
\end{corollary}
\begin{remark} Notons que le Corollaire \ref{0.4} est vrai pour tout groupe $\mathrm{GL}_2(F)$, où $F$ est une extension finie de $\mathbb{Q}_p$. Pour la démonstration, voir \cite{kutzko1978on}.
\end{remark}
\begin{corollary}\label{classsuper} Soit $\pi$ une représentation supercuspidale admissible irréductible de $G$, alors il existe une représentation irréductible $\sigma$ de $KZ$ de dimension finie, telle que $$\mathrm{ind}_{KZ}^G\sigma=\pi$$  
	ou bien $$\mathrm{ind}_{KZ}^G\sigma=\pi\oplus(\pi\otimes\mu_{-1}),$$ 
	où $\mu_\lambda$ est le caractère de $\mathbb{Q}_p^*$ défini par $x\mapsto\lambda^{v_p(x)}$. 
\end{corollary}
\begin{proof}[Preuve] Soit $\pi$ une représentation supercuspidale irréductible de $G$. Il résulte du Théorème \ref{B.0.1} que l'on peut écrire 
	$$\pi=\mathrm{ind}_J^G\sigma.$$
	
	Si $J\subseteq KZ$, alors on a $$\pi=\mathrm{ind}_{KZ}^G\sigma_1$$ où $\sigma_1:=\mathrm{ind}_J^{KZ}\sigma$ est une représentation de dimension finie de $KZ$.
	
	Si $J\subseteq N(I)Z$, alors on a $\pi=\mathrm{ind}_{N(I)Z}^G(\mathrm{ind}_J^{N(I)Z}\sigma)=\mathrm{ind}_{N(I)Z}^G\sigma_2$ où $\sigma_2:=\mathrm{ind}_J^{N(I)Z}\sigma$. Il suit de la formule de projection que l'on a 
	$$\mathrm{ind}_{IZ}^{N(I)Z}(\sigma_2|_{IZ})=\sigma_2\otimes(1\oplus\mu_{-1})=\sigma_2\oplus\sigma_2\otimes\mu_{-1}.$$
	En prenant l'induite compacte $\mathrm{ind}_{N(I)Z}^G$, on déduit que 
	\begin{align*}\pi\oplus(\pi\otimes\mu_{-1})&=(\mathrm{ind}_{N(I)Z}^G\sigma_2)\oplus(\mathrm{ind}_{N(I)Z}^G(\sigma_2\otimes\mu_{-1}))=\mathrm{ind}_{N(I)Z}^G(\mathrm{ind}_{IZ}^{N(I)Z}(\sigma_2|_{IZ}))\\ &=\mathrm{ind}_{IZ}^G(\sigma_2|_{IZ})=\mathrm{ind}_{KZ}^G(\mathrm{ind}_{IZ}^{KZ}(\sigma_2|_{IZ})). 
	\end{align*}
	Cela permet de conclure.
\end{proof}

   \section{Groupes analytiques $p$-adiques compacts FAb}\label{C}
Dans cette section, on montre quelques propriétés sur les groupes analytiques $p$-adique. En utilisant les résultats au groupe $\mathrm{SL}_2(\z_p)$, on peut montrer que l'abélianisé de tout sous-groupe ouvert de $\mathrm{SL}_2(\z_p)$ est fini.

L'un des résultats principaux de Lazard sur la structure des groupes analytiques $p$-adiques est le suivant
\begin{theorem}[{\cite[Theorem 8.32]{dss1999analytic}}]\label{C.0.1} Soit $G$ un groupe topologique, alors $G$ dispose d'une structure d'un groupe analytique $p$-adique si et seulement si $G$ contient un groupe pro-$p$ uniforme ouvert.
\end{theorem} 
Soit $G$ un groupe de Lie $p$-adique, alors il suit du \cite[Theorem 8.32]{dss1999analytic} que $G$ contient un sous-groupe pro-$p$ uniforme ouvert $U$. Donc on a une algèbre de $\z_p$-Lie $L_U:=\log(U)$. On peut définir l'algèbre de Lie $\mathcal{L}(G):=\mathbb{Q}_p\otimes_{\z_p}L_U$ de $G$ dont la définition est indépendante du choix de $U$. Notons que si $G$ est compact, alors $G$ est un groupe profini par définition. 
\begin{definition} Soit $G$ un groupe profini, alors on dit que $G$ est FAb si l'abélianisé de tout sous-groupe ouvert de $G$ est fini.
\end{definition}
\begin{remark} Notons que si $G$ est un groupe profini topologiquement de type fini, alors un sous-groupe $H\subseteq G$ est d'indice fini si et seulement si $H$ est ouvert. Pour la preuve, voir \cite{ns2003finite}.
\end{remark}
\begin{proposition}\label{C.0.4} Soit $U$ un groupe pro-$p$ uniforme, alors $U^{\mathrm{ab}}$ est fini si et seulement si $\mathcal{L}(U)$ est parfait, i.e., si $[\mathcal{L}(U), \mathcal{L}(U)]=\mathcal{L}(U)$. 
\end{proposition}
\begin{proof}[Preuve] Supposons que $U^{\mathrm{ab}}$ est infini, alors on a $U^{\mathrm{ab}}\cong (U^{\mathrm{ab}})_{\mathrm{tor}}\oplus\z_p^n$ pour quelque $n\geq1$. Il s'ensuit qu'il existe un sous-groupe distingué fermé $H$ de $U$ tel que $U/H\cong\z_p$. On déduit de la \cite[Proposition 4.31]{dss1999analytic} que le groupe $H$ est uniforme, que l'algèbre de $\z_p$-Lie $L_H$ est un idéal de $L_U$ et que $L_{U/H}\cong L_U/L_H$. Comme l'algèbre de $\z_p$-Lie $L_{U/H}\cong\z_p$ est commutatif, $L_H$ contient $[L_U, L_U]$, donc $L_U/[L_U, L_U]\twoheadrightarrow L_U/L_H\cong\z_p$. Donc $\mathcal{L}(U)$ n'est pas parfait.
	
	Réciproquement, supposons que $\mathcal{L}(U)$ n'est pas parfait, alors le $\z_p$-rang de $L_U/[L_U, L_U]$ est supérieur à $0$. On note $T/[L_U, L_U]$ la partie de torsion de $L_U/[L_U, L_U]$, alors $L_U/T\cong\z_p^n$ pour quelque $n$. Il suit de la \cite[Proposition 7.15]{dss1999analytic} que $T$ est un sous-groupe uniforme distingué fermé de $U$ et que $U/T$ est uniforme. Selon la \cite[Proposition 4.31]{dss1999analytic}, on a $L_{U/T}\cong L_U/L_T\cong\z_p^n$. Compte tenu du \cite[Theorem 9.10]{dss1999analytic}, on a $U/T\cong\z_p^n$. Il s'ensuit que $T\supseteq \overline{[U, U]}$ et que l'on a une surjection $U^{\mathrm{ab}}\twoheadrightarrow U/T\cong\z_p^n$. En particulier, $U^{\mathrm{ab}}$ est infini. Cela permet de conclure.
\end{proof}

\begin{corollary}\label{C.0.5} Soit $U$ un groupe pro-$p$ uniforme, alors $U^{\mathrm{ab}}$ est fini si et seulement si $U$ est FAb.
\end{corollary} 
\begin{proof}[Preuve] Une implication est triviale. Réciproquement, soit $K$ un sous-groupe ouvert de $U$, alors $K$ est un groupe de Lie $p$-adique. il suit du Théorème \ref{C.0.1} que $K$ contient un sous-groupe pro-$p$ uniforme ouvert $H$. Comme $H$ est aussi un sous-groupe uniforme ouvert de $U$, on a $\mathcal{L}(U)=\mathcal{L}(H)$. Supposons que $U^{\mathrm{ab}}$ est fini, alors il suit de la Proposition \ref{C.0.4} que $H^{\mathrm{ab}}$ est fini. Comme $[K: H]$ est fini, on déduit que $K^{\mathrm{ab}}$ est fini. Cela permet de conclure.
\end{proof}
\begin{corollary}\label{C.0.6} Soit $G$ un groupe analytique $p$-adique compact, alors $G$ est FAb si et seulement si $\mathcal{L}(G)$ est parfait, i.e., si $[\mathcal{L}(G), \mathcal{L}(G)]=\mathcal{L}(G)$. 
\end{corollary}
\begin{proof}[Preuve] Supposons que $\mathcal{L}(G)$ est parfait. Soit $K$ un sous-groupe ouvert de $G$, alors $K$ est un groupe de Lie $p$-adique. il suit du Théorème \ref{C.0.1} que $K$ contient un sous-groupe pro-$p$ uniforme ouvert $U$. Comme $\mathcal{L}(U)=\mathcal{L}(G)$ est parfait, il suit du Corollaire \ref{C.0.5} que $U^{\mathrm{ab}}$ est fini, donc $K^{\mathrm{ab}}$ l'est aussi.
	
	Réciproquement, supposons que $G$ est FAb. Comme $G$ est un groupe de Lie $p$-adique, il suit du Théorème \ref{C.0.1} que $G$ contient un sous-groupe pro-$p$ uniforme ouvert $U$. Comme $G$ est FAb, $U^{\mathrm{ab}}$ est fini. Il suit de la Proposition \ref{C.0.4} que $\mathcal{L}(G)=\mathcal{L}(U)$ est parfait. Cela permet de conclure.
\end{proof}
Comme l'algèbre de Lie de $\mathrm{SL}_2(\z_p)$ est parfait, on en déduit, à partir du Corollaire \ref{C.0.6}, le résultat suivant
\begin{corollary}\label{C.0.7} Le groupe analytique $p$-adique compact $\mathrm{SL}_2(\z_p)$ est FAb.
\end{corollary}
    \section{Produits tensoriels complétés}\label{D}
Dans cette section, on définit le produit tensoriel complété de manière élémentaire, bien qu'il existe une approche plus sophistiquée faisant appel aux mathématiques condensées.

Soient $V, W$ deux espaces de Banach sur $L$. On note $(e_i)_{i\in\mathbb{N}}$ une base de Banach de $V$ et $(f_j)_{j\in\mathbb{N}}$ une base de Banach de $W$. On définit alors $V_i^+:=\oplus_{i'\leq i}\mathscr{O}_L\cdot e_{i'}$ et $W_j^+:=\oplus_{j'\leq j}\mathscr{O}_L\cdot f_{j'}$.

\begin{definition} On définit
	$$V\hat{\otimes}W:=(\varprojlim_n\varinjlim_i\varinjlim_j(V_i^+/\pi_L^n\otimes_{\mathscr{O}_L/\pi_L^n}W_j^+/\pi_L^n))[\tfrac{1}{\pi_L}],$$
	$$V^*\hat{\otimes}W:=(\varprojlim_n\varprojlim_i\varinjlim_j\mathrm{Hom}_{\mathscr{O}_L/\pi_L^n}(V_i^+/\pi_L^n, W_j^+/\pi_L^n))[\tfrac{1}{\pi_L}],$$
	et
	$$V^*\hat{\otimes}W^*:=(\varprojlim_n\varprojlim_i\varprojlim_j(V_i^{+, *}/\pi_L^n\otimes_{\mathscr{O}_L/\pi_L^n}W_j^{+, *}/\pi_L^n))[\tfrac{1}{\pi_L}].$$
\end{definition}
\begin{remark} On peut aussi définir une autre forme du produit tensoriel complété par
	$$V^*\hat{\otimes}'W:=(\varprojlim_n\varinjlim_j\varprojlim_i\mathrm{Hom}_{\mathscr{O}_L/\pi_L^n}(V_i^+/\pi_L^n, W_j^+/\pi_L^n))[\tfrac{1}{\pi_L}],$$
	alors $V^*\hat{\otimes}'W$ et $V\hat{\otimes}W^*$ sont duaux l'un de l'autre.
\end{remark}
Si l'un des deux espaces est un espace de Fréchet ou un espace LB, on a les définitions suivantes.
\begin{definition} Soient $V$ un espace de Banach et $W=\varprojlim_iW_i$ un espace de Fréchet, limite projective d'espaces de Banach $W_i$. On définit
	$$V\hat{\otimes}W:=\varprojlim_iV\hat{\otimes}W_i,$$
	et
	$$V^*\hat{\otimes}W:=\varprojlim_iV^*\hat{\otimes}W_i.$$
\end{definition}
\begin{definition} Soient $V$ un espace de Banach et $W=\varinjlim_iW_i$ un espace LB, limite inductive d'espaces de Banach $W_i$. On définit
	$$V\hat{\otimes}W:=\varinjlim_iV\hat{\otimes}W_i,$$
	et
	$$V^*\hat{\otimes}W:=\varinjlim_iV^*\hat{\otimes}W_i.$$
\end{definition}
Enfin, on a besoin de la définition suivante, qui nous permet de définir l'espace $\mathscr{O}_{\mathbf{P}^1}(U)^*\hat{\otimes}_{\mathscr{O}_{\mathbf{P}^1}(U)}\underline{\Omega}(U)^\diamond$ dans la section \ref{section4.1.1}.
\begin{definition}
	Soient $U, V$ deux espaces de Banach et $W=\varinjlim_iW_i$ un espace LB, limite inductive d'espaces de Banach $W_i$. On définit
	$$V^*\hat{\otimes}\mathrm{Hom}_L(W, U):=(V^*\hat{\otimes} W^*)\hat{\otimes} U.$$
\end{definition}
\end{appendices}
\newpage
\bibliographystyle{abbrv}
\bibliography{Revetements}
YANG PEI, IMJ-PRG, SORBONNE UNIVERSITÉ, 4 PLACE JUSSIEU, 75005 PARIS, FRANCE
\\
E-mail: yang.pei@imj-prg.fr

	\end{document}